\documentclass{amsart}

\usepackage[pdftex]{graphicx,color}
\graphicspath{{figure/}}

\usepackage[margin=1.1in]{geometry}
\usepackage{makecell}
\newcolumntype{?}{!{\vrule width 2\arrayrulewidth}}

\newcommand{\ba}{{\bf a}}

\newcommand{\bd}{{\bf d}}

\newcommand{\bef}{{\bf f}}
\newcommand{\bg}{{\bf g}}

\newcommand{\bn}{{\bf n}}

\newcommand{\bu}{{\bf u}}
\newcommand{\bv}{{\bf v}}
\newcommand{\bw}{{\bf w}}
\newcommand{\bx}{{\bf x}}
\newcommand{\by}{{\bf y}}

\newcommand{\bA}{{\bf A}}
\newcommand{\bB}{{\bf B}}
\newcommand{\bC}{{\bf C}}

\newcommand{\bF}{{\bf F}}
\newcommand{\bG}{{\bf G}}

\newcommand{\bI}{{\bf I}}

\newcommand{\bR}{{\bf R}}
\newcommand{\bS}{{\bf S}}
\newcommand{\bT}{{\bf T}}
\newcommand{\bU}{{\bf U}}

\newcommand{\bW}{{\bf W}}

\newcommand{\bnu}{{\boldsymbol \nu}}

\newcommand\strain[1]{\ensuremath \varepsilon(#1)}
\renewcommand\div[1]{\ensuremath \mathrm{div}(#1)}
\newcommand\tr[1]{\ensuremath \mathrm{Tr}(#1)}

\newtheorem{thm}{Theorem}[section]

\newtheorem{lem}[thm]{Lemma}
\newtheorem{prop}[thm]{Proposition}

\newtheorem{rem}[thm]{Remark}

\newcommand{\Bk}{\color{black}}
\newcommand{\Rd}{\color{red}}

\definecolor{aggiemaroon}{rgb}{0.8,0,0}

\newcommand\bLambda{ {\boldsymbol \Lambda} }
\newcommand\bTheta{ {\boldsymbol \Theta} }
\newcommand\bel{ {\boldsymbol \ell} }
\newcommand\dd{\,\mathrm{d}}

\def\Xint#1{\mathchoice
{\XXint\displaystyle\textstyle{#1}}%
{\XXint\textstyle\scriptstyle{#1}}%
{\XXint\scriptstyle\scriptscriptstyle{#1}}%
{\XXint\scriptscriptstyle\scriptscriptstyle{#1}}%
\!\int}
\def\XXint#1#2#3{{\setbox0=\hbox{$#1{#2#3}{\int}$ }
\vcenter{\hbox{$#2#3$ }}\kern-.59\wd0}}

\numberwithin{equation}{section}

\title[Finite Element Approximation of a Strain-Limiting Elastic Model]{Finite Element Approximation of a\\Strain-Limiting Elastic Model}

\author{Andrea Bonito}\address{Andrea Bonito, Texas A\&M University, Department of Mathematics, 3368 TAMU, College Station, TX 77843-3368, USA. email: \texttt{bonito@math.tamu.edu};
}
\thanks{Andrea Bonito is supported in part by NSF grant DMS-1817691.}
\author{Vivette Girault}\address{Vivette Girault, Universite Pierre et Marie Curie, Paris VI, Laboratoire Jacques-Louis Lions,
4, place Jussieu, F-75230 Paris Cedex 05, France. email: \texttt{girault@ann.jussieu.fr};}
\author{Endre S\"uli}\address{Endre S\"uli, Mathematical Institute, University of Oxford, Andrew Wiles Building, Woodstock Road, Oxford OX2 6GG, United Kingdom. email: \texttt{suli@maths.ox.ac.uk}.}

\begin{document}

\begin{abstract}
We construct a finite element approximation of a strain-limiting elastic model on a bounded open domain in $\mathbb{R}^d$, $d \in \{2,3\}$. The sequence of finite element approximations is shown to exhibit strong convergence to the unique weak solution of the model. Assuming that the material parameters featuring in the model are Lipschitz-continuous, and assuming that the weak solution has additional regularity, the sequence of finite element approximations is shown to converge with a rate. An iterative algorithm is constructed for the solution of the system of nonlinear algebraic equations that arises from the finite element approximation. An appealing feature
of the iterative algorithm is that it decouples the monotone and linear elastic parts of the nonlinearity in the model. In particular, our choice of piecewise constant approximation for the stress tensor (and continuous piecewise linear approximation for the displacement) allows us to compute the monotone part of the nonlinearity by solving an algebraic system with $d(d+1)/2$ unknowns independently on each element in the subdivision of the computational domain. The theoretical results are illustrated by numerical experiments.
\end{abstract}

\maketitle


\smallskip


~\vspace{3cm}

\section{Introduction and statement of the problem}

\label{sec:Intro}

Until recently, the term \textit{elasticity} referred to \textit{Cauchy elasticity}, and within such a theory, strain-limiting models are not possible. Motivated by the work of Rajagopal in~\cite{Raj2003}, see also~\cite{Raj2007}, the objective of this
paper is to design, analyze and implement numerical approximations of models that fall outside the realm of classical Cauchy elasticity. These models are implicit and nonlinear, and are referred to as strain-limiting, because they permit the linearized strain to remain bounded even when the stress is very large: a property that cannot be guaranteed within the framework of standard elastic or nonlinear elastic models.

On a bounded  domain $\Omega \subset \mathbb R^d$, $d \in \{2, 3\}$, and for a given external force $\bef : \Omega \rightarrow \mathbb R^d$, we consider the nonlinear elastic model
\begin{equation}\label{e:div}
-\div{\bT} = \bef \qquad \textrm{in } \Omega,
\end{equation}
where the symmetric stress tensor $\bT$ is related to the strain tensor $\strain{\bu}:=\frac{1}{2}(\nabla \bu + (\nabla \bu)^{\rm T})$,
for a given displacement vector $\bu$, via a nonlinear constitutive relation of the form
\begin{equation}\label{e:strain_stress}
\strain{\bu} = \lambda( \tr{\bT} ) \tr{\bT} \bI + \mu(|\bT^{\bd}|)\bT^{\bd} \qquad \textrm{in } \Omega.
\end{equation}
Here $\lambda \in \mathcal{C}^0(\mathbb R)$ and $\mu \in \mathcal{C}^0([0,+\infty))$ are given functions and $\bT^{\bd}$ denotes the deviatoric part of the tensor $\bT$, defined by
$$
\bT^{\bd} := \bT - \frac1d \tr{\bT}\bI.
$$
Additional assumptions on $\lambda$ and $\mu$ are required (see \eqref{e:cond_lambda1}--\eqref{e:cond_mu2} below), which guarantee that, in particular, the right-hand side of \eqref{e:strain_stress} is a monotone operator applied to $\bT$.
This strain-limiting model is used to describe, for example, the behavior of brittle materials in the vicinity of fracture tips, or in the neighborhood of concentrated loads, where there is concentration of stress even though the magnitude of
the strain tensor is limited. The model itself is derived and analyzed in the work of {Bul\'\i\v cek} {\it et al.}~\cite{BMRS14}; some of the ideas introduced in \cite{BMRS14} will also be used in the numerical analysis developed in the sequel.
Of course, there are several strain-limiting models: the reader will find other models in~\cite{BMRS14} and the references quoted therein.
This being the first effort though to construct and rigorously analyze a numerical algorithm for a strain-limiting elastic model,
we shall confine ourselves to the model \eqref{e:div}, \eqref{e:strain_stress}.

The analysis of the model \eqref{e:div}, \eqref{e:strain_stress} is far from trivial because the operator involved, although monotone, lacks coercivity. The authors of~\cite{BMRS14} show the existence of a weak solution to the problem by
first regularizing \eqref{e:strain_stress} with the addition of an appropriate coercive term,
$$
\frac{\tr{\bT} \bI}{n|\tr{\bT}|^{1-\frac{1}{n}}} + \frac{\bT^{\mathbf d}}{n|\bT^{\mathbf d}|^{1-\frac{1}{n}}},
$$
see \eqref{e:weak_reg}, eventually providing a control of $\bT$ in $L_{1+\frac 1 n}(\Omega)^{d \times d}$.
It is then shown in~\cite{BMRS14} that, as $n\to \infty$,
the limit of the sequence of solutions to the regularized system satisfies the original problem.
This nonlinear regularization is necessary in order to be able to cope with possibly rough data $\bef$.
However, for smoother data, the simpler \emph{linear} regularization
$
\frac 1 n \bT
$
has been used in~\cite{BMRS14} to recover additional regularity of the solution; see  \eqref{e:weak_reg,m=1}.


The same framework is used here in the discrete case. More precisely, the regularized problems \eqref{e:weak_reg} and \eqref{e:weak_reg,m=1} are discretized by means of a simple finite element scheme: for instance, on simplices,
by discontinuous piecewise ${\mathcal P}_0$ elements for the each component of the stress tensor $\bT$, and globally continuous, piecewise ${\mathcal P}_1$ elements
for each component of the displacement vector $\bu$; see \eqref{e:discrete_reg} and \eqref{e:weak_regh,m=1}.
It is worth noting here that for quadrilateral subdivisions of the domain $\Omega$, the corresponding $({\mathcal Q}_0,{\mathcal Q}_1)$ stress/displacement pair of finite element spaces is (inf-sup) unstable,
and discontinuous polynomials of degree 1 in each direction should be selected for the stress approximations instead of ${\mathcal Q}_0$ elements so as
to restore (inf-sup) stability; see Sections~\ref{subsec:otherfit} and~\ref{subsec:notheory}.
Convergence to the exact solution is established by first passing to the limit as the mesh-size tends to zero, for a fixed value of the regularization parameter $n$, and then we let $n$ tend to infinity.
For rough data,  the delicate part in the approximation of \eqref{e:weak_reg}  is the derivation of a suitable rate of convergence for the approximation error.
The difficulty stems from the lack of a meaningful error bound in a standard Lebesgue norm. Our analysis therefore relies on modular forms and associated Orlicz norms (see Theorem~\ref{thm:err.inequ} and the subsequent discussion).
For smoother data, the $\frac{1}{n}\mathbf{T}$ regularization mentioned above can be used, and the numerical analysis of \eqref{e:weak_regh,m=1}
is then somewhat simpler because estimates for the stress, for the regularized problem at least, are naturally obtained in $L_2(\Omega)^{d\times d}$ (see Theorem~\ref{thm:err.inequm=1}) instead of $L_1(\Omega)^{d \times d}$ (or $L_{1+\frac{1}{n}}(\Omega)^{d \times d}$).

The proposed finite element discretizations \eqref{e:discrete_reg} and \eqref{e:weak_regh,m=1} yield nonlinear systems with constraints.
Since the nonlinear operator is the sum of a monotone and a coercive operator, we take advantage of  the algorithm developed by Lions and
Mercier in~\cite{LM79} to decouple these two parts:  the unconstrained monotone system is solved first, followed by solving a constrained coercive system.
As the stress tensor is potentially discontinuous, its simplest possible discretization is, as was suggested above, by means of a piecewise constant
approximation on simplices; thus the associated nonlinearity can be resolved element-by-element.
We establish convergence of this splitting algorithm when applied to \eqref{e:weak_regh,m=1} (see Theorem~\ref{t:convLMm=1}).
When applied to \eqref{e:discrete_reg}, the rigorous proof of convergence of the splitting algorithm is an open problem,
although our numerical experiments at least appear to indicate that the splitting algorithm may well be convergent in this case as well.


\subsection{Setting of the problem}

\label{subsec:pbsetting}

We consider the system \eqref{e:div}, \eqref{e:strain_stress} and describe the assumptions required on $\lambda$ and $\mu$.
In addition to $\lambda \in \mathcal{C}^0(\mathbb R)$ and $\mu \in \mathcal{C}^0([0,+\infty))$, we assume that $s \in \mathbb R \mapsto \lambda(s)s \in \mathcal{C}^1(\mathbb R)$.
Complementing these regularity hypotheses, we assume that $\lambda$ and $\mu$ satisfy, for some positive constants $C_1, C_2, \kappa$ and $\alpha$, the following inequalities:
\begin{alignat}{2}
\frac{C_1 s^2}{\kappa + |s|} &\leq \lambda(s) s^2 \leq C_2 |s| \qquad &&\forall\, s \in \mathbb R;
\tag{A1} \label{e:cond_lambda1} \\
\frac{C_1 s^2}{\kappa + s} &\leq \mu(s) s^2 \leq C_2 s \qquad &&\forall\, s \in \mathbb R_{\geq 0}; \tag{A2} \label{e:cond_mu1} \\
0 &\leq \frac{\dd}{\dd s} (\lambda(s)s) \qquad &&\forall\, s \in \mathbb R;
\tag{A3} \label{e:cond_lambda2} \\
\frac{C_1}{(\kappa+s)^{\alpha+1}} &\leq \frac{\dd}{\dd s} (\mu(s)s) \qquad &&\forall\, s \in \mathbb R_{>0}.
\tag{A4} \label{e:cond_mu2}
\end{alignat}
We note that, using the continuity of $\lambda$, the first inequality in assumption \eqref{e:cond_lambda1} implies that
$\lambda(s) > 0$ when $s \not = 0$ and $\lambda(s)\geq 0$ for $s\in \mathbb R$.
In addition, using now the second inequality in \eqref{e:cond_lambda1} we have that
\begin{equation}
\label{e:lambdasbdd}
| \lambda(s) s | = \lambda(s) |s| \leq C_2 \qquad \forall\, s \in \mathbb R.
\end{equation}
The same argument applied to the function $\mu$ gives  $\mu(s) > 0$ when $s > 0$, and
 \begin{equation}
\label{e:musbdd}
\mu(s) \ge  0,\quad \mu(s)\,s \leq C_2 \qquad \forall\, s \in \mathbb R_{\geq 0}.
\end{equation}

In particular, these assumptions guarantee that the system will only exhibit finite strain (see Theorem~\ref{t:existence} below).
At this point, we also recall a result from \cite{MNRR96} (see also Lemma 4.1 in \cite{BMRS14}), which will play a crucial role in the subsequent analysis.

Under the assumptions \eqref{e:cond_lambda1}--\eqref{e:cond_mu2} stated above, there exists a positive constant $C$ such that the following inequalities hold for all $\bR_1, \bR_2 \in \mathbb R^{d\times d}_{\textrm{sym}}$ (the set of all $d \times d$ symmetric matrices with real-valued entries):
\begin{align}
&\left( \mu( |\bR_1|) \bR_1 - \mu( |\bR_2|) \bR_2 \right): (\bR_1 - \bR_2)
\geq C \frac{|\bR_1-\bR_2|^2}{(\kappa+|\bR_1|+|\bR_2|)^{1+\alpha}}; \label{e:mon_mu1}\\
&\left( \mu( |\bR_1|) \bR_1 - \mu( |\bR_2|) \bR_2 \right): (\bR_1 - \bR_2)
\geq C \left| (\kappa+|\bR_1|)^{\frac{1-\alpha}{2}}-  (\kappa+|\bR_2|)^{\frac{1-\alpha}{2}}\right|^2; \label{e:mon_mu2}\\
&\left( \lambda( \tr{\bR_1}) \tr{\bR_1} - \lambda( \tr{\bR_2}) \tr{\bR_2} \right) (\tr{\bR_1} - \tr{\bR_2}) \geq 0. \label{e:mon_lambda1}
\end{align}
If, in addition,
\begin{equation}
0 < \frac{\dd}{\dd s} (\lambda(s)s) \qquad \forall\, s \in \mathbb R,
\tag{A3'} \label{e:cond_lambda2bis}
\end{equation}
then, for all  $\bR_1, \bR_2 \in \mathbb R^{d\times d}_{\textrm{sym}}$ such that $\tr{\bR_1} \not = \tr{\bR_2}$, we have
\begin{equation}
\left( \lambda( \tr{\bR_1}) \tr{\bR_1} - \lambda( \tr{\bR_2}) \tr{\bR_2} \right) (\tr{\bR_1} - \tr{\bR_2}) > 0.\label{e:mon_lambda2}
\end{equation}

The system \eqref{e:div}, \eqref{e:strain_stress} is supplemented with the boundary conditions
$$
\bu = \bg \quad \textrm{on } \partial_D \Omega \qquad \textrm{and} \qquad \bT \bnu = \bel \quad \textrm{on } \partial_N \Omega,
$$
where the boundary of $\Omega$ is decomposed into two parts, $\partial_D \Omega$ and $\partial_N \Omega$, with $\partial_D \Omega \cap \partial_N \Omega = \emptyset$ and $\overline{\partial_D \Omega \cup \partial_N \Omega} = \partial \Omega$, $\bnu$ is the outward-pointing unit normal to $\partial \Omega$, $\bg: \partial \Omega \rightarrow \mathbb R^d$ is a given displacement on $\partial_D \Omega$, and  $\bel : \partial \Omega \rightarrow \mathbb R^d$ is a given traction force on $\partial_N \Omega$.


\subsection{Notation}

\label{subsec:notation}

We shall suppose for the rest of this section that $\Omega$ is a bounded simply connected John domain; see, for instance,~\cite{AcostaDuranMuschietti2006} or~\cite{Jiang}.
Henceforth, $L_p(\Omega)$ and $W^{k,p}(\Omega)$ will denote the standard Lebesgue and Sobolev
spaces, and the corresponding spaces of $d$-component vector-valued functions and symmetric $d \times d$-component
tensor-valued functions will be denoted, respectively, by $L_p(\Omega)^d$, $L_p(\Omega)^{d\times d}_{\rm sym}$
and $W^{k,p}(\Omega)^d$,  $W^{k,p}(\Omega)^{d\times d}_{\rm sym}$.
In order to characterize displacements that
vanish on the boundary, $\partial\Omega$, of $\Omega$, we consider for $p \in [1,\infty)$ the Sobolev space $W_0^{1,p}(\Omega)$,
defined as the closure of the linear space $\mathcal{C}^\infty_0(\Omega)$, consisting of infinitely many
times continuously differentiable functions with compact support in $\Omega$, in the norm of the space
$W^{1,p}(\Omega)$:
$$ W_0^{1,p}(\Omega) = \overline{\mathcal{C}^\infty_0(\Omega)}^{\|\cdot\|_{1,p}}.$$

We recall the Poincar\'e and Korn inequalities, which, for each $p \in (1,\infty)$, assert the existence of positive constants
$\mathcal S_p$ and $\mathcal K_p$, such that, respectively (cf. Theorem 1.5 in \cite{Jiang}),
\begin{alignat}{2}
\label{e:poincare}
\|v\|_{L_p(\Omega)} &\le \mathcal S_p\,\|\nabla v\|_{L_p(\Omega)}\qquad &&\forall\, v \in W^{1,p}_0(\Omega),\\
\label{e:Korn1}
\|\nabla \bv\|_{L_p(\Omega)} &\le \mathcal K_p\,\|\strain{\bv}\|_{L_p(\Omega)}\qquad &&\forall\, \bv \in W^{1,p}_0(\Omega)^d.
\end{alignat}
By combining inequalities \eqref{e:Korn1} and \eqref{e:poincare} we obtain the inequality
\begin{equation}
\label{e:korn}
\| \bv \|_{L_p(\Omega)} \leq C_{K,p} \| \strain{\bv} \|_{L_p(\Omega)} \qquad \forall\, \bv\in W^{1,p}_0(\Omega)^d,
\end{equation}
with $C_{K,p} = \mathcal S_p\,\mathcal K_p>0$.

For any two symmetric $d\times d$ tensors $\bS=(S_{ij})$ and $\bT=(T_{ij})$, we shall use a colon to denote their contraction product,
$$ \bS:\bT = \sum_{i=1}^d \sum_{j=1}^d S_{ij} T_{ij},
$$
so that the Frobenius norm of $\bS$ reads
$$|\bS|^2 = \bS:\bS = {\rm Tr}(\bS^2).
$$
It is then easy to show that
\begin{equation}
\label{e:normT0}
| \bS |^2 = |\bS^{\bd} |^2 +  \frac 1 d |\tr{\bS}|^2 \leq |\bS^{\bd} |^2 +  |\tr{\bS}|^2 \qquad \forall\, \bS \in \mathbb R^{d\times d}_{\textrm{sym}},
\end{equation}
which implies that
\begin{equation}
\label{e:normT1}
| \bS | \leq |\bS^{\bd} | +  |\tr{\bS}|\qquad \forall\, \bS \in \mathbb R^{d\times d}_{\textrm{sym}}.
\end{equation}
Conversely,
\begin{equation}
\label{e:normT2}
| \tr{\bS} | + | \bS^{\bd} |  \leq \sqrt{2d}\, | \bS | \qquad \forall\, \bS \in \mathbb R^{d\times d}_{\textrm{sym}},
\end{equation}
since, by elementary inequalities and by noting the equality stated in \eqref{e:normT0},
$$
| \tr{\bS} | + | \bS^{\bd} | \leq \sqrt{d} \left( \frac{1}{\sqrt{d}}| \tr{\bS}| +  | \bS^{\bd} | \right) \leq
\sqrt{2d} \left( \frac{1}{d}| \tr{\bS}|^2 +  | \bS^{\bd} |^2 \right)^{\frac{1}{2}} = \sqrt{2d}\, | \bS |.
$$
Moreover, for any nonnegative real numbers $a$ and $b$, and for any $p\geq 1$ and $\theta \in (0,1]$, we have
\[ a^p + b^p \geq 2^{1-p}  (a+b)^p \geq 2^{1-p}  (a^2 + b^2)^{\frac{p}{2}} \geq 2^{1-p}  (\theta a^2 +  b^2)^{\frac{p}{2}}.\]
Thus, by taking $a=|\tr{\bS}|$, $b=|\bS^{\bd}|$ and $\theta =1$, we have by \eqref{e:normT0} that,
for any $p \geq 1$,
\begin{equation}\label{e:pbound}
2^{1-p} | \bS |^{p} \leq | \tr{\bS} |^p + | \bS^\bd |^p.
\end{equation}
%

%

The remainder of this article is organized as follows. The problem is set into variational form in Section \ref{sec:Weakform} and the associated existence and uniqueness results are recalled. Sections \ref{sec:weakform} and \ref{sec:aprioriexct} are devoted to the analysis of the sequence of regularized problems \eqref{e:weak_reg}
 that will be discretized by finite elements in Section \ref{sec:FEM}; this includes a priori estimates, convergence, and identification of the limit. The simpler analysis of \eqref{e:weak_reg,m=1} is sketched in Section \ref{s:smoother}. In Section \ref{s:decoupled}, we present an iterative algorithm that dissociates the computation of the nonlinearity from the elastic constraint, and we prove its convergence when applied to \eqref{e:weak_regh,m=1}. In Section \ref{sec:experiment}, we report numerical experiments aimed at assessing the performance of the iterative algorithm and the discretization scheme.


\section{Weak Formulation}

\label{sec:Weakform}

We begin by recalling Theorem 4.3 from \cite{BMRS14}, which  guarantees the existence and uniqueness
of a solution to the problem \eqref{e:div},  \eqref{e:strain_stress} in the case when
$\partial_D \Omega = \partial \Omega$ and $\bg = \mathbf 0$.

When the Neumann part of the boundary $\partial_N \Omega$ is nonempty, the structure of the solution is
potentially much more complicated. It was shown in \cite{ref:BBMS} that, in general, the solution in that
case belongs to the space of Radon measures, but if the problem is equipped with a so-called asymptotic radial
structure, then the solution can in fact be understood as a standard weak solution, with one proviso: the
attainment of the boundary value is penalized by a measure supported on $\partial_N \Omega$. For simplicity,
in this initial effort to construct a provably convergent numerical algorithm for the problem under consideration,
we shall therefore suppose henceforth that $\partial_D \Omega = \partial \Omega$ (i.e., $\partial_N \Omega = \emptyset$)
and that the Dirichlet boundary datum is $\bg = \mathbf 0$ on $\partial\Omega$.

\begin{thm}[Theorem 4.3 in \cite{BMRS14}]\label{t:existence}
Assume that $\partial_N \Omega = \emptyset$ and that $\lambda$, $\mu$ satisfy \eqref{e:cond_lambda1}--\eqref{e:cond_mu2} with $0 \leq \alpha < 1/d$;
then, the following statements hold:
\begin{enumerate}
\item Assume that $\bef = -\div{\bF}$ for  $\bF \in W^{\beta,1}(\Omega)^{d\times d}_{\rm sym}$ with $\beta \in (\alpha d,1)$. Then, there exists a pair
$(\bT, \bu)$, such that
\begin{align}
\bT &\in L_1(\Omega)^{d\times d}_{\rm sym},\nonumber\\
\bu &\in W^{1,p}_0(\Omega)^d \qquad \forall\, p \in [1,\infty),\nonumber\\
\strain{\bu} &\in L_{\infty}(\Omega)^{d\times d}_{\rm sym}\nonumber,
\end{align}
is a weak solution in the sense that it satisfies
\begin{equation}\label{e:div_weak}
\int_\Omega \bT : \strain{\bw}\dd\bx = \int_\Omega \bF : \strain{\bw} \dd\bx \qquad \forall\, \bw \in \mathcal D(\Omega)^d,
\end{equation}
where $\mathcal D(\Omega)^d:= C^\infty_0(\Omega)^d$,
and the nonlinear relationship between the strain $\strain{\bu}$ and the stress $\bT$ stated in \eqref{e:strain_stress} holds almost everywhere in $\Omega$;
\item Moreover, if $\Omega$ has a  continuous boundary, then the equality \eqref{e:div_weak} holds for all $\bw \in W^{1,1}_0(\Omega)^{d}$ such that $\strain{\bw} \in L_\infty(\Omega)^{d\times d}_{\rm sym}$;
\item In addition, $\bu$ is unique and if $\lambda$ satisfies the assumption \eqref{e:cond_lambda2bis},
then $\bT$ is also unique;
\item Furthermore, if $\bF$ belongs to $W^{2,2}(\Omega)^{d\times d}_{\rm sym}$, then $\bT \in W^{1,q}_\textrm{loc}(\Omega)^{d\times d}_{\rm sym}$ with
$$
q \left\lbrace
\begin{array}{ll}
:= 2- \frac{2(d-2)(1+\alpha)}{(d-2)(1+\alpha)+d(1-\alpha)}, &\quad \textrm{for }d \geq 3, \\
\in [1,2), &\quad \textrm{arbitrary for }d=2.
\end{array}\right.
$$
\end{enumerate}
\end{thm}

\begin{rem}
\label{rem:rangeq}
{\rm We note in connection with part (d) of the above theorem that when $d=3$, then
$$q = 2- \frac{1+\alpha}{2-\alpha}
$$
is a monotonic decreasing function of $\alpha$. Thus, as $0 \le \alpha < \frac{1}{3}$, we have $\frac{6}{5} < q \le \frac{3}{2}$.}
\end{rem}


\section{Analysis of a Regularized Problem}

\label{sec:weakform}

The proof of existence of weak solutions to the problem is based on constructing a sequence of solutions to a
regularized problem, where the original stress-strain relationship \eqref{e:strain_stress} is modified to become
\begin{equation}\label{e:reg_T}
\strain{\bu} = \lambda( \tr{\bT} ) \tr{\bT} \bI + \mu(|\bT^{\bd}|)\bT^{\bd} +
\frac{\tr{\bT} \bI}{n|\tr{\bT}|^{1-\frac{1}{n}}} + \frac{\bT^{\mathbf d}}{n|\bT^{\mathbf d}|^{1-\frac{1}{n}}};
\end{equation}
here $n \in \mathbb{N}$ (where $\mathbb{N}$ denotes the set of all positive integers) is a regularization parameter, which we shall ultimately send to the limit $n \rightarrow \infty$.

Following this idea, we study in this work the finite element approximation of this
regularized problem, stated in the following variational form:
find $(\bT_n,\bu_n) \in \mathbb M_n \times \mathbb X_n$ satisfying
\begin{alignat}{2}\label{e:weak_reg}
\begin{aligned}
\quad a_n(\bT_n,\bS)+c(\bT_n;\bT_n,\bS)-b(\bS,\bu_n) &=0 \qquad &&\forall\, \bS \in\mathbb M_n,\\
\quad b(\bT_n,\bv) &= \int_\Omega \bF : \strain{\bv} \dd\bx \qquad &&\forall\, \bv \in \mathbb X_n,
\end{aligned}
\end{alignat}
where
\begin{align*}
a_n(\bT,\bS)&:= \frac{1}{n}\int_\Omega \left(\frac{\tr{\bT} \bI}{|\tr{\bT}|^{1-\frac{1}{n}}} + \frac{\bT^{\mathbf d}}{|\bT^{\mathbf d}|^{1-\frac{1}{n}}}\right): \bS \dd \bx, \\
c(\bT;\bR,\bS)&:= \int_\Omega  \left( \lambda(\tr{\bT})\tr{\bR}\bI + \mu(|\bT^{\bd}|)\bR^{\bd}\right) : \bS \dd\bx,
\\
b(\bS,\bv)&:= \int_\Omega \bS : \strain{\bv} \dd\bx,
\end{align*}
and
$$
\mathbb M_n:= L_{1+\frac{1}{n}}(\Omega)^{d\times d}_\textrm{sym}, \qquad \mathbb X_n:= W^{1,n+1}_0(\Omega)^d,\qquad n \in \mathbb{N}.
$$

Motivated by the form of the expression appearing on the right-hand side of the relationship
\eqref{e:reg_T}, we define the mapping
$\mathcal A_{n}: L_{1+\frac{1}{n}}(\Omega)^{d\times d}_\textrm{sym} \rightarrow L_{n+1}(\Omega)^{d\times d}_\textrm{sym}$ by
\begin{equation}\label{e:def_A}
\mathcal A_{n}(\bS) := \lambda(\tr{\bS}) \tr{\bS} \bI + \mu( | \bS^{\bd}|) \bS^{\bd}
+ \frac{\tr{\bS} \bI}{n|\tr{\bS}|^{1-\frac{1}{n}}} + \frac{\bS^{\bd}}{n|\bS^\bd|^{1-\frac{1}{n}}}.
\end{equation}
It follows from the inequalities \eqref{e:lambdasbdd} and \eqref{e:musbdd} that $\mathcal A_{n}$ does indeed take its values in $L_{n+1}(\Omega)^{d\times d}_\textrm{sym}$, since the first two terms belong to $L_{\infty}(\Omega)^{d\times d}_\textrm{sym}$ for all $\bS \in  L_{1+\frac{1}{n}}(\Omega)^{d\times d}_\textrm{sym}$, while the third and fourth term
belong to $L_{n+1}(\Omega)^{d \times d}_\textrm{sym}$ for all $\bS \in  L_{1+\frac{1}{n}}(\Omega)^{d\times d}_\textrm{sym}$, $n \in \mathbb{N}$.
Moreover, the mapping $\mathcal{A}_{n} :   L_{1+\frac{1}{n}}(\Omega)^{d\times d}_\textrm{sym} \rightarrow  L_{n+1}(\Omega)^{d\times d}_\textrm{sym}$ is bounded, continuous and coercive for all $n \in \mathbb{N}$,
as is asserted in the following lemma.

\begin{lem}[Boundedness, continuity and coercivity of $\mathcal A_{n}$]\label{l:A_cont}
Let $\lambda \in \mathcal \mathcal{C}^0(\mathbb R)$ and $\mu \in \mathcal \mathcal{C}^0([0,+\infty))$, and suppose that hypotheses (A1) and (A2) are valid.
Then, the following assertions hold:
\begin{itemize}
\item[(i)]
For any $n \in \mathbb{N}$, the mapping $\mathcal{A}_n:L_{1+ \frac{1}{n}}(\Omega)^{d\times d}_{\rm sym} \rightarrow L_{n+1}(\Omega)^{d\times d}_{\rm sym}$ is bounded; i.e.,  every
bounded set in $L_{1+ \frac{1}{n}}(\Omega)^{d\times d}_{\rm sym}$ is mapped by $\mathcal{A}_n$ into a bounded set in
$L_{n+1}(\Omega)^{d\times d}_{\rm sym}$;
\item[(ii)]
For any $n \in \mathbb{N}$, the mapping $\mathcal A_{n}: L_{1+ \frac{1}{n}}(\Omega)^{d\times d}_{\rm sym} \rightarrow L_{n+1}(\Omega)^{d\times d}_{\rm sym}$ is continuous, i.e.,
for any sequence $(\bS_k)_{k>0}\subset L_{1+ \frac{1}{n}}(\Omega)^{d\times d}_{\rm sym}$, which strongly converges in the norm of $L_{1+ \frac{1}{n}}(\Omega)^{d\times d}$ to some $\bS \in L_{1+ \frac{1}{n}}(\Omega)^{d\times d}_{\rm sym}$, we have that
$$
\mathcal A_{n}(\bS_k ) \rightarrow \mathcal A_{n} (\bS) \qquad \textrm{strongly in } L_{n+1}(\Omega)^{d\times d}_{\rm sym};
$$
\item[(iii)] For any $n \in \mathbb{N}$, the mapping $\mathcal{A}_n$ is coercive, i.e.,
\[ \frac{\int_\Omega \mathcal{A}_n(\bS):\bS \dd \bx}{\|\bS\|_{L_{1+\frac{1}{n}}(\Omega)}} \rightarrow \infty \qquad \mbox{as
$\|\bS\|_{L_{1+\frac{1}{n}}(\Omega)} \rightarrow \infty$.}\]
\end{itemize}
\end{lem}


\begin{proof}
(i) It suffices to prove that any bounded ball in $L_{1+ \frac{1}{n}}(\Omega)^{d\times d}_\textrm{sym}$, centred at the origin, is mapped by $\mathcal{A}_n$ into a bounded set in  $L_{n+1}(\Omega)^{d\times d}_\textrm{sym}$. Consider, to this end, the bounded ball
$$B_R := \{\bS \in L_{1+ \frac{1}{n}}(\Omega)^{d\times d}_\textrm{sym}: \|\bS\|_{L_{1+ \frac{1}{n}}(\Omega)} \leq R\} \quad \mbox{with $R>0$}.$$
For every $\bS \in B_R$, we have that
\begin{align*} \|\mathcal{A}_n(\bS)\|_{L_{n+1}(\Omega)} &\leq
C_2 d^{\frac{1}{2}}|\Omega|^{\frac{1}{n+1}} + C_2 |\Omega|^{\frac{1}{n+1}} + \frac{1}{n}d^{\frac{1}{2}}  \|\tr{\bS}\|_{L_{1+\frac{1}{n}}(\Omega)}^{\frac{1}{n}}
+ \frac{1}{n} \|\bS^{\bd}\|_{L_{1+\frac{1}{n}}(\Omega)}^{\frac{1}{n}}\\
& \leq C_2 d^{\frac{1}{2}}|\Omega|^{\frac{1}{n+1}} + C_2 |\Omega|^{\frac{1}{n+1}} + \frac{1}{n}d^{\frac{1}{2}(1+\frac{1}{n})}  \|\bS\|_{L_{1+\frac{1}{n}}(\Omega)}^{\frac{1}{n}}
+ \frac{1}{n} \|\bS\|_{L_{1+\frac{1}{n}}(\Omega)}^{\frac{1}{n}},
\end{align*}
where in the transition to the second inequality we have made use of the facts that, by the identity \eqref{e:normT0} we have  $|\tr{\bS}| \leq d^{\frac{1}{2}}|\bS|$ and
$|\bS^{\bd}|^2 = |\bS|^2 - \frac{1}{d}(\tr{\bS})^2$, whereby $|\bS^{\bd}|\leq |\bS|$. Hence,
\[ \|\mathcal{A}_n(\bS)\|_{L_{n+1}(\Omega)} \leq
C_2 d^{\frac{1}{2}}|\Omega|^{\frac{1}{n+1}} + C_2 |\Omega|^{\frac{1}{n+1}} + \frac{1}{n}d^{\frac{1}{2}(1+\frac{1}{n})}  R^{\frac{1}{n}}
+ \frac{1}{n} R^{\frac{1}{n}}=:R_*, \]
which implies that $\mathcal{A}_n(B_R)$ is contained in a bounded ball in $L_{n+1}(\Omega)^{d\times d}_\textrm{sym}$, centred at the origin, of radius $R_*$. Thus, $\mathcal{A}_n\,:\, L_{\Rd 1 \Bk+\frac{1}{n}}(\Omega)^{d\times d}_\textrm{sym} \rightarrow L_{n+1}(\Omega)^{d\times d}_\textrm{sym}$ is a bounded mapping.

(ii) Suppose that $\bS_k \rightarrow \bS$ strongly in $L_{1+ \frac{1}{n}}(\Omega)^{d\times d}_\textrm{sym}$.
We begin by showing that
$$
\lambda(\tr{\bS_k }) \tr{\bS_k} \bI \rightarrow \lambda(\tr{\bS}) \tr{\bS} \bI \quad \textrm{strongly in }
L_{n+1}(\Omega)^{d \times d}_\textrm{sym}.$$
By defining
 $\varphi_k := \lambda(\tr{\bS_k})\tr{\bS_k}$ and using \eqref{e:lambdasbdd},  we get
$$
| \varphi_k | \leq C_2 \qquad a.e. ~ \textrm{in }\Omega.
$$
Now, the strong convergence of $\{\bS_k\}_{k>0}$ in $L_{1+ \frac{1}{n}}(\Omega)^{d\times d}_\textrm{sym}$ implies that there exists a subsequence (not indicated) such that
$\bS_k \rightarrow \bS$ a.e. on $\Omega$. Thanks to the assumed continuity of $\lambda$, it then follows that
$$\varphi_k \rightarrow \varphi:= \lambda(\tr{\bS})\tr{\bS}$$ a.e. in $\Omega$ and $|\varphi| \leq C_2$ a.e. in $\Omega$.
By Lebesgue's dominated convergence theorem we therefore have that $\varphi_k \rightarrow \varphi$ strongly in $L_1(\Omega)$. When combined with the boundedness of $\varphi_k$, the strong convergence  $\varphi_k \rightarrow \varphi$ in $L_1(\Omega)$,
implies that $\varphi_k \rightarrow \varphi$ strongly in $L_p(\Omega)$ for all $p \in [1,\infty)$.
Therefore, taking $p=n+1$, the first term of $\mathcal{A}_{n}(\bS_k)$ strongly converges in $L_{n+1}(\Omega)^{d\times d}_\textrm{sym}$ to the first term in $\mathcal{A}_{n}(\bS)$. The same is true of the second term.

To handle the third term, we note that since, for any $a \in (0,1]$,
\[ \bigg|\frac{x}{|x|} |x|^{a} - \frac{y}{|y|} |y|^{a}\bigg| \leq 2^{1-a}|x - y|^{a}\qquad \forall\,x,y \in \mathbb{R}\setminus\{0\},\]
it follows with $a  = \frac{1}{n}$, $n \in \mathbb{N}$, that
\begin{align}\label{e:holder}
 \bigg|\frac{\tr{\bS_k}}{|\tr{\bS_k}|^{1-\frac{1}{n}}}  - \frac{\tr{\bS}}{|\tr{\bS}|^{1-\frac{1}{n}}} \bigg|
\leq 2^{1-\frac{1}{n}}|\tr{\bS_k} - \tr{\bS}|^{\frac{1}{n}} \leq 2^{1-\frac{1}{n}}d^{\frac{1}{2n}}\,|\bS_k - \bS|^{\frac{1}{n}},
\end{align}
whereby the assumed strong convergence $\bS_k \rightarrow \bS$ in $L_{1+\frac{1}{n}}(\Omega)^{d \times d}_{\textrm{sym}}$ implies that
\[ \frac{\tr{\bS_k}}{|\tr{\bS_k}|^{1-\frac{1}{n}}} \bI \rightarrow \frac{\tr{\bS}}{|\tr{\bS}|^{1-\frac{1}{n}}} \bI\]
in $L_{n+1}(\Omega)^{d\times d}_\textrm{sym}$, $n \in \mathbb{N}$.
By an identical argument the fourth term strongly converges in $L_{n+1}(\Omega)^{d\times d}_\textrm{sym}$, $n \in \mathbb{N}$.

(iii) Note that, by assumptions \eqref{e:cond_lambda1} and \eqref{e:cond_mu1},
\begin{align*}
\int_\Omega \mathcal{A}_n(\bS) : \bS \dd \bx &\geq \frac{1}{n}
\int_\Omega \Big( |\tr{\bS}|^{1+\frac{1}{n}} + |\bS^{\bd}|^{1+\frac{1}{n}} \Big) \dd \bx.
\end{align*}
By taking $p=1+\frac{1}{n}$ with $n \in \mathbb{N}$ in \eqref{e:pbound} and using \eqref{e:pbound}, we then have that
$$
\int_\Omega \mathcal{A}_n(\bS) : \bS \dd \bx
\ge \frac{2^{-\frac{1}{n}}}{n}
\int_\Omega |\bS|^{\frac{n+1}{n}} \dd \bx = \frac{2^{-\frac{1}{n}}}{n}\|\bS\|^{1+\frac{1}{n}}_{L_{1+\frac{1}{n}}(\Omega)}.
$$
As the exponent $1+\frac{1}{n}$ appearing on the right-hand side of the last equality is strictly greater than $1$ for all
$n \in \mathbb{N}$, the
coercivity of the mapping $\mathcal{A}_n: L_{1+ \frac{1}{n}}(\Omega)^{d\times d}_\textrm{sym} \rightarrow L_{n+1}(\Omega)^{d\times d}_\textrm{sym}$ directly follows.
\end{proof}

\begin{rem}
\label{rem:Lipsch}
{\rm One can simplify the proof of the continuity of $\mathcal{A}_n$ asserted in Lemma \ref{l:A_cont} (ii) by assuming that $s\mapsto \lambda(s)s$ and $s\mapsto \mu(s)s$ are globally H\"older-continuous functions over their respective domains of definition. The latter assumption will be required in Theorem \ref{thm:err.inequ} to deduce rates of convergence for the finite element approximation of the regularized problem; prior to that, we do not assume the global H\"older-continuity of $s\mapsto \lambda(s)s$ and $s\mapsto \mu(s)s$.
}
\end{rem}

\begin{lem}[Monotonicity of $\mathcal{A}_n$]
\label{l:A_mono}
Assume that $\lambda \in \mathcal \mathcal{C}^0(\mathbb R)$ and $\mu \in \mathcal \mathcal{C}^0([0,+\infty))$, and that
hypotheses \eqref{e:cond_lambda1}--\eqref{e:cond_mu2} are satisfied. Then, for any $n \in \mathbb{N}$, the mapping $\mathcal{A}_n:L_{1+ \frac{1}{n}}(\Omega)^{d\times d}_{\rm sym} \rightarrow L_{n+1}(\Omega)^{d\times d}_{\rm sym}$ is
monotone, i.e.,
\begin{equation}
\int_\Omega (\mathcal{A}_n(\bS_1) - \mathcal{A}_n(\bS_2)):(\bS_1 - \bS_2) \dd \bx \geq 0
\label{e:A_nmonotone}
\end{equation}
for any pair of functions $\bS_1, \bS_2 \in L_{1+ \frac{1}{n}}(\Omega)^{d\times d}_{\rm sym}$. Furthermore, monotonicity is strict, in the sense that equality holds if, and only if, $\bS_1 = \bS_2$ a.e. on $\Omega$.
\end{lem}

\begin{proof}
To prove the monotonicity of $\mathcal{A}_n$, note first that for any pair of matrices $\bS, \bR \in \mathbb{R}^{d \times d}_{\textrm{sym}}$ one has
\begin{equation}\label{e:mono1}
(\bS|\bS|^{\frac{1}{n}-1} - \bR|\bR|^{\frac{1}{n}-1}):(\bS - \bR) \geq \frac{1}{n}|\bS - \bR|^2 \int_0^1 |\bR + \theta (\bS-\bR)|^{\frac{1}{n}-1}
\dd \theta \geq 0,
\end{equation}
and since $n \ge 1$, the expression on the right-hand side
is equal to 0 if, and only if, $\bS = \bR$. Similarly, for any $s, r \in \mathbb R$,
\begin{equation}\label{e:mono2}
(s|s|^{\frac{1}{n}-1} - r|r|^{\frac{1}{n}-1})\,(s - r)  =  \frac{1}{n} |s - r|^2 \int_0^1 |r + \theta (s-r)|^{\frac{1}{n}-1}\dd \theta  \geq 0,
\end{equation}
and the expression on the right-hand side
is equal to 0 if, and only if, $s = r$. Hence, and by noting
the inequalities \eqref{e:mon_mu1} and \eqref{e:mon_lambda1}, we have that
\begin{align}\label{e:A-monotone}
\begin{aligned}
&\int_\Omega (\mathcal{A}_n(\bS_1) - \mathcal{A}_n(\bS_2)):(\bS_1 - \bS_2) \dd \bx \\
&\qquad\geq
\frac{1}{n^2} |\tr{\bS_1-\bS_2}|^2 \int_0^1 \big|\tr{\bS_2 + \theta (\bS_1 - \bS_2)}\big|^{\frac{1}{n}-1}\dd \theta \\
& \qquad \quad +
\frac{1}{n^2} |\bS_1^\bd - \bS_2^\bd|^2\int_0^1 \big|\bS_2^\bd + \theta (\bS_1^\bd - \bS_2^\bd)|^{\frac{1}{n}-1}\dd \theta.
\end{aligned}
\end{align}
The expression on the right-hand side
of this inequality is nonnegative and it is equal to 0 if, and only if, $\tr{\bS_1}=\tr{\bS_2}$
a.e. on $\Omega$ and $\bS_1^\bd = \bS_2^\bd$ a.e. on $\Omega$, that is, when $\bS_1 = \bS_2$ a.e. on $\Omega$.
\end{proof}


\section{A-Priori Estimates for the Regularized Problem}

\label{sec:aprioriexct}

Our aim in this section is to derive a-priori estimates for the regularized problem \eqref{e:weak_reg}. Clearly, problem \eqref{e:weak_reg} can be interpreted as a constrained system with a (strictly) monotone nonlinearity. The constraint is the second equation in problem \eqref{e:weak_reg}; it is linear and nonhomogeneous, and can be, as is usual in mixed variational
problems, transformed into a homogenous constraint
via an inf-sup property, which we state in the next lemma.
\begin{lem}[Inf-sup property]\label{l:infsup}
The following inequality holds for all $n \in \mathbb{N}$:
\begin{equation}
\label{e:LBB}
\inf_{\bv \in \mathbb X_n} \,\sup_{\bS \in \mathbb M_n} \, \frac{b(\bS,\bv)}{\| \bS \|_{L_{1+\frac{1}{n}}(\Omega)} \| \strain{\bv} \|_{L_{n+1}(\Omega)}} \geq 1.
\end{equation}
\end{lem}
\begin{proof}
Given $\bv \in \mathbb X_n = W^{1,n+1}_0(\Omega)^d$, it suffices to note that
$\bR = \strain{\bv} |\strain{\bv}|^{n-1} \in L_{1+\frac{1}{n}}(\Omega)^{d \times d}_{\rm sym}$ and that we have
$$
b(\bR,\bv) = \| \strain{\bv} \|_{L_{n+1}(\Omega)}^{n+1} = \| \strain{\bv} \|_{L_{n+1}(\Omega)} \| \strain{\bv} \|_{L_{n+1}(\Omega)}^n = \| \strain{\bv} \|_{L_{n+1}(\Omega)} \|\bR\|_{L_{1+\frac{1}{n}}(\Omega)}.
$$
Whence,
$$
\sup_{\bS \in \mathbb M_n} \frac{b(\bS,\bv)}{\| \bS \|_{L_{1+\frac{1}{n}}(\Omega)}} \geq  \| \strain{\bv} \|_{L_{n+1}(\Omega)},
$$
and the stated inf-sup property follows.
\end{proof}

We shall assume henceforth that, as in Theorem \ref{t:existence}, $\mathbf{f}=-\rm{div}(\bF)$, with $\bF \in W^{\beta,1}(\Omega)^{d \times d}_{\rm sym}$
and $\beta \in (\alpha d, 1)$ (recall that, by hypothesis, $0<\alpha<\frac{1}{d}$); hence, by Sobolev embedding $\bF \in L_{1+ \frac{\beta}{d-\beta}}(\Omega)^{d \times d}_{\rm sym}$
whereby also $\bF \in L_{1+\frac{1}{n}}(\Omega)^{d \times d}_{\rm sym} = \mathbb M_n$ for all $n \geq \frac{d}{\beta} -1$
(consequently, $\mathbf{f} \in  W^{-1,1+ \frac{1}{n}}(\Omega)^{d} = (W^{1,n+1}_0(\Omega)^{d})' = (\mathbb X_n)'$ for all $n \geq \frac{d}{\beta} -1$), and we define
$$ \bT^{\mathbf{f}}_n := \bF.$$
Clearly, the subscript $_n$ in the expression on the left-hand side of this equality is redundant, as $\bT^{\mathbf{f}}_n$ is equal to $\bF$
for all $n \geq \frac{d}{\beta} -1$. We shall however continue to carry this redundant subscript in order to emphasize
the fact that the problem, as a whole, is dependent on $n$. Should it be desired that $\bF \in L_{1+\frac{1}{n}}(\Omega)^{d \times d}_{\rm sym}$
for all $n \in \mathbb{N}$, one can, instead, adopt the slightly stronger assumption that
$\bF \in W^{\beta,1}(\Omega)^{d \times d}_{\rm sym}\cap L_2(\Omega)^{d \times d}_{\rm sym}$.

%
The use of the function $\bT^{\mathbf{f}}_n$ will allow us to lift the constraint imposed by the second equation in
problem \eqref{e:weak_reg} by converting it into a homogeneous equation;
we can then replace the first equation in \eqref{e:weak_reg} by
one that is considered on a linear subspace $\mathbb V_n$ of $\mathbb M_n$, defined below, which we choose to be
the kernel of the mapping $\mathrm{div} : \mathbb M_n \rightarrow (\mathbb X_n)'$.

Trivially,
\begin{equation}\label{e:constraint}
\int_\Omega \bT^{\mathbf{f}}_n : \strain{\bv} \dd \bx = \int_{\Omega} \bF : \strain{\bv} \dd \bx\qquad \forall\, \bv \in \mathbb X_n.
\end{equation}
We define
\begin{equation}
\label{eq:Vn}
\mathbb V_n := \left\lbrace \bS \in \mathbb M_n ~ : ~ \ b(\bS,\bv) = 0 \quad \forall\, \bv \in \mathbb X_n \right\rbrace
= \left\lbrace \bS \in \mathbb M_n ~:~ \div{\bS} = {\mathbf{0}}  \in  (\mathbb X_n)' \right \rbrace = \mathrm{Ker(div)}.
\end{equation}
As $\mathbb X_n$ is a reflexive Banach space, transposition yields that the
transpose $(-\mathrm{div})' : (\mathbb X_n)''=\mathbb X_n  \rightarrow  (\mathbb M_n)'$ of the linear operator $- \mathrm{div} : \mathbb M_n \rightarrow (\mathbb X_n)'$ is $(-\mathrm{div})' = \strain{\cdot}$.
The annihilator ${\mathbb V}_n^\perp$ of $\mathbb V_n$ is, by definition,
$${\mathbb V}_n^\perp :=\{\ell \in (\mathbb M_n)' : \ell(\bS) = 0 \quad \forall\, \bS \in \mathbb V_n\}.$$

By the Riesz representation theorem the dual space
$(\mathbb M_n)'$ of $L_{1+\frac{1}{n}}(\Omega)^{d \times d}_{\rm sym}$ is isometrically isomorphic to
$L_{n+1}(\Omega)^{d \times d}_{\rm sym}$.  Furthermore, since $-\mathrm{div} : \mathbb M_n \rightarrow
(\mathbb X_n)'$ is a bounded linear operator, it is also a closed linear operator. Hence, by Banach's closed range theorem,
\begin{align*}
{\mathbb V}_n^\perp &:= \left\lbrace \bR \in L_{n+1}(\Omega)^{d \times d}_{\rm sym} ~ : ~ \ \int_\Omega \bR : \bS \dd \bx = 0
\quad \forall\, \bS \in \mathbb V_n \right\rbrace \\
& ~= [\mathrm{Ker}(\mathrm{div})]^\perp = [\mathrm{Ker}(-\mathrm{div})]^\perp = \mathrm{Range}((-\mathrm{div})')
 = \mathrm{Range}(\strain{\cdot}).
\end{align*}
Furthermore, once again by the closed range theorem,
\[ \mathrm{Range}(\mathrm{div}) = \mathrm{Range}(\mathrm{-div}) =  [\mathrm{Ker}((-\mathrm{div})')]^\perp = [\mathrm{Ker}(\strain{\cdot})]^\perp = [\{0\}]^\perp = (\mathbb X_n)',\]
where the penultimate equality follows from the inequality \eqref{e:korn}.

Thanks to the definition of $\bT^{\mathbf{f}}_n$,
\begin{equation}\label{e:bound_Tf}
\| \bT^{\mathbf{f}}_n \|_{L_{1+\frac{1}{n}}(\Omega)} = \| \bF \|_{L_{1+\frac{1}{n}}(\Omega)}.
\end{equation}
Using $\bT^{\mathbf{f}}_n$, we can eliminate the constraint \eqref{e:weak_reg}$_2$ by setting $$\bT^0_n := \bT_n - \bT^{\mathbf{f}}_n  \in \mathbb V_n$$
 and consider the problem: find $\bT^0_n   \in \mathbb V_n$ such that
\begin{equation}\label{e:T0}
a_n(\bT^0_n + \bT^{\mathbf{f}}_n,\bS)+c_n(\bT^0_n+\bT^{\mathbf{f}}_n;\bT^0_n+\bT^{\mathbf{f}}_n,\bS) =0\qquad \forall\, \bS \in \mathbb V_n.
\end{equation}
From here, by using Lemma \ref{l:A_cont} and Lemma \ref{l:A_mono}, we easily deduce that the mapping
$$\bS \in \mathbb V_n \mapsto \mathcal A_{n}(\bS + \bT^{\mathbf{f}}_n) \in L_{n+1}(\Omega)^{d \times d}_{\rm sym} = (\mathbb M_n)' \subset (\mathbb{V}_n)'
$$
is bounded, continuous (and therefore hemi-continuous),  coercive and monotone; in addition, $\mathbb{V}_n$ is a separable reflexive Banach space, as it is a closed
linear subspace of the separable and reflexive Banach space $\mathbb M_n = L_{1+ \frac{1}{n}}(\Omega)^{d \times d}_{\rm sym}$. Therefore, by the Browder--Minty theorem (cf., for instance,  \cite{ref:Sch,ref:Lions}) problem \eqref{e:T0}, and hence also problem \eqref{e:weak_reg}, has a solution $\bT_n  = \bT^0_n + \bT^{\mathbf{f}}_n \in \mathbb M_n$, and since by Lemma \ref{l:A_mono} the operator $\mathcal{A}_n$ is strictly monotone, the solution is unique.

With $\bT_n \in \mathbb M_n$ thus uniquely fixed, we seek $\bu_n \in \mathbb{X}_n$ such that
\[ b(\bS,\bu_n) = a_n(\bT_n,\bS)+c(\bT_n;\bT_n,\bS) \qquad \forall\, \bS \in\mathbb M_n.\]
Consider the linear functional $\ell_n \in (\mathbb M_n)'$ defined by
\[ \ell_n(\bS):= a_n(\bT_n,\bS)+c(\bT_n;\bT_n,\bS),\qquad \bS \in \mathbb M_n.\]
Hence, thanks to equation \eqref{e:T0}, we have that $\ell_n(\bS) = 0$ for all $\bS \in \mathbb{V}_n$; consequently, $\ell_n \in {\mathbb V}_n^\perp$.
Thus, we are seeking $\bu_n \in \mathbb{X}_n$ such that
\begin{align} \label{e:b-ell-eq}
 b(\bS,\bu_n) = \ell_n(\bS)\qquad \forall\, \bS \in\mathbb M_n.
\end{align}
As $\ell_n \in {\mathbb V}_n^\perp=[\mathrm{Ker}(\mathrm{div})]^\perp 
= \mathrm{Range}(\strain{\cdot})$, there exists a $\bu_n \in \mathbb{X}_n$ such that
 $\strain{\bu_n} = \ell_n$; that is $\bu_n \in \mathbb X_n$ solves problem \eqref{e:b-ell-eq}.
The inf-sup property \eqref{e:LBB}, together with the inequality \eqref{e:korn},
then implies that such a $\bu_n \in \mathbb{X}_n$ is unique.
Thus we have shown the existence of a unique solution pair $(\bT_n, \bu_n)
\in \mathbb{M}_n \times \mathbb{X}_n$ to the regularized problem \eqref{e:weak_reg}.

Next we shall prove the following a-priori bounds on $\strain{\bu_n}$ and $\bT_n$.

\begin{lem}[A-priori estimates]\label{l:apriori}
Suppose that $\bF \in L_{1+\frac{1}{n}}(\Omega)^{d \times d}_{\rm sym}$, and that $\lambda$ and $\mu$ satisfy the properties \eqref{e:cond_lambda1} and \eqref{e:cond_mu1}.
We then have that
\begin{align*}
\| \strain{\bu_n} \|_{L_{n+1}(\Omega)}
&\leq
\frac{1}{n}\,d\sqrt{2} \left[16d^2\|\bF\|_{L_{1+\frac{1}{n}}(\Omega)}^{1+\frac{1}{n}}
+ 2 (n+1) C_2 \sqrt{2d}\, | \Omega|^{\frac{1}{n+1}} \|\bF\|_{L_{1+\frac{1}{n}}(\Omega)} + 4 (n+1) C_1 \kappa |\Omega|      \right]^{\frac{1}{n+1}}\\
& \qquad + C_2 \sqrt{2d}\, | \Omega|^{\frac{1}{n+1}}.
\end{align*}
Moreover,
\begin{align*}
\frac{1}{n+1}\|\bT_n\|_{L_{1+\frac{1}{n}}(\Omega)}^{1+\frac{1}{n}} + C_1 \|\bT_n\|_{L_1(\Omega)}
\leq \frac{16d^2}{n+1}\|\bF\|^{1+\frac{1}{n}}_{L_{1+\frac{1}{n}}(\Omega)}
+ 2C_2 \sqrt{2d}\, | \Omega|^{\frac{1}{n+1}}\|\bF\|_{L_{1+\frac{1}{n}}(\Omega)} + 4C_1 \kappa |\Omega|.
\end{align*}
\end{lem}

\begin{proof}
We start by testing problem \eqref{e:weak_reg} with  $\bv = \bu_n$ and $\bS = \bT_n$ to get
\begin{alignat}{2}
\begin{aligned}\label{e:weak_reg1}
\quad a_n(\bT_n,\bT_n)+c(\bT_n;\bT_n,\bT_n)-b(\bT_n,\bu_n) &=0,\\
\quad b(\bT_n,\bu_n) &= \int_\Omega \bF : \strain{\bu_n} \dd\bx,
\end{aligned}
\end{alignat}
whence, by substituting equation \eqref{e:weak_reg1}$_2$ into equation \eqref{e:weak_reg1}$_1$, we have
\begin{align*}
&\frac{1}{n}\int_\Omega \bigg(\frac{|\tr{\bT_n}|^2}{|\tr{\bT_n}|^{1-\frac{1}{n}}} + \frac{|\bT^\bd_n|^2}{|\bT^\bd_n|^{1-\frac{1}{n}}}\bigg)  \dd \bx +
\int_\Omega  \bigg( \lambda(\tr{\bT_n})|\tr{\bT_n}|^2 + \mu(|\bT^{\bd}_n|)|\bT^{\bd}_n|^2\bigg) \dd\bx
\\
&\qquad = \int_\Omega \bF : \strain{\bu_n} \dd\bx,
\end{align*}
where we have used that
$$
\bT^{\bd} : \bT = \bT^{\bd} : \bigg(\bT^{\bd} + \frac 1 d \tr{\bT}\bI\bigg) = | \bT^{\bd} |^2 + \frac 1 d \tr{\bT} \underbrace{\tr{\bT^{\bd}}}_{=0} = | \bT^{\bd} |^2.
$$
Hence, H\"older's inequality yields
\begin{align*}
&\frac{1}{n}\int_\Omega \bigg(|\tr{\bT_n}|^{1+\frac{1}{n}} + |\bT^{\bd}_n|^{1+\frac{1}{n}}\bigg)  \dd \bx +
\int_\Omega  \bigg( \lambda(\tr{\bT_n})|\tr{\bT_n}|^2 + \mu(|\bT^{\bd}_n|)|\bT^{\bd}_n|^2\bigg) \dd\bx
\\
&\qquad \leq \|\bF\|_{L_{1+\frac{1}{n}}(\Omega)} \|\strain{\bu_n}\|_{L_{n+1}(\Omega)}.
\end{align*}

For the $\lambda$ and $\mu$ terms on the left-hand side of this inequality we note that for $s  >  0$ one has
$$
\frac{s^2}{\kappa +s} =
s - \frac{\kappa}{1+\frac{\kappa}{s}}.
$$
This, together with the properties \eqref{e:cond_lambda1} and \eqref{e:cond_mu1}, leads to
\begin{align*}
&\frac{1}{n}\int_\Omega \bigg(|\tr{\bT_n}|^{1+\frac{1}{n}} + |\bT^{\bd}_n|^{1+\frac{1}{n}}\bigg)  \dd \bx + C_1 \int_\Omega \big(|\tr{\bT_n}| + |\bT^{\bd}_n| \big)\dd \bx\\
& \qquad \leq \| \bF \|_{L_{1+\frac{1}{n}}(\Omega)} \| \strain{\bu_n} \|_{L_{n+1}(\Omega)} + C_1 \kappa \int_\Omega \bigg( \frac{1}{1+\frac{\kappa}{|\tr{\bT_{\Rd n\Bk}}|}}+\frac{1}{1+\frac{\kappa}{|\bT^{\bd}_{\Rd n\Bk}|}}\bigg) \dd \bx\\
& \qquad \leq \| \bF \|_{L_{1+\frac{1}{n}}(\Omega)} \| \strain{\bu_n} \|_{L_{n+1}(\Omega)} + 2C_1 \kappa |\Omega|,
\end{align*}
since $\sup_{s > 0}   \frac{1}{1+\frac{\kappa}{s}} = 1$.

Moreover, it follows from the inequality \eqref{e:pbound} with $p=1+\frac 1 n$ that
\[ |\tr{\bT_n}|^{1+\frac{1}{n}} + |\bT^{\bd}_n|^{1+\frac{1}{n}} \geq 2^{-\frac{1}{n}}|\bT_n|^{1+\frac{1}{n}},\]
and therefore \eqref{e:normT1} yields
\begin{equation}\label{e:T_n_est}
\frac{2^{-\frac{1}{n}}}{n} \| \bT_n \|^{1+\frac{1}{n}}_{L_{1+\frac{1}{n}}(\Omega)} + C_1 \| \bT_n \|_{L_1(\Omega)} \leq
\| \bF \|_{L_{1+\frac{1}{n}}(\Omega)} \| \strain{\bu_n} \|_{L_{n+1}(\Omega)} + 2C_1 \kappa |\Omega|.
\end{equation}

We now derive a bound on $\| \strain{\bu_n} \|_{L_{n+1}(\Omega)}$, using the inf-sup property \eqref{e:LBB}.  We begin by noting that
$$
\| \strain{\bu_n} \|_{L_{n+1}(\Omega)}  \leq \sup_{\bS \in L_{1+\frac{1}{n}}(\Omega)^{d\times d}_\textrm{sym}} \frac{b(\bS,\bu_n)}{\| \bS \|_{L_{1+\frac{1}{n}}(\Omega)}}
=  \sup_{\bS \in L_{1+\frac{1}{n}}(\Omega)^{d \times d}_\textrm{sym}} \frac{a_n(\bT_n,\bS)+c(\bT_n;\bT_n,\bS)}{\| \bS \|_{L_{1+\frac{1}{n}}(\Omega)}}.
$$
We invoke H\"older's inequality, the equality $|\bI|=\sqrt{d}$, the elementary inequality $a + b \leq 2^{1-\frac{1}{n}}(a^n + b^n)^{\frac{1}{n}}$ where $a,b \geq 0$ and $n \in \mathbb{N}$ with $a=\sqrt{d}\,|\tr{\bT_n}|^{\frac{1}{n}}$, $b=|\bT_n^{\bd}|^{\frac{1}{n}}$, and note that $\frac{1}{\sqrt{d}}|\tr{\bT_n}| + |\bT_n^{\bd}| \leq \sqrt{2}\, |\bT_n|$, to deduce that
\[ a_n(\bT_n,\bS) \leq \frac{2\sqrt{d}}{n}\bigg(\frac{d}{2}\bigg)^{\frac{1}{2n}}\|\bT_{n}\|_{L_{1+\frac{1}{n}}(\Omega)}^{\frac{1}{n}}\|\bS\|_{L_{1+\frac{1}{n}}(\Omega)}.\]
%
Further, by noting the properties \eqref{e:cond_lambda1} and \eqref{e:cond_mu1} again together with the inequality \eqref{e:normT2}, we have
\begin{equation}
 c(\bT_n;\bT_n,\bS) \leq  C_2 \sqrt{2d}\, | \Omega|^{\frac{1}{n+1}} \|\bS\|_{L_{1+\frac{1}{n}}(\Omega)},
 \label{e:bound_c}
 \end{equation}
where we have bounded $\sqrt{d+1}$ by $\sqrt{2d}$ for the sake of simplifying the constants appearing in the subsequent calculations.
Hence,
\begin{equation}\label{e:strain_first}
\| \strain{\bu_n} \|_{L_{n+1}(\Omega)} \leq \frac{2\sqrt{d}}{n}\bigg(\frac{d}{2}\bigg)^{\frac{1}{2n}}\|\bT_n\|_{L_{1+\frac{1}{n}}(\Omega)}^{\frac{1}{n}} + C_2 \sqrt{2d}\, | \Omega|^{\frac{1}{n+1}}.
\end{equation}
By substituting the inequality \eqref{e:strain_first} into the inequality \eqref{e:T_n_est} we obtain
$$
\frac{2^{-\frac{1}{n}}}{n} \| \bT_n \|^{1+\frac{1}{n}}_{L_{1+\frac{1}{n}}(\Omega)} + C_1 \| \bT_n \|_{L_1(\Omega)} \leq
\| \bF \|_{L_{1+\frac{1}{n}}(\Omega)} \bigg(\frac{2\sqrt{d}}{n}\bigg(\frac{d}{2}\bigg)^{\frac{1}{2n}}\|\bT_n\|_{L_{1+\frac{1}{n}}(\Omega)}^{\frac{1}{n}} + C_2 \sqrt{2d}\, | \Omega|^{\frac{1}{n+1}}\bigg) + 2C_1 \kappa |\Omega|;
$$
thus, by applying Young's inequality,
\[ ab \leq \varepsilon \frac{a^p}{p} + \varepsilon^{-\frac{1}{p-1}} \frac{b^{q}}{q}
\qquad \mbox{for $a,b \geq 0$, $\varepsilon>0$, $p>1$ and $\frac{1}{p} + \frac{1}{q}=1$},\]
to the first term on the right-hand side with $p=n+1$, $\varepsilon = \frac{1}{n}\,2^{-\frac{1}{n}}$,
$$a=\|\bT_n\|_{L_{1+\frac{1}{n}}(\Omega)}^{\frac{1}{n}} \qquad\mbox{and}\qquad  b = \frac{2\sqrt{d}}{n}\bigg(\frac{d}{2}\bigg)^{\frac{1}{2n}}\| \bF \|_{L_{1+\frac{1}{n}}(\Omega)} $$
in order to absorb the factor $\|\bT_n\|_{L_{1+\frac{1}{n}}(\Omega)}^{\frac{1}{n}}$ into the left-hand side, we deduce that
\begin{align*}
&\frac{2^{-\frac{1}{n}}}{n+1}\|\bT_n\|_{L_{1+\frac{1}{n}}(\Omega)}^{1+\frac{1}{n}} + C_1 \|\bT_n\|_{L_1(\Omega)}
\\
&\qquad \qquad \qquad \leq \frac{2^{1+\frac{1}{2n} + \frac{1}{2 n^2}}(\sqrt{d})^{(1+\frac{1}{n})^2}}{n+1}\|\bF\|^{1+\frac{1}{n}}_{L_{1+\frac{1}{n}}(\Omega)}
+ C_2 \sqrt{2d}\, | \Omega|^{\frac{1}{n+1}}\|\bF\|_{L_{1+\frac{1}{n}}(\Omega)} + 2C_1 \kappa |\Omega|.
\end{align*}
Hence, after multiplying by $2^{\frac{1}{n}}$ and noting that $1 \leq 2^{\frac{1}{n}} \leq 2$ and $1+\frac{3}{2n} +\frac{1}{2n^2} \leq \big(1+\frac{1}{n}\big)^2$, we obtain


$$
\frac{1}{n+1}\|\bT_n\|_{L_{1+\frac{1}{n}}(\Omega)}^{1+\frac{1}{n}} + C_1 \|\bT_n\|_{L_1(\Omega)}
\leq \frac{(2\sqrt{d})^{(1+\frac{1}{n})^2}}{n+1}\|\bF\|^{1+\frac{1}{n}}_{L_{1+\frac{1}{n}}(\Omega)}
+ 2C_2 \sqrt{2d}\, | \Omega|^{\frac{1}{n+1}}\|\bF\|_{L_{1+\frac{1}{n}}(\Omega)} + 4C_1 \kappa |\Omega|.
$$
Bounding $1 + \frac{1}{n}$ by $2$ in the exponent of $2 \sqrt{d}$ in the first term on the right-hand side then yields the second inequality in the statement of the lemma.

Omitting the second term from the left-hand side of that inequality and multiplying by $n+1$ then yields
$$
\|\bT_n\|_{L_{1+\frac{1}{n}}(\Omega)}^{1+\frac{1}{n}} \leq 16d^2\|\bF\|_{L_{1+\frac{1}{n}}(\Omega)}^{1+\frac{1}{n}}
+ 2 (n+1) C_2 \sqrt{2d}\, | \Omega|^{\frac{1}{n+1}} \|\bF\|_{L_{1+\frac{1}{n}}(\Omega)} + 4 (n+1) C_1 \kappa |\Omega|.
$$
Therefore, by the inequality \eqref{e:strain_first}, we have that
{\small
\begin{align*}
\| \strain{\bu_n} \|_{L_{n+1}(\Omega)}
&\leq
\frac{2\sqrt{d}}{n}\bigg(\frac{d}{2}\bigg)^{\frac{1}{2n}}\left[16d^2\|\bF\|_{L_{1+\frac{1}{n}}(\Omega)}^{1+\frac{1}{n}}
+ 2 (n+1) C_2 \sqrt{2d}\, | \Omega|^{\frac{1}{n+1}} \|\bF\|_{L_{1+\frac{1}{n}}(\Omega)} + 4 (n+1) C_1 \kappa |\Omega|      \right]^{\frac{1}{n+1}}\\
& \qquad + C_2 \sqrt{2d}\, | \Omega|^{\frac{1}{n+1}}.
\end{align*}
}
~\vspace{-3mm}

\noindent
Bounding the exponent $\frac{1}{2n}$ by $\frac{1}{2}$ in the prefactor on the right-hand side yields the first bound
in the lemma.
\end{proof}

Lemma \ref{l:apriori} implies in particular that
\begin{align*}
\limsup_{n \rightarrow \infty} \| \strain{\bu_n} \|_{L_{n+1}(\Omega)}
\leq C_2 \sqrt{2d}
\end{align*}
and
\begin{align*}
\limsup_{n \rightarrow \infty} \|\bT_n\|_{L_1(\Omega)}
\leq \frac{2C_2}{C_1} \sqrt{2d}\, \|\bF\|_{L_1(\Omega)} + 4 \kappa |\Omega|.
\end{align*}
These bounds are consistent with the properties of the strain-limiting model under consideration, expressed by Theorem \ref{t:existence} (a), which asserts that the strain tensor is contained in $L_\infty(\Omega)^{d \times d}_{\rm sym}$, even though the stress tensor is, in general, an element of $L_1(\Omega)^{d \times d}_{\rm sym}$ only.

In connection with this, we recall from Section 4 of \cite{BMRS14} that the sequence of (unique) weak solution pairs $((\bT_n,\bu_n))_{n \in \mathbb{N}}$ to the
regularized problem \eqref{e:weak_reg} converges to a weak solution pair $(\bT,\bu)$ of the problem \eqref{e:div}, \eqref{e:strain_stress},
supplemented by a homogeneous Dirichlet boundary condition on $\partial\Omega$ (which is also unique if the condition \eqref{e:cond_lambda2bis} holds),
in the sense that, as $n \rightarrow \infty$,
\begin{alignat}{2}\label{e:strongcon1}
\bT_n &\rightarrow \bT&\,\, &\mbox{strongly in $L_q(\Omega_0)^{d \times d}_{\rm sym}$ for any $q \in \left[1,1 + \frac{1}{2} \frac{\beta-\alpha d}{d-\beta}\right)$, $\beta \in (\alpha d, 1)$, $0 \leq \alpha < \frac{1}{d}$, $\Omega_0 \subset\subset \Omega$;}
\end{alignat}
furthermore,
\begin{alignat}{2}\label{e:strongcon2}
\begin{aligned}
\bu_n &\rightharpoonup\bu  &&\qquad \mbox{weakly in $W^{1,2d}_0(\Omega)^d$},\\
\bu_n &\rightarrow \bu &&\qquad \mbox{strongly in $\mathcal{C}(\overline{\Omega})^d$},\\
\frac{(\bT_n)^{\bd}}{n|(\bT_n)^{\bd}|^{1-\frac{1}{n}}} &\rightarrow \mathbf{0} &&\qquad \mbox{strongly in $L_1(\Omega)^{d\times d}_{\rm sym}$},\\
\frac{\tr{\bT_n}}{n|\tr{\bT_n}|^{1-\frac{1}{n}}} &\rightarrow 0 &&\qquad \mbox{strongly in $L_1(\Omega)$}.
\end{aligned}
\end{alignat}
In particular,
\begin{equation}
\label{e:conv_strain_0}
~\hspace{1.35cm}\strain{\bu_n} \rightharpoonup \strain{\bu} \qquad \mbox{weakly in $L_{2d}(\Omega)^{d\times d}_{\rm sym}$}.
\end{equation}

We note though that the weak convergence result \eqref{e:conv_strain_0} can be strengthened to
\begin{equation}
\label{e:conv_strain}
\strain{\bu_n} \rightarrow \strain{\bu} \qquad \mbox{strongly in }L_p(\Omega_0)^{d\times d}_{\textrm{sym}} \qquad
\forall\,\Omega_0 \subset\subset \Omega, \quad \forall\, p \in [1,\infty)
\end{equation}
and consequently to
\begin{equation}
\label{e:conv_strain_strong}
\strain{\bu_n} \rightarrow \strain{\bu} \qquad \mbox{strongly in }L_{2d}(\Omega)^{d\times d}_\textrm{sym}.
\end{equation}

To show this, we fix any $\Omega_0 \subset\subset \Omega$ and note that by subtracting the constitutive relation \eqref{e:strain_stress} from its regularized
counterpart \eqref{e:reg_T} we have
$$
\strain{\bu_{n} - \bu} = \mathcal A(\bT_n) - \mathcal A(\bT) + \frac{\tr{\bT_n} \bI}{n|\tr{\bT_n}|^{1-\frac{1}{n}}} + \frac{\bT_n^{\mathbf d}}{n|\bT_n^{\mathbf d}|^{1-\frac{1}{n}}},
$$
where $\mathcal A: L_{1}(\Omega_0)^{d \times d}_{\rm sym} \rightarrow L_{\infty}(\Omega_0)^{d \times d}_{\rm sym}$ is given by
\begin{equation}\label{e:A_no_n}
\mathcal A(\bS) := \lambda(\tr{\bS}) \tr{\bS} \bI + \mu( | \bS^{\bd}|) \bS^{\bd}.
\end{equation}
A similar argument to the one in the proof of Lemma~\ref{l:A_cont} yields that the mapping $\mathcal A: L_{1}(\Omega_0)^{d \times d}_{\rm sym} \rightarrow L_{p}(\Omega_0)^{d \times d}_{\rm sym}$
is well-defined and continuous for all $p \in [1,\infty)$.
Whence, because $\bT_n$ converges strongly to $\bT$ in $L_1(\Omega_0)^{d\times d}_{\rm sym}$, it follows that so does $\mathcal A(\bT_n)$ to
$\mathcal A(\bT)$ in $L_p(\Omega_0)^{d\times d}_{\rm sym}$ for all $p \in [1,\infty)$. For the first regularization term, H\"older's inequality implies that
$$
\frac 1 n \| |\tr{\bT_n}|^{\frac 1 n} \|_{L_p(\Omega_0)} \leq \frac 1 n \| \tr{\bT_n} \|_{L_1(\Omega_0)}^{\frac 1 n}
| \Omega_0|^{\frac 1 p(1-\frac p n)} \rightarrow 0 \qquad \textrm{as }n \to \infty,
$$
and similarly for the second regularization term, containing $\bT_n^{\mathbf d}$.
The convergence result \eqref{e:conv_strain} then follows by collecting the above results.
To show \eqref{e:conv_strain_strong}, we consider a nested sequence of $\Omega_0$ that exhausts $\Omega$.
By \eqref{e:conv_strain} there exists a subsequence (still indexed by $n$) such that $\varepsilon(\bu_n) \rightarrow \varepsilon(\bu)$ almost everywhere on $\Omega$. Hence, in view of  \eqref{e:conv_strain_0}, \eqref{e:conv_strain_strong} follows by Vitali's theorem.

Motivated by these convergence results
our objective is to construct a sequence of finite element approximations $((\bT_{n,h},\bu_{n,h}))_{h \in (0,1]}$ to the solution pair $(\bT_n,\bu_n)$ of the regularized
problem, for a fixed value of $n$, and then pass to the limit $h \rightarrow 0_+$ with the discretization parameter
$h \in (0,1]$, followed by passage to the limit $n\rightarrow \infty$ with the regularization parameter $n \in \mathbb{N}$,
--- instead of approximating the solution pair $(\bT,\bu)$ directly by a finite element method. Our reasons for proceeding in this way
will be made clear at the start of Section \ref{subsec:convg}.


\section{Finite Element Approximation}
\label{sec:FEM}

For the sake of simplicity we shall suppose from now on that $\Omega$ is a polygon when $d=2$ or a Lipschitz polyhedron when $d=3$.

We consider a sequence of shape-regular  simplicial  subdivisions $(\mathcal T_h)_{h \in (0,1]}$ of $\overline{\Omega}$;
by this we mean that there exists a positive real number $\eta$, independent of the mesh-size $h$,  such that all closed simplices $K$ in the subdivision $\mathcal T_h$ satisfy the inequality
\begin{equation}
\label{e:reg}
\frac{h_K }{\varrho_K} \le \eta,
\end{equation}
where $h_K$ is the diameter of $K$ and $\varrho_K$ is the diameter of the largest ball inscribed in $K$; see for instance~\cite{Cia}.
The extension to quadrilateral and hexahedral meshes is discussed in Sections~\ref{subsec:otherfit} and~\ref{subsec:notheory}.

Let $\mathcal P^r_h$ be the space of piecewise (subordinate to $\mathcal T_h$) polynomials of degree at most $r$.
We consider the conforming finite element spaces
\begin{equation}
\label{e:P0P1}
\mathbb M_{n,h} :=  \left(\mathcal P^{0}_h\right)^{d\times d}_{\rm sym}  \subset \mathbb M_n,
\qquad \mathbb X_{n,h} :=   \left(\mathcal P^{1}_h\right)^{d}  \cap \mathbb X_n \subset \mathbb X_n,
\end{equation}
for the approximation of $\bT_n$ and $\bu_n$, respectively.
We note in passing that in the set-theoretical sense $\mathbb M_{n,h}$
and $\mathbb X_{n,h}$ are independent of $n$; we shall however continue to label them with the double subscript $_{n,h}$ instead
of just $_h$ in order to emphasize that they are being thought of as finite-dimensional normed linear
subspaces of $\mathbb{M}_n$ and $\mathbb{X}_n$, respectively, throughout the paper.

As the exact solution is not expected to be very smooth, we have restricted ourselves to considering a first-order finite element approximation.
There are, of course, other choices of first-order spaces than the one we shall be focusing on, but for the sake of
brevity we shall not dwell on those here in detail; for extensions and alternative choices of spaces, we refer the reader again to  Sections~\ref{subsec:otherfit} and~\ref{subsec:notheory}.


\subsection{Discrete Scheme}

\label{subsec:discretescheme}

The discrete counterpart of problem
\eqref{e:weak_reg}, based on $\mathbb X_{n,h}$ and $\mathbb M_{n,h}$, is then defined as follows:
find $(\bT_{n,h}, \bu_{n,h}) \in \mathbb M_{n,h} \times \mathbb X_{n,h} $ such that
\begin{alignat}{2}\begin{aligned}\label{e:discrete_reg}
&a_n(\bT_{n,h},\bS_{h})+c(\bT_{n,h};\bT_{n,h},\bS_h)-b(\bS_h,\bu_{n,h}) =0\qquad &&\forall\, \bS_h \in \mathbb M_{n,h}, \\
&b(\bT_{n,h},\bv_h) = \int_\Omega \bF : \strain{\bv_h} \dd \bx\qquad &&\forall\, \bv_h \in \mathbb X_{n,h}.
\end{aligned}
\end{alignat}

We start by proving the discrete version of the inf-sup property  \eqref{e:LBB}.

\begin{lem}[Discrete inf-sup property]\label{l:discrete_infsup}
For each $n \in \mathbb{N}$, we have
\begin{equation}
\label{e:LBBh}
\inf_{\bv_h \in \mathbb X_{n,h}}\, \sup_{\bS_h \in \mathbb M_{n,h}} \frac{b(\bS_h,\bv_h)}{\| \bS_h \|_{L_{1+\frac{1}{n}}(\Omega)} \| \strain{\bv_h} \|_{L_{n+1}(\Omega)}} \geq 1.
\end{equation}
\end{lem}
\begin{proof}
The argument is based on mimicking the proof of Fortin's Lemma. Indeed, the assertion directly follows from the continuous inf-sup property \eqref{e:LBB},
upon noting that for all $\bv_h \in \mathbb X_{n,h}\subset \mathbb X_n$, $\strain{\bv_h}$ belongs to $\mathbb M_{n,h}$. Thus,
for all $\bv_h \in \mathbb X_{n,h}$, all $\bS \in \mathbb M_n$, and all $K \in \mathcal T_h$, one has
$$
\int_{K} \strain{\bv_h} : \Pi_h \bS \dd \bx = \int_{K} \strain{\bv_h} : \bS \dd \bx,
$$
where $\Pi_h \bS \in \mathbb M_{n,h}$ is defined componentwise by
\begin{equation}
\label{e:PIh}
\Pi_h f|_K:= \frac{1}{|K|} \int_K f  \dd \bx \qquad \forall\, K \in \mathcal T_h,\; \forall\, f \in L_1(K),
\end{equation}
and the projector
$\Pi_h : \mathbb{M}_n \rightarrow \mathbb M_{n,h}$ is stable in (the norm of) $\mathbb M_n$ because
$$
\|\Pi_h f\|_{L_{1+\frac{1}{n}}(K)} \le \|f\|_{L_{1+\frac{1}{n}}(K)}\qquad \forall\, K \in \mathcal T_h,\; \forall\, f \in L_{1+\frac{1}{n}}(K),
$$
whereby
$$
\|\Pi_h \bS\|_{L_{1+\frac{1}{n}}(\Omega)} \le \|\bS\|_{L_{1+\frac{1}{n}}(\Omega)}\qquad \forall\, \bS \in L_{1+\frac{1}{n}}(\Omega)^{d\times d}_{\rm sym}.
$$
That completes the proof.
\end{proof}

The above proof suggests eliminating the constraint by defining
\begin{equation}
\label{e:Thf}
\bT_{n,h}^{\mathbf f} = \Pi_h \bT^{\mathbf{f}}_n = \Pi_h \bF;
\end{equation}
we recall that $\bT^{\mathbf{f}}_n \in {\mathbb M}_n$ satisfies the equality \eqref{e:constraint}.
Then, by setting
\begin{equation}
\mathbb V_{n,h} := \left\lbrace \bS_h \in \mathbb M_{n,h} ~ : ~ \ b(\bS_h,\bv_h) = 0\quad \forall\, \bv_h \in \mathbb X_{n,h} \right\rbrace
\label{e:V_n,h}
\end{equation}
and
$$
{\mathbb V}_{n,h}^\perp := \left\lbrace \bS_h \in \mathbb M_{n,h} ~ :  \int_\Omega \bS_h : \bR_h \dd \bx= 0 \quad
\forall\, \bR_h \in \mathbb V_{n,h} \right\rbrace,
$$
we deduce from the equalities \eqref{e:constraint}, \eqref{e:Thf},
and by noting that $\strain{\bv_h} \in \mathbb{M}_{n,h}$ for all $\bv_h \in \mathbb{X}_{n,h}$,
that $\bT_{n,h}^{\mathbf f} \in \mathbb  M_{n,h}$  satisfies
\begin{equation}\label{e:constraint_h}
\int_\Omega \strain{\bv_h}: \bT_{n,h}^{\mathbf f} \dd \bx= \int_{\Omega} \bF : \strain{\bv_h} \dd \bx\qquad \forall\, \bv_h \in \mathbb X_{n,h};
\end{equation}
furthermore, the equality \eqref{e:Thf} implies that
\begin{equation}\label{e:bound_Tf_h}
\| \bT_{n,h}^{\mathbf f} \|_{L_{1+\frac{1}{n}}(\Omega)} \leq \| \bF \|_{L_{1+\frac{1}{n}}(\Omega)}.
\end{equation}
We observe further that, as $h \rightarrow 0_+$,
\begin{equation}\label{e:Tf_strong}
\bT_{n,h}^{\mathbf f} \rightarrow \bT^{\mathbf{f}}_n \qquad \textrm{strongly in }\mathbb M_n.
\end{equation}

Given $\bT_{n,h}^{\mathbf f} \in \mathbb  M_{n,h}$ defined by the equality \eqref{e:Thf},  we shall seek $\bT_{n,h}^0 := \bT_{n,h} - \bT_{n,h}^{\mathbf f} \in  \mathbb V_{n,h}$ that solves
\begin{equation}\label{e:Th0}
a_n(\bT_{n,h}^0 + \bT_{n,h}^{\mathbf f},\bS_h)+c(\bT_{n,h}^0+\bT_{n,h}^{\mathbf f};\bT_{n,h}^0+\bT_{n,h}^{\mathbf f},\bS_h) =0\qquad \forall\, \bS_h \in \mathbb V_{n,h}.
\end{equation}

The existence of a unique such $\bT_{n,h}^0 \in \mathbb{V}_{n,h}$, and therefore of a unique $\bT_{n,h}=\bT_{n,h}^0+\bT_{n,h}^{\mathbf f} \in \mathbb{M}_{n,h}$ and a unique $\bu_{n,h} \in \mathbb{X}_{n,h}$
satisfying equation \eqref{e:discrete_reg}$_1$ for all $\bS_h \in \mathbb{M}_{n,h}$, can be shown by proceeding as in the case of the continuous problem discussed in Section \ref{sec:aprioriexct},
but with the continuous inf-sup property stated in Lemma \ref{l:infsup}, that was used there, now replaced by the discrete inf-sup property stated in Lemma \ref{l:discrete_infsup}. Indeed, let
$\mathcal A_{n,h}: \mathbb{M}_{n,h} \rightarrow (\mathbb{M}_{n,h})^\prime$ be defined by the projection of $\mathcal A_{n}$ onto $\mathbb{M}_{n,h}$,
$$\int_\Omega \mathcal A_{n,h}(\bS_h): \bR_h \dd \bx = a_n(\bS_{h},\bR_{h})+c(\bS_{h};\bS_{h},\bR_h)\qquad \forall\, \bR_h \in \mathbb M_{n,h}.
$$
In the present case where the tensors of $\mathbb M_{n,h}$ are piecewise constant functions, $\mathcal A_{n,h}(\bS_h)$ coincides with $\mathcal A_{n}(\bS_h)$, but this equality is not necessary. It is easy to check that $\mathcal A_{n,h}$ has the same boundedness, continuity, coercivity, and monotonicity properties (all uniform in $h$) as $\mathcal A_{n}$, as stated in Lemmas \ref{l:A_cont} and \ref{l:A_mono}. The same is true of the mapping
$$\bS_h \in \mathbb V_{n,h} \mapsto \mathcal A_{n,h}(\bS_h + \bT^{\mathbf{f}}_{n,h}).$$
Therefore, another application of the Browder--Minty theorem gives existence and uniqueness of $\bT_{n,h}^0$ solving \eqref{e:Th0} and we set $\bT_{n,h} = \bT_{n,h}^0 + \bT_{n,h}^{\mathbf f}$. The discrete inf-sup property  \eqref{e:LBBh} then guarantees the existence of a corresponding $\bu_{n,h} \in \mathbb X_{n,h}$ such that $(\bT_{n,h},\bu_{n,h})$ solves the system \eqref{e:discrete_reg}. This is summarized in the following lemma.

\begin{lem}[Existence and uniqueness of the discrete solution]
\label{lem:uniq.discrete}
Assume that $\lambda$ and $\mu$ satisfy the hypotheses \eqref{e:cond_lambda1}--\eqref{e:cond_mu2}.
Then, the system \eqref{e:discrete_reg} has exactly one solution $(\bT_{n,h},\bu_{n,h}) \in \mathbb M_{n,h} \times \mathbb X_{n,h}$.
\end{lem}

\subsection{Convergence of the sequence of discrete solutions}
\label{subsec:convg}

Without regularization (i.e., with $\frac{1}{n}$ formally set equal to zero in problem \eqref{e:discrete_reg}, resulting in
the absence of the form $a_n(\cdot,\cdot)$ from the left-hand side of \eqref{e:discrete_reg}$_1$),
the proof of convergence of the sequence of solutions generated by the resulting numerical method to $(\bT,\bu)$ is an open problem. The source of the technical difficulties is that,
as $n \rightarrow \infty$, the only uniform (w.r.t. $n \in \mathbb{N})$ bound on $\bT_{n,h}$, with $h \in (0,1]$ fixed,
that is directly available to us is in the $L_1(\Omega)^{d\times d}$ norm; a uniform bound in the $L_1(\Omega)^{d \times d}$
norm only guarantees biting weak convergence, via Chacon's biting lemma, for example, and this is insufficient to deduce even convergence of a subsequence
in the weak topology of $L_1(\Omega)^{d \times d}$.
The proof of existence of a solution to the continuous problem in
reference~\cite{BMRS14} succeeds because the $L_1(\Omega)^{d\times d}$ norm bound on $\bT_n$ in the sequence of solution pairs $(\bT_n,\bu_n)$
to the regularized problem is supplemented by fractional derivative estimates.
Unfortunately, the extension of those fractional derivative estimates to the finite element discretization considered here is problematic. For this reason, we freeze the parameter $n \in \mathbb{N}$ and
we now discuss convergence, without rates, of the sequence of solution pairs $(\bT_{n,h},\bu_{n,h})$ of the discrete scheme to the
solution $(\bT_n,\bu_n)$ of the regularized problem as $h \rightarrow 0_+$.
Having done so, we will invoke the converge results stated at the end of
Section \ref{sec:aprioriexct} to pass to the limit $n\rightarrow \infty$ to deduce that
$\lim_{n \rightarrow \infty} \lim_{h \rightarrow 0_+} (\bT_{n,h},\bu_{n,h}) = (\bT,\bu)$ in the strong topology of
$L_1(\Omega_0)^{d \times d}_{\rm sym} \times \mathcal{C}(\overline\Omega)^d$, for any $\Omega_0 \subset\subset \Omega$.

We begin by establishing the weak convergence
of the sequence $(\bT_{n,h})_{h \in (0,1]} \subset \mathbb{M}_n$, with $n \in \mathbb{N}$ fixed.

\begin{lem}[Weak convergence of $\bT_{n,h}$]\label{l:weak}
Assume that $\bF\in L_{1+\frac{1}{n}}(\Omega)^{d \times d}_{\rm sym}$ and that the functions $\lambda$ and $\mu$ satisfy the hypotheses \eqref{e:cond_lambda1}--\eqref{e:cond_mu2}. Let
$(\bT_n,\bu_n) \in \mathbb M_n \times \mathbb X_n$ be the unique solution of the regularized problem \eqref{e:weak_reg}.
Then, as $h \rightarrow 0_+$,
$$
\bT_{n,h} \rightharpoonup \bT_n \qquad \textrm{weakly in $\mathbb{M}_n =L_{1+\frac{1}{n}}(\Omega)^{d\times d}_{\rm sym}$}.
$$
\end{lem}

\begin{proof}
In this proof $C$ denotes a generic positive constant that is independent of $n$ and $h$.
We use again the lift $\bT_{n,h}^{\mathbf f} = \Pi_h \bT^{\mathbf{f}}_n$  of the data satisfying  the equality \eqref{e:constraint_h} and set $\bT_{n,h}^0 = \bT_{n,h} - \bT_{n,h}^{\mathbf f} \in \mathbb V_{n,h}$, which satisfies the equation \eqref{e:Th0}.
The a-priori estimates provided by Lemma~\ref{l:apriori} guarantee that
\begin{align}\label{e:strainbound}
\| \strain{\bu_{n,h}} \|_{L_{n+1}(\Omega)}
&\leq
\frac{d\sqrt{2} }{n} \left[16d^2 \|\bF\|_{L_{1+\frac{1}{n}}(\Omega)}^{1+\frac{1}{n}}
+ 2 (n+1) C_2 \sqrt{2d}\, | \Omega|^{\frac{1}{n+1}} \|\bF\|_{L_{1+\frac{1}{n}}(\Omega)} + 4 (n+1) C_1 \kappa |\Omega|      \right]^{\frac{1}{n+1}} \nonumber\\
& \qquad + C_2 \sqrt{2d}\, | \Omega|^{\frac{1}{n+1}}
\end{align}
and
\begin{align}\label{e:stressbound}
\begin{aligned}
\frac{1}{n+1}\|\bT_{n,h}\|_{L_{1+\frac{1}{n}}(\Omega)}^{1+\frac{1}{n}}+ C_1 \|\bT_{n,h}\|_{L_1(\Omega)}
&\leq \frac{16d^2}{n+1}\|\bF\|^{1+\frac{1}{n}}_{L_{1+\frac{1}{n}}(\Omega)}\\
&\quad + 2C_2 \sqrt{2d}\, | \Omega|^{\frac{1}{n+1}}\|\bF\|_{L_{1+\frac{1}{n}}(\Omega)} + 4C_1 \kappa |\Omega|.
\end{aligned}
\end{align}
Hence, in particular,
$$
\frac{1}{n+1} \| \bT_{n,h} \|^{1+\frac{1}{n}}_{L_{1+\frac{1}{n}}(\Omega)} + \| \bT_{n,h}\|_{L_1(\Omega)}  + \| \strain{\bu_{n,h}} \|_{L_{n+1}(\Omega)}
\leq c_n\qquad\forall\, n \in \mathbb{N},\quad \forall\, h \in (0,1],
$$
where $c_n$ is a positive constant, depending on $\|\bF\|_{L_{1+\frac{1}{n}}(\Omega)}$, $C_1$, $C_2$, $d$, $\kappa$ and $|\Omega|$ only.
Thanks to the stability inequality \eqref{e:bound_Tf_h} satisfied by the lift $\bT_{n,h}^{\mathbf f}$ we then deduce that
$$
\frac{1}{n+1} \| \bT_{n,h}^0 \|_{L_{1+\frac{1}{n}}(\Omega)}^{1+\frac{1}{n}} + \| \bT_{n,h}^0\|_{L_1(\Omega)}  \leq c_n\qquad \forall\,n \in \mathbb{N}, \quad \forall\, h \in (0,1].
$$
Therefore, for each fixed $n \in \mathbb{N}$ there exists a subsequence with respect to $h$ (and still indexed by $h$) and $\overline{\bT}^0_n \in L_{1+\frac{1}{n}}(\Omega)^{d \times d}_{\rm sym}$,
such that, as $h \rightarrow 0_+$,
\begin{equation}\label{e:weak_conv_T0}
\bT_{n,h}^0  \rightharpoonup \overline{\bT}^0_n \qquad \textrm{weakly in $L_{1+\frac{1}{n}}(\Omega)^{d \times d}_{\rm sym}$}.
\end{equation}

We note that $\overline{\bT}^0_n \in \mathbb{V}_n$, in fact. Indeed, for any $\bv_n \in \mathbb{X}_n$ there exists a
sequence $(\bv_{n,h})_{h \in (0,1]}$, with $\bv_{n,h} \in \mathbb{X}_{n,h}$, such that $\varepsilon(\bv_{n,h}) \rightarrow \varepsilon(\bv_n)$ strongly in $L_{n+1}(\Omega)^{d \times d}_{\rm sym}$.
As
$$b(\bT^0_{n,h},\bv_{n,h}) = b(\bT_{n,h},\bv_{n,h})-b(\bT^{\mathbf{f}}_{n,h},\bv_{n,h}) = (\mathbf{f},\bv_{n,h})-(\mathbf{f},\bv_{n,h})=0,$$
passage to the limit $h \rightarrow 0_+$, using the weak convergence \eqref{e:weak_conv_T0} and the strong convergence $\varepsilon(\bv_{n,h}) \rightarrow \varepsilon(\bv_n)$ in $L_{n+1}(\Omega)^{d \times d}_{\rm sym}$
implies that $b(\overline{\bT}^0_n,\bv_n)=0$ for all $\bv_n \in \mathbb{X}_n$. Hence, $\overline{\bT}^0_n \in \mathbb{V}_n$ thanks to the definition of the linear space $\mathbb{V}_n$.

We now show, using Minty's method, that $\overline{\bT}^0_n$ satisfies the equation \eqref{e:T0}. To this end,
we recall the notation \eqref{e:def_A} for $\mathcal A_n$ and first prove that, for $\bS_{n,h} \in \mathbb V_{n,h}$,
\begin{equation}\label{e:sign_A}
\int_\Omega \mathcal A_n(\bS_{n,h} + \bT_{n,h}^{\mathbf f})  : (\bT_{n,h}^0-\bS_{n,h}) \dd \bx \leq 0.
\end{equation}

We begin the proof of the inequality \eqref{e:sign_A} by invoking
the monotonicity result \eqref{e:A_nmonotone}
to deduce that, for $\bS_{n,h} \in \mathbb V_{n,h}$,
\begin{align*}
&\int_\Omega \left( \mathcal A_n(\bT_{n,h}^0 + \bT_{n,h}^{\mathbf f}) - \mathcal A_n(\bS_{n,h} + \bT_{n,h}^{\mathbf f}) \right) : (\bT_{n,h}^0-\bS_{n,h}) \dd \bx\\
& = \int_\Omega \left( \mathcal A_n(\bT_{n,h}^0 + \bT_{n,h}^{\mathbf f}) - \mathcal A_n(\bS_{n,h} + \bT_{n,h}^{\mathbf f}) \right) : (\bT_{n,h}^0+\bT_{n,h}^{\mathbf f} - \bT_{n,h}^{\mathbf f} -\bS_{n,h}) \dd \bx
\geq 0.
\end{align*}
Moreover, as $\bT_{n,h}^0$ and $\bS_{n,h}$ both belong to $\mathbb V_{n,h}$,  we use the relation \eqref{e:Th0} satisfied by $\bT_{n,h}^0$ to deduce that
$$
\int_\Omega \mathcal A_n(\bT_{n,h}^0 + \bT_{n,h}^{\mathbf f}) : (\bT_{n,h}^0  -  \bS_{n,h}) \dd \bx = 0,
$$
and thus we obtain the inequality \eqref{e:sign_A}.

We can now use the inequality \eqref{e:sign_A} to show that $\overline{\bT}^0_n$ solves the problem \eqref{e:T0}.
To see this, we consider $\Pi_h \bS_n$ for a given  $\bS_n \in \mathbb V_n$.
As $h \rightarrow 0_+$, the weak convergence \eqref{e:weak_conv_T0}, the strong convergence
$\Pi_h \bS_n \rightarrow \bS_{n}$ (by density) and $\bT_{n,h}^{\mathbf f}  = \Pi_h \bT^{\mathbf{f}}_n  \rightarrow \bT^{\mathbf{f}}_n$
(see \eqref{e:Tf_strong}) in $\mathbb M_n$ guarantee that
$$
\bT_{n,h}^0 - \Pi_h \bS_n \rightharpoonup \overline{\bT}^0_n-\bS_n \qquad \textrm{weakly in $\mathbb M_n$},
$$
and
$$
\Pi_h \bS_n + \bT_{n,h}^{\mathbf f} \rightarrow \bS_n + \bT^{\mathbf{f}}_n \qquad \textrm{strongly in $\mathbb M_n$}.
$$
Hence, the  inequality \eqref{e:sign_A} and the continuity of $\mathcal A_n$ (cf. Lemma~\ref{l:A_cont} (ii)) lead to
$$
\int_\Omega \mathcal A_n(\bS_n+\bT^{\mathbf{f}}_n):(\overline{\bT}^0_n-\bS_n) \dd \bx \leq 0 \qquad \forall\, \bS_n \in \mathbb V_n.
$$
Choosing $\bS_n = \overline{\bT}^0_n - t \bW_n$ for $t>0$ and some $\bW_n \in \mathbb V_n$, we get
$$
\int_\Omega \mathcal A_n(\overline{\bT}^0_n+\bT^{\mathbf{f}}_n-t\bW_n):\bW_n \dd \bx \leq 0 \qquad \forall\, \bW_n \in \mathbb V_n.
$$
Thanks to the continuity (and therefore hemicontinuity) of $\mathcal A_n$ (cf., again, Lemma~\ref{l:A_cont} (ii)),
we can pass to the limit $t \to 0_+$ to deduce that
$$
\int_\Omega \mathcal A_n(\overline{\bT}^0_n+\bT^{\mathbf{f}}_n):\bW_n \dd \bx \leq 0 \qquad \forall\, \bW_n \in \mathbb V_n,
$$
and consequently, since  $\mathbb V_n$ is a linear space, after replacing $\bW_n$ by $-\bW_n$ in the inequality above and then
combining the two inequalities,
$$
\int_\Omega \mathcal A_n(\overline{\bT}^0_n  +\bT^{\mathbf{f}}_n):\bW_n \dd \bx = 0 \qquad \forall\, \bW_n \in \mathbb V_n,
$$
which shows that $\overline{\bT}^0_n= \bT^0_n=\bT_n - \bT^{\mathbf{f}}_n$ satisfies equation \eqref{e:T0}, and thus $\bT_{n,h} \rightharpoonup \bT_n$ in $\mathbb M_n$
as $h \to 0_+$.
\end{proof}

\begin{lem}[Strong convergence]\label{l:strong}
Assume that $\bF \in L_{1+\frac{1}{n}}(\Omega)^{d \times d}_{\rm sym}$, that the functions $\lambda$ and $\mu$ satisfy the assumptions \eqref{e:cond_lambda1}--\eqref{e:cond_mu2}, and let
$(\bT_n,\bu_n)$ denote the unique solution to the regularized problem \eqref{e:weak_reg},
with $n \in \mathbb{N}$. Then, for each fixed $n \in \mathbb{N}$,
as $h \rightarrow 0_+$,
\begin{alignat}{2}
\bT_{n,h} &\rightarrow \bT_n&& \qquad \textrm{strongly in $L_{p}(\Omega)^{d\times d}_{\rm sym}$ for all $p \in \big[1,1+\textstyle{\frac{1}{n}}\big)$},\label{e:T-Lp-strong}\\
\strain{\bu_{n,h}} &\rightharpoonup \strain{\bu_n}&& \qquad \textrm{weakly in $L_{p}(\Omega)^{d\times d}_{\rm sym}$ for all $p \in [1,n+1]$}. \label{e:u-Lp-weak}
\end{alignat}
When $n=1$, the strong convergence result \eqref{e:T-Lp-strong} holds for all $p \in [1,2]$. In addition, for each $n \in \mathbb{N}$, %
\begin{align}
\bu_{n,h} &\rightarrow \bu_n&& \qquad \left\{\begin{array}{lr} \textrm{strongly in $L_{p}(\Omega)^d$ for all $p \in \big[1,\frac{d(n+1)}{d-(n+1)}\big)$ when $1\leq n \leq d-1$},\\
\textrm{strongly in $C^{0,\alpha}(\overline{\Omega})^d$ for all $\alpha \in \big(0,1-\frac{d}{n+1}\big)$ when $d<n+1$},
\end{array}\right. \label{e:u-Lp-strong}
\end{align}
and for each $n \in \mathbb{N}$, $n\ge 2$,
\begin{equation}
\strain{\bu_{n,h}}  \rightarrow \strain{\bu_n} \qquad \textrm{strongly in $L_{n}(\Omega)^{d\times d}_{\rm sym}$.}
\label{e:u-H1-strong}
\end{equation}

Furthermore, if $\lambda$ satisfies \eqref{e:cond_lambda2bis}, we have that, for any $\Omega_0 \subset \subset \Omega$,
\[ \lim_{n \rightarrow \infty} \lim_{h \rightarrow 0_+} \| \bT_{n,h} - \bT \|_{L_1(\Omega_0)} = 0\qquad \mbox{and}\qquad
\lim_{n \rightarrow \infty} \lim_{h \rightarrow 0_+} \| \bu_{n,h} - \bu \|_{C(\overline\Omega)} = 0,\]
and
$$
\lim_{n \rightarrow \infty} \lim_{h \rightarrow 0_+} \|\strain{\bu_{n,h}} - \strain{\bu} \|_{L_p(\Omega_0)} = 0
\qquad \forall\, \Omega_0\ \subset \subset \Omega, \quad \forall\, p \in [1,\infty).$$
where $(\bT,\bu)$ denotes the unique solution of the original (nonregularized) continuous problem \eqref{e:div}, \eqref{e:strain_stress}
subject to a homogeneous Dirichlet boundary condition on $\partial \Omega$.
\end{lem}

\begin{proof}
In this proof, again, $C$ denotes a generic positive constant, independent of $h$ and $n$. Also, we use again the notation
$\bT_n = \bT^0_n + \bT^{\mathbf{f}}_n$ and $\bT_{n,h} = \bT_{n,h}^0 + \bT_{n,h}^{\mathbf f}$, where
$\bT^{\mathbf{f}}_n=\bF$ satisfies the equality \eqref{e:constraint} and $\bT^{\mathbf{f}}_{n,h}= \Pi_h \bT^{\mathbf{f}}_n$ satisfies the equality \eqref{e:constraint_h}.

To establish control on $\bT_{n,h}^0 - \bT^0_n$, we write
\begin{equation}\label{e:decomp}
\bT_{n,h}^0 - \bT^0_n = (\bT_{n,h}^0 - \Pi_h \bT^0_n) + (\Pi_h \bT^0_n - \bT^0_n).
\end{equation}

Since $\Pi_h \bT^0_n  \rightarrow \bT^0_n$ strongly in $L_{1+\frac{1}{n}}(\Omega)^{d\times d}_{\rm sym}$ for all $n \in \mathbb{N}$, it suffices to
focus on the discrepancy $ \bT_{n,h}^0 - \Pi_h \bT_{n}^0$.

Thanks to the inequality \eqref{e:mono1}, for any pair of
matrices $\bS, \bR \in \mathbb{R}^{d \times d}_{\rm sym}$, one has
\begin{align*}
(\bS|\bS|^{\frac{1}{n}-1} - \bR|\bR|^{\frac{1}{n}-1}):(\bS - \bR) &\geq \frac{1}{n}|\bS - \bR|^2 \int_0^1 |\bR + \theta (\bS-\bR)|^{\frac{1}{n}-1}
\dd \theta \\
& \geq  \frac{1}{n}\frac{|\bS - \bR|^2}{ (|\bR| + |\bS - \bR|)^{1-\frac{1}{n}}}.
\end{align*}
Analogously, for any pair of real numbers $s, r \in \mathbb{R}$,
\begin{align*}
(s|s|^{\frac{1}{n}-1} - r|r|^{\frac{1}{n}-1})\,(s - r)
& \geq  \frac{1}{n}\frac{|s - r|^2}{ (|r| + |s - r|)^{1-\frac{1}{n}}}.
\end{align*}
Hence, and by invoking the inequalities \eqref{e:mon_mu1} and \eqref{e:mon_lambda1} (guaranteed by the assumptions \eqref{e:cond_lambda1}--\eqref{e:cond_mu2}), we have that
\begin{equation}\label{e:strong_conv1bis}
\begin{split}
&\int_\Omega \left(\mathcal A_n (\bT_{n,h}^0 + \bT_{n,h}^{\mathbf f}) - \mathcal A_n(\Pi_h \bT^0_n + \bT_{n,h}^{\mathbf f}) \right):(\bT_{n,h}^0-\Pi_h \bT^0_n) \dd \bx\\
&\qquad \geq
\frac{1}{n^2}\int_\Omega \frac{|\tr{\bT_{n,h}^0 + \bT_{n,h}^{\mathbf f}} - \tr{\Pi_h \bT^0_n + \bT_{n,h}^{\mathbf f}}|^2}
{(|\tr{\Pi_h \bT^0_n + \bT_{n,h}^{\mathbf f}}| + |\tr{\bT_{n,h}^0 + \bT_{n,h}^{\mathbf f}} - \tr{\Pi_h \bT^0_n + \bT_{n,h}^{\mathbf f}}|)^{1-\frac{1}{n}}} \dd \bx \\
&\qquad \quad \,\,+ \frac{1}{n^2}\int_\Omega \frac{|(\bT_{n,h}^0 + \bT_{n,h}^{\mathbf f})^\bd - (\Pi_h \bT^0_n + \bT_{n,h}^{\mathbf f})^\bd|^2}
{(|(\Pi_h \bT^0_n + \bT_{n,h}^{\mathbf f})^\bd| + |(\bT_{n,h}^0 + \bT_{n,h}^{\mathbf f})^\bd - (\Pi_h \bT^0_n + \bT_{n,h}^{\mathbf f})^\bd|)^{1-\frac{1}{n}}}
\dd \bx.
\end{split}
\end{equation}

Moreover, because $\bT^0_n  = \overline{\bT}^0_n \in \mathbb V_n$  (cf. the last sentence in the proof of Lemma \ref{l:weak}), we have
$$
\int_\Omega \strain{\bv_h} : \Pi_h \bT^0_n \dd \bx = \int_\Omega \strain{\bv_h} :  \bT^0_n  \dd \bx = 0 \qquad \forall\, \bv_h \in \mathbb X_{n,h},
$$
and so $\Pi_h \bT^0_n \in \mathbb V_{n,h}$.
As a consequence, $\bT_{n,h}^0 - \Pi_h \bT^0_n \in \mathbb V_{n,h}$ and there holds
$$
\int_\Omega \mathcal A_n (\bT_{n,h}^0 + \bT_{n,h}^{\mathbf f}):(\bT_{n,h}^0-\Pi_h \bT^0_n) \dd \bx = 0.
$$
Using this in the inequality \eqref{e:strong_conv1bis} and noting that $\Pi_h \bT^0_n + \bT_{n,h}^{\mathbf f} = \Pi_h  (\bT^0_n + \bT_{n}^{\mathbf f}) = \Pi_h \bT_n$, we obtain
\begin{equation}\label{e:strong_conv2bis}
\begin{split}
&- \int_\Omega \mathcal A_n(\Pi_h \bT^0_n + \bT_{n,h}^{\mathbf f}) :(\bT_{n,h}^0-\Pi_h \bT^0_n) \dd \bx\\
&\qquad\geq
\frac{1}{n^2}\int_\Omega \frac{|\tr{\bT_{n,h}^0 -\Pi_h \bT^0_n}|^2}
{(|\tr{\Pi_h \bT _n}| + |\tr{\bT_{n,h}^0 -\Pi_h \bT^0_n}|)^{1-\frac{1}{n}}} \dd \bx \\
&\qquad\qquad + \frac{1}{n^2}\int_\Omega \frac{|(\bT_{n,h}^0 -\Pi_h \bT^0_n)^\bd|^2}
{(|(\Pi_h \bT_n)^\bd| + |(\bT_{n,h}^0 - \Pi_h \bT^0_n)^\bd|)^{1-\frac{1}{n}}}
\dd \bx \geq 0.
\end{split}
\end{equation}
On the one hand, $\bT_{n,h}^0 - \Pi_h \bT^0_n$ weakly converges to $\bT^0_n - \bT^0_n = \mathbf 0$ in $\mathbb{M}_n$ as $h \rightarrow 0_+$ (cf. Lemma~\ref{l:weak}). On the other hand, $\Pi_h \bT^0_n + \bT_{n,h}^{\mathbf f}$ strongly converges to $\bT^0_n + \bT^{\mathbf{f}}_n$ in $\mathbb{M}_n
= L_{1+\frac{1}{n}}(\Omega)^{d \times d}_{\rm sym}$ as $h \rightarrow 0_+$. Therefore the continuity of the mapping $\mathcal A_n : L_{1+\frac{1}{n}}(\Omega)^{d \times d}_{\rm sym} \rightarrow L_{n+1}(\Omega)^{d \times d}_{\rm sym}$ (cf. Lemma~\ref{l:A_cont} (ii)), which implies that $\mathcal A_n(\Pi_h \bT^0_n + \bT_{n,h}^{\mathbf f})$ strongly converges to
$\mathcal{A}_n(\bT^0_n + \bT^{\mathbf{f}}_n)$ in $L_{n+1}(\Omega)^{d \times d}_{\rm sym}$ as $h \rightarrow 0_+$,
yields
$$
- \int_\Omega \ \mathcal A_n(\Pi_h \bT^0_{n} + \bT_{n,h}^{\mathbf f}):(\bT_{n,h}^0-\Pi_h \bT^0) \dd \bx\rightarrow 0\qquad \mbox{as $h \rightarrow 0_+$}.
$$
Whence, returning to the inequality \eqref{e:strong_conv2bis},
\begin{align*}
&0 \leq
\frac{1}{n^2}\int_\Omega \frac{|\tr{\bT_{n,h}^0 -\Pi_h \bT^0_n}|^2}
{(|\tr{\Pi_h \bT _n}| + |\tr{\bT_{n,h}^0 -\Pi_h \bT^0_n}|)^{1-\frac{1}{n}}} \dd \bx \\
&\qquad + \frac{1}{n^2}\int_\Omega \frac{|(\bT_{n,h}^0 -\Pi_h \bT^0_n)^\bd|^2}
{(|(\Pi_h \bT_n)^\bd| + |(\bT_{n,h}^0 - \Pi_h \bT^0_n)^\bd|)^{1-\frac{1}{n}}}
\dd \bx \rightarrow 0\qquad \mbox{as $h \rightarrow 0_+$}.
\end{align*}

\noindent
Consequently, for each $n \in \mathbb{N}$,
\begin{align*}
\lim_{h \rightarrow 0_+} \int_\Omega \frac{|\tr{\bT_{n,h}^0 -\Pi_h \bT^0_n}|^2}
{(|\tr{\Pi_h \bT _n}| + |\tr{\bT_{n,h}^0 -\Pi_h \bT^0_n}|)^{1-\frac{1}{n}}} \dd \bx &= 0,\\
\lim_{h \rightarrow 0_+} \int_\Omega \frac{|(\bT_{n,h}^0 -\Pi_h \bT^0_n)^\bd|^2}
{(|(\Pi_h \bT_n)^\bd| + |(\bT_{n,h}^0 - \Pi_h \bT^0_n)^\bd|)^{1-\frac{1}{n}}}
\dd \bx &= 0.
\end{align*}

In the special case when $n=1$, we directly deduce from these, the equality \eqref{e:normT0} and the strong convergence
of $\Pi_h \bT^0_n$ to $\bT^0_n$ in $L_2(\Omega)^{d \times d}_{\rm sym}$,  that $\bT_{n,h}^0 \rightarrow \bT_{n}^0$
strongly in $L_2(\Omega)^{d \times d}_{\rm sym}$, as $h \rightarrow 0_+$. Since  $\bT^{\mathbf{f}}_{n,h}= \Pi_h\bT^{\mathbf{f}}_n\rightarrow \bT^{\mathbf{f}}_n$ strongly in $L_2(\Omega)^{d \times d}_{\rm sym}$ as
$h \rightarrow 0_+$, it follows that, for $n=1$, $\bT_{n,h} \rightarrow \bT_n$ strongly in $L_2(\Omega)^{d \times d}_{\rm sym}$, and therefore also strongly in $L_p(\Omega)^{d \times d}_{\rm sym}$ for all $p\in [1,2]$,
as $h \rightarrow 0_+$. That completes the proof of the assertion of the lemma concerning $(\bT_{n,h})_{h \in (0,1]}$ for $n=1$.

Let us now consider the case when $n>1$.
Let $\mathcal{M}(f)$ denote the Hardy--Littlewood maximal function of $f \in L_1(\Omega)$, with $f$
extended by zero outside $\Omega$ to the whole of $\mathbb{R}^d$, and let $B_r(\bx)$ denote the
$d$-dimensional ball of radius $r$ centred at $\bx \in \mathbb{R}^d$. Clearly,
\[ |\Pi_h \bT_n(\bx)| \leq \frac{1}{|K|} \int_K |\bT_n(\by)| \dd \by
\leq \frac{|B_{h_K}(\bx)|}{|K|}\left(\frac{1}{|B_{h_K}(\bx)|} \int_{B_{h_K}(\bx)}|\bT_n(\by)| \dd \by\right)
\leq c(\eta)\,\mathcal{M}(|\bT_n|)(\bx)\]
for all $\bx \in K$ and all $K \in \mathcal{T}_h$, where $h_K = \mathrm{diam}(K)$ and $c(\eta)$ is a positive constant
that only depends on the shape-regularity parameter $\eta$ of the family $(\mathcal{T}_h)_{h \in (0,1]}$ of simplicial
subdivisions of the domain $\overline\Omega$ (see \eqref{e:reg}). Thus,
\begin{equation}\label{e:maxfbound}
|\Pi_h \bT_n(\bx)|
\leq c(\eta)\,\mathcal{M}(|\bT_n|)(\bx) \qquad \forall\,\bx \in \Omega.
\end{equation}
Since the Hardy--Littlewood maximal function is of weak-type $(L_1,L_{1,\infty})$ (with $L_{1,\infty}$ signifying a Lorentz space) with norm at most $3^d$ (cf. Theorem 2.1.6 and inequality (2.1.3) in \cite{ref:Grafakos}), we have that
\[ |\{ \bx \in \Omega\,:\, \mathcal{M}(|\bT_n|)(\bx) > t\}| \leq \frac{3^d}{t}\, \|\bT_n\|_{L_1(\Omega)} \qquad \forall\, t>0.\]
For $k \in \mathbb{N}$ we define
\[ \Omega_k := \{\bx \in \Omega\, : \, \mathcal{M}(|\bT_n|)(\bx) \leq k\}. \]
Hence,
\begin{align}\label{e:nested}
\Omega_1 \subset \Omega_2 \subset \cdots \subset \Omega\quad \mbox{and}\quad |\Omega \setminus \Omega_k| \leq \frac{3^d}{k} \,
\|\bT_n\|_{L_1(\Omega)}\qquad \forall\, k \in \mathbb{N};
\end{align}
in particular,
\begin{align}
\label{e:nested1} \lim_{k \rightarrow \infty} |\Omega \setminus \Omega_k| = 0.
\end{align}

By recalling \eqref{e:normT0}, \eqref{e:maxfbound} and the definition of the set $\Omega_k$, we have that
\begin{align*}
|\tr{\Pi_h \bT _n}(\bx)| &\leq d^{\frac{1}{2}}|\Pi_h \bT _n(\bx)| \leq d^{\frac{1}{2}}\,c(\eta)\, \mathcal{M}(|\bT _n|)(\bx) \leq d^{\frac{1}{2}}\,c(\eta)\,k
\qquad \forall\, \bx \in \Omega_k,\quad \forall\, k \in \mathbb{N},\\
|(\Pi_h \bT_n)^\bd(\bx)| &\leq |\Pi_h \bT _n(\bx)| \leq c(\eta)\, \mathcal{M}(|\bT _n|)(\bx) \leq c(\eta)\,k\qquad \forall\, \bx \in \Omega_k, \quad \forall\, k \in \mathbb{N}.
\end{align*}
Thus we deduce that, for each $k \in \mathbb{N}$,
\begin{align*}
&\lim_{h \rightarrow 0_+} \int_{\Omega_k} \frac{|\tr{\bT_{n,h}^0 -\Pi_h \bT^0_n}|^2}
{(\sqrt{d}c(\eta)\, k + |\tr{\bT_{n,h}^0 -\Pi_h \bT^0_n}|)^{1-\frac{1}{n}}} \dd \bx = 0,\\
&\lim_{h \rightarrow 0_+} \int_{\Omega_k} \frac{|(\bT_{n,h}^0 -\Pi_h \bT^0_n)^\bd|^2}
{(c(\eta)\, k + |(\bT_{n,h}^0 - \Pi_h \bT^0_n)^\bd|)^{1-\frac{1}{n}}}
\dd \bx = 0,
\end{align*}
as $h \rightarrow 0_+$; hence, for each $k$ there exists a null-sequence $(h^{(k)}) \subset (0,1]$, with $(h^{(k+1)}) \subset (h^{(k)})$ for all $k \in \mathbb{N}$,
such that
\[ \frac{|\tr{\bT_{n,{h^{(k)}}}^0 -\Pi_{h^{(k)}} \bT^0_n}|^2}
{(\sqrt{d}c(\eta)\, k + |\tr{\bT_{n,{h^{(k)}}}^0 -\Pi_{h^{(k)}} \bT^0_n}|)^{1-\frac{1}{n}}} \rightarrow 0\quad\mbox{and}\quad
\frac{|(\bT_{n,{h^{(k)}}}^0 -\Pi_{h^{(k)}} \bT^0_n)^\bd|^2}
{(c(\eta)\,k + |(\bT_{n,{h^{(k)}}}^0 - \Pi_{h^{(k)}} \bT^0_n)^\bd|)^{1-\frac{1}{n}}} \rightarrow 0\]
a.e. on $\Omega_k$ as ${h^{(k)}} \rightarrow 0_+$. Since
\begin{align}\label{e:a-ineq}
\frac{a^2}{(k+a)^{1-\frac{1}{n}}} \geq 2^{\frac{1}{n}-1} \min\left(\frac{a^2}{k^{1-\frac{1}{n}}}, a^{1+\frac{1}{n}}\right)\quad
\forall\, a \geq 0,\quad \forall\, k \in \mathbb{N},
\end{align}
it follows that
\[ |\tr{\bT_{n,{h^{(k)}}}^0 -\Pi_{h^{(k)}} \bT^0_n}| \rightarrow 0\quad\mbox{and}\quad
|(\bT_{n,h^{(k)}}^0 -\Pi_{h^{(k)}} \bT^0_n)^{\bd}| \rightarrow 0\qquad \mbox{a.e. on $\Omega_k, \;\forall\, k \in \mathbb{N}$}\]
as $h^{(k)} \rightarrow 0_+$. We then deduce from inequality \eqref{e:normT1} that
\[ \bT_{n,h^{(k)}}^0 -\Pi_{h^{(k)}} \bT^0_n \rightarrow 0\qquad \mbox{a.e. on $\Omega_k$ for all $k \in \mathbb{N}$ as $h^{(k)} \rightarrow 0_+$}.\]
Hence,
\[ \bT_{n,h^{(k)}}^0 \rightarrow \bT^0_n \qquad \mbox{a.e. on $\Omega_k$ for all $k \in \mathbb{N}$ as $h^{(k)} \rightarrow 0_+$}.\]
By Cantor's diagonal argument we can then extract a `diagonal' null-sequence $(h^{(\infty)})$ such that
\[ \bT_{n,h^{(\infty)}}^0 \rightarrow \bT^0_n \qquad \mbox{a.e. on $\Omega_k$ for all $k \in \mathbb{N}$ as $h^{(\infty)} \rightarrow 0_+$}.\]
Since the sets $\Omega_k$ are nested (cf. \eqref{e:nested}) and they exhaust the whole of $\Omega$ (cf. \eqref{e:nested1}),
it follows that
\[ \bT_{n,h^{(\infty)}}^0 \rightarrow \bT^0_n \qquad \mbox{a.e. on $\Omega$ as $h^{(\infty)} \rightarrow 0_+$}.\]
For the sake of simplicity of our notation we shall henceforth suppress the superscript $^{(\infty)}$ and will simply write
\[ \bT_{n,h}^0  \rightarrow  \bT^0_n\qquad \mbox{a.e. on $\Omega$ as $h \rightarrow 0_+$}.\]
As $\bT^{\mathbf{f}}_{n.h}=\Pi_h \bT^{\mathbf{f}}_n  \rightarrow \bT^{\mathbf{f}}_n$ strongly in $L_{1+\frac{1}{n}}(\Omega)^{n \times n}_{\rm sym}$, and therefore (for a subsequence, not indicated) a.e. in $\Omega$, it follows that
\[ \bT_{n,h}=\bT_{n,h}^0 + \bT^{\mathbf{f}}_{n.h} \rightarrow \bT^0_n+  \bT^{\mathbf{f}}_n = \bT_n\qquad \mbox{a.e. on $\Omega$ as $h \rightarrow 0_+$.}\]
As, by Lemma \ref{l:weak}, $\bT_{n,h} \rightharpoonup \bT_n$ weakly in $L_{1+\frac{1}{n}}(\Omega)^{d \times d}_{\rm sym}$, it
follows that the sequence $(\bT_{n,h})_{h \in (0,1]}$ is equiintegrable in $L_{1}(\Omega)^{d \times d}_{\rm sym}$,
and therefore by Vitali's theorem
\begin{equation}\label{e:T-L1-conv}
\bT_{n,h} \rightarrow \bT_n\qquad \mbox{strongly in $L_{1}(\Omega)^{d \times d}_{\rm sym}$ as $h \rightarrow 0_+$},
\end{equation}
whereby, because of the weak convergence $\bT_{n,h} \rightharpoonup \bT_n$ in $L_{1+\frac{1}{n}}(\Omega)^{d \times d}_{\rm sym}$, it follows that
\[\bT_{n,h} \rightarrow \bT_n\qquad \mbox{strongly in $L_{p}(\Omega)^{d \times d}_{\rm sym}$
for all $p \in [1,1+\frac{1}{n})$, as $h \rightarrow 0_+$,}\]
where the limiting function $\bT_n$ is the first component of the \textit{unique} solution $(\bT_n,\bu_n)$ of the regularized problem. That completes the proof of the strong convergence result \eqref{e:T-Lp-strong} for $n>1$. For $n=1$, \eqref{e:T-Lp-strong} was already shown in the first part of this proof for all $p \in [1,2]$; hence \eqref{e:T-Lp-strong} holds for all $n \in \mathbb{N}$.

To prove the strong convergence of the sequence $(\bu_{n,h})_{h \in (0,1]} \subset \mathbb X_{n,h}$ to $\bu_n \in \mathbb{X}_n$,
we note that the inequality \eqref{e:strainbound} implies that the sequence
$(\strain{\bu_{n,h}})_{h \in (0,1]}$ is bounded in $L_{n+1}(\Omega)^{d \times d}_{\rm sym}$. Hence there exists a
subsequence, not indicated, and $\overline{\strain{\bu_{n,h}}} \in L_{n+1}(\Omega)^{d \times d}_{\rm sym}$ such that, as $h \rightarrow 0_+$,
%
\begin{align}\label{e:weak-conv-1}
\strain{\bu_{n,h}} \rightharpoonup \overline{\strain{\bu_{n,h}}}\qquad \mbox{weakly in $L_{n+1}(\Omega)^{d \times d}_{\rm sym}$}.
\end{align}
Here, and henceforth, for any weakly (respectively, strongly) convergent sequence of the form $(a_{n,h})_{h \in (0,1]}$
in a function space, with $n \geq 1$ held fixed, $\overline{a_{n,h}}$ will denote the weak (respectively, strong) limit
of the sequence as $h \rightarrow 0_+$, in instances where the limit of the sequence is yet to be identified.

This will imply the assertion \eqref{e:u-Lp-weak} once we have shown that $\overline{\strain{\bu_{n,h}}}= \strain{\bu_{n}}$, which we shall do now.
For $1 \leq n \leq d-1$, Korn's inequality \eqref{e:Korn1} and Poincar\'e's inequality \eqref{e:poincare} together
imply that $(\bu_{n,h})_{h \in (0,1]}$ is bounded in $W_0^{1,n+1}(\Omega)^{d }$, and by Kondrashov's
compact embedding theorem the sequence therefore possesses
a strongly convergent subsequence (not indicated), with limit $\overline{\bu_{n,h}} \in L_p(\Omega)^d$, such that
\[ \bu_{n,h} \rightarrow \overline{\bu_{n,h}}\qquad \mbox{strongly in $L_p(\Omega)^d$ as $h \rightarrow 0_+$ for all $p \in \big[1,\frac{(n+1)d}{d-(n+1)}\big)$}.
\]
This will imply the first line of the assertion \eqref{e:u-Lp-strong} once we have shown that $\overline{\bu_{n,h}}= \bu_{n}$, which we shall do below. In any case, by the uniqueness of the weak limit it then follows that $\overline{\strain{\bu_{n,h}}} = \strain{\overline{\bu_{n,h}}}$,
and therefore
\[  \strain{\bu_{n,h}} \rightharpoonup \strain{\overline{\bu_{n,h}}}\qquad \mbox{weakly in $L_{n+1}(\Omega)^{d \times d}_{\rm sym}$ as $h \rightarrow 0_+$.}\]

For $n>d-1$, by an analogous argument,
\[ \bu_{n,h} \rightarrow \overline{\bu_{n,h}}\qquad \mbox{strongly in $C^{0,\alpha}(\overline\Omega)^d$
as $h \rightarrow 0_+$ for all
$\alpha \in (0,1 - \frac{d}{n+1})$.}
\]
This will imply the second line of the assertion \eqref{e:u-Lp-strong} provided we show  that $\overline{\bu_{n,h}}= \bu_{n}$, the second component of the unique solution $(\bT_n,\bu_n)$
of the regularized problem. We shall do so by passing to the limit in equation \eqref{e:discrete_reg}$_1$.

To this end, take any $\bS \in L_{1+\frac{1}{n}}(\Omega)^{d \times d}_{\rm sym}$ and let $\bS_h:=\Pi_h \bS$ in equation
\eqref{e:discrete_reg}$_1$, resulting in
\begin{align}\label{e:discrete_reg1}
&a_n(\bT_{n,h},\Pi_h \bS)+c(\bT_{n,h};\bT_{n,h},\Pi_h \bS)-b(\Pi_h \bS,\bu_{n,h}) =0.
\end{align}
As
\begin{align}\label{e:proj-strong}
\Pi_h \bS \rightarrow \bS\qquad \mbox{strongly in $L_{1+\frac{1}{n}}(\Omega)^{d \times d}_{\rm sym}$}
\end{align}
it follows from the weak convergence \eqref{e:weak-conv-1} that, for each $n \in \mathbb{N}$,
\begin{align}\label{e:b-conv}
\lim_{h \rightarrow 0_+} b(\Pi_h \bS,\bu_{n,h}) = b(\bS,\overline{\bu_{n,h}}) \qquad \forall\, \bS \in L_{1+\frac{1}{n}}(\Omega)^{d \times d}_{\rm sym}.
\end{align}
We shall now pass to the limit $h \rightarrow 0_+$ in the first two terms on the left-hand side of the equation
\eqref{e:discrete_reg1}.

Thanks to the strong convergence result $\bT_{n,h} \rightarrow \bT_n$ in $L_1(\Omega)^{d \times d}_{\rm sym}$,
which follows from the assertion \eqref{e:T-Lp-strong} for all $n \in \mathbb{N}$, as $h \rightarrow 0_+$, an identical argument to the one in the proof of Lemma \ref{l:A_cont} (ii) implies that, as $h \rightarrow 0_+$,
$$
\lambda(\tr{\bT_{n,h} }) \tr{\bT_{n,h}} \bI \rightarrow \lambda(\tr{\bT_{n}}) \tr{\bT_{n}} \bI \quad \textrm{strongly in }
L_{n+1}(\Omega)^{d \times d}_\textrm{sym}$$
and
$$
\mu(| \bT^{\bd}_{n,h}|) \bT^{\bd}_{n,h}  \rightarrow \mu(| \Bk \bT^{\bd}_{n} |) \bT^{\bd}_{n} \quad \textrm{strongly in }
L_{n+1}(\Omega)^{d \times d}_\textrm{sym}.$$
Together with the strong convergence \eqref{e:proj-strong} these then imply that, for each $n \in \mathbb{N}$,
\begin{align}\label{e:c-conv} \lim_{h \rightarrow 0_+} c(\bT_{n,h};\bT_{n,h},\Pi_h \bS) = c(\bT_{n};\bT_{n}, \bS)\qquad
\forall\, \bS \in L_{1+\frac{1}{n}}(\Omega)^{d \times d}_{\rm sym}.
\end{align}

Finally, we consider the first term on the left-hand side of the equation \eqref{e:discrete_reg1}. By the inequality \eqref{e:holder},
\begin{align*}
 \bigg|\frac{\tr{\bT_{n,h} }}{|\tr{\bT_{n,h} }|^{1-\frac{1}{n}}}  - \frac{\tr{\bT_{n} }}{|\tr{\bT_{n}}|^{1-\frac{1}{n}}} \bigg|
\leq 2^{1-\frac{1}{n}}d^{\frac{1}{2n}}\,|\bT_{n,h}  - \bT_{n} |^{\frac{1}{n}}.
\end{align*}
Thus, because of the strong convergence \eqref{e:T-Lp-strong}, we have that, as $h \rightarrow 0_+$,
\[ \frac{\tr{\bT_{n,h} }}{|\tr{\bT_{n,h} }|^{1-\frac{1}{n}}} \bI \rightarrow \frac{\tr{\bT_{n} }}{|\tr{\bT_{n}}|^{1-\frac{1}{n}}}\bI
\qquad \mbox{strongly in $L_p(\Omega)^{d \times d}_{\rm sym}$ for all $p \in [1,n+1)$}.\]
Furthermore, by the uniform bound \eqref{e:stressbound}, for each fixed $n \in \mathbb{N}$,
\[ \left(\frac{\tr{\bT_{n,h} }}{|\tr{\bT_{n,h} }|^{1-\frac{1}{n}}} \bI\right)_{h \in (0,1]}\]
is a bounded sequence in $L_{n+1}(\Omega)^{d \times d}_{\rm sym}$, which therefore has a weakly convergent subsequence (not indicated),
whose (weak) limit in $L_{n+1}(\Omega)^{d \times d}_{\rm sym}$, by the uniqueness of the weak limit, coincides with
\[ \frac{\tr{\bT_{n} }}{|\tr{\bT_{n} }|^{1-\frac{1}{n}}} \bI.\]
Hence, as $ h \rightarrow 0_+$,
\[ \frac{\tr{\bT_{n,h} }}{|\tr{\bT_{n,h} }|^{1-\frac{1}{n}}} \bI \rightharpoonup \frac{\tr{\bT_{n} }}{|\tr{\bT_{n}}|^{1-\frac{1}{n}}}\bI
\qquad \mbox{weakly in $L_{n+1}(\Omega)^{d \times d}_{\rm sym}$}.\]
By an identical argument,
\[ \frac{{\bT^\bd_{n,h} }}{|{\bT^\bd_{n,h} }|^{1-\frac{1}{n}}} \rightharpoonup \frac{{\bT^\bd_{n} }}{|{\bT^\bd_{n}}|^{1-\frac{1}{n}}}
\qquad \mbox{weakly in $L_{n+1}(\Omega)^{d \times d}_{\rm sym}$}.\]
By combining these two weak convergence results with the strong convergence result \eqref{e:proj-strong} we deduce that, for each $n \in \mathbb{N}$,
\begin{align}\label{e:a-conv} \lim_{h \rightarrow 0_+} a_n(\bT_{n,h},\Pi_h \bS) = a_n(\bT_{n}, \bS)\qquad
\forall\, \bS \in L_{1+\frac{1}{n}}(\Omega)^{d \times d}_{\rm sym}.
\end{align}

Using the convergence results \eqref{e:b-conv}, \eqref{e:c-conv} and \eqref{e:a-conv} we can now pass to the limit $h \rightarrow 0_+$ in equation \eqref{e:discrete_reg1} to deduce that
\begin{align}\label{e:discrete_reg2}
&a_n(\bT_{n},\bS)+c(\bT_{n};\bT_{n}, \bS)-b(\bS,\overline{\bu_{n,h}}) =0 \qquad \forall\, \bS \in L_{1+\frac{1}{n}}(\Omega)^{d \times d}_{\rm sym}.
\end{align}
By subtracting equation \eqref{e:discrete_reg2} from equation \eqref{e:weak_reg} we deduce that
\[ b(\bS,\overline{\bu_{n,h}}-\bu_n) = 0 \qquad \forall\, \bS \in L_{1+\frac{1}{n}}(\Omega)^{d \times d}_{\rm sym}.\]
Hence,
\[ \strain{\overline{\bu_{n,h}}-\bu_n} = \mathbf{0}\qquad \mbox{in $L_{n+1}(\Omega)^{d \times d}_{\rm sym}$}.\]
Thus, by noting the inequality \eqref{e:korn} we deduce that
\[ \overline{\bu_{n,h}}-\bu_n = {\bf 0}  \qquad \mbox{in $W_0^{1,n+1}(\Omega)^{d}$}.\]
In other words, $\overline{\bu_{n,h}}=\bu_n \in W_0^{1,n+1}(\Omega)^{d}$, as has been asserted above.

The strong convergence \eqref{e:u-H1-strong} in $L_n(\Omega)^d$ for $n\ge 2$ follows by an
argument which we have already used, so we only sketch the proof.
For any $\bS_h \in \mathbb M_{n,h}$, the constitutive relations in \eqref{e:weak_reg} and \eqref{e:discrete_reg} imply
$$
\int_\Omega (\strain{\bu_{n,h}} - \strain{\bu_n}):  \bS_h  \dd \bx  = \int_\Omega \mathcal (\mathcal A_n(\bT_{n,h}) - \mathcal A_n(\bT_n)):\bS_h  \dd \bx  \leq \| \mathcal A_n(\bT_{n,h}) - \mathcal A_n(\bT_n))\|_{L_{n}(\Omega)} \| \bS_h \|_{L_{\frac n {n-1}}(\Omega)}.
$$
Now, using an argument similar to the one leading to \eqref{e:holder}, we find that
\begin{equation}\label{e:reg_conv_0}
\begin{split}
\frac 1 n \int_{\Omega}\bigg|\frac{\tr{\bT_{n,h}}}{|\tr{\bT_{n,h}}|^{1-\frac{1}{n}}}  - \frac{\tr{\bT_n}}{|\tr{\bT_n}|^{1-\frac{1}{n}}} \bigg|^n  \dd \bx
&+ \frac 1 n \int_{\Omega}\bigg|\frac{\bT_{n,h}^{\mathbf d}}{|\bT_{n,h}^{\mathbf d}|^{1-\frac{1}{n}}}  - \frac{\bT_n^{\mathbf d}}{|\bT_n^{\mathbf d}|^{1-\frac{1}{n}}} \bigg|^n  \dd \bx  \\
& \leq C \| \bT_{n,h} - \bT_n \|_{L_1(\Omega)} \rightarrow 0 \qquad \textrm{as }h\to 0_+,
\end{split}
\end{equation}
for a constant $C$ only depending on $d$ and $n$.
For the monotone part, $\mathcal A$, in $\mathcal A_n$ (cf. \eqref{e:A_no_n}), we invoke a
similar argument to the one used in Lemma~\ref{l:A_cont} to deduce that
$$
 \| \mathcal A(\bT_{n,h}) - \mathcal A(\bT_n)\|_{L_{n}(\Omega)} \rightarrow 0 \qquad \textrm{as }h \to 0_+.
$$
Hence, in conjunction with \eqref{e:reg_conv_0}, we arrive at
$$
\frac{1}{\| \bS_h \|_{L_{\frac n {n-1}}(\Omega)}}\int_\Omega (\strain{\bu_{n,h}} - \strain{\bu_n}):
\bS_h  \dd \bx  \rightarrow 0 \qquad \textrm{as }h\to 0_+.
$$
Using the decomposition $\strain{\bu_{n,h}} - \strain{\bu_n} = \strain{\bu_{n,h}} - \Pi_h(\strain{\bu_n}) + \Pi_h(\strain{\bu_n}) - \strain{\bu_n}$, we write
\begin{align*}
&\frac{1}{\| \bS_h \|_{L_{\frac n {n-1}}(\Omega)}} \int_\Omega (  \strain{\bu_{n,h}} - \Pi_h(\strain{\bu_n})) : \bS_h   \dd \bx 
\\
&\qquad = \frac{1}{\| \bS_h \|_{L_{\frac n {n-1}}(\Omega)}} \int_\Omega (  \strain{\bu_{n,h}} - \strain{\bu_n}) : \bS_h  \dd \bx
- \frac{1}{\| \bS_h \|_{L_{\frac n {n-1}}(\Omega)}} \int_\Omega (  \Pi_h (\strain{\bu_{n}}) -\strain{\bu_n}) : \bS_h  \dd \bx .
\end{align*}
Choosing $\bS_h = (  \strain{\bu_{n,h}} - \Pi_h(\strain{\bu_n}))  |  \strain{\bu_{n,h}} - \Pi_h(\strain{\bu_n})  |^{n-2} \in \mathbb M_{n,h}$ yields
$$
\|  \strain{\bu_{n,h}} - \Pi_h(\strain{\bu_n})) \|_{L_n(\Omega)} \rightarrow 0 \quad \textrm{as }h\to 0_+.
$$
It remains to employ a density argument to deduce the strong convergence result \eqref{e:u-H1-strong}.

The final claim in the statement of the lemma follows from the strong convergence results
\eqref{e:strongcon1}, \eqref{e:strongcon2}$_2$, \eqref{e:conv_strain}, \eqref{e:T-Lp-strong},
and \eqref{e:u-Lp-strong}$_2$, which together imply that, for any $\Omega_0 \subset \subset \Omega$,
\[ \lim_{n \rightarrow \infty} \lim_{h \rightarrow 0_+} \| \bT_{n,h} - \bT \|_{L_1(\Omega_0)} = 0,\]
\[ \lim_{n \rightarrow \infty} \lim_{h \rightarrow 0_+} \| \bu_{n,h} - \bu \|_{C(\overline\Omega)} = 0,\]
and
$$
\lim_{n \rightarrow \infty} \lim_{h \rightarrow 0_+} \| \strain{\bu_{n,h} - \bu}\|_{L_p(\Omega_0)} =0
\qquad \forall\, p \in [1,\infty).
$$


The assertions concerning the uniqueness of $\bu$ and $\bT$ follow from Theorem \ref{t:existence} (c).
\end{proof}

The hypotheses \eqref{e:cond_lambda2bis} and \eqref{e:cond_mu2} adopted in the statement
of Lemma \ref{l:strong} guarantee that the derivatives of the functions $s \in \mathbb{R}
\mapsto \lambda(s) s$ and $s \in \mathbb{R}_{\geq 0} \mapsto \mu(s) s$ are bounded below by $0$
on $\mathbb{R}$ and $\mathbb{R}_{>0}$, respectively.
These two functions are, in fact, Lipschitz-continuous on any compact subinterval of $\mathbb{R}$
and $\mathbb{R}_{\geq 0}$, respectively. If they are assumed to be globally H\"older-continuous on $\mathbb{R}$
and $\mathbb{R}_{\geq 0}$, respectively, with H\"older exponent $\beta \in (0,1]$, then an error inequality holds, for all $n \in \mathbb{N}$, in the limit of $h \rightarrow 0_+$, as we shall now show.

\begin{thm}
\label{thm:err.inequ}
In addition to the assumptions of Lemma \ref{l:strong}, let us also suppose that the functions $s \in \mathbb{R} \mapsto \lambda(s)s \in \mathbb{R}$
and $\bS \in \mathbb{R}^{d \times d}_{\rm sym} \mapsto \mu(|\bS|)\bS \in \mathbb{R}^{d \times d}_{\rm sym}$ are H\"older-continuous with exponent $\beta \in (0,1]$,
i.e.,  there exists a positive constant $\Lambda$ such that
\begin{equation}
\label{e:lipsch}
|\lambda(r)r-\lambda(s)s| \leq \Lambda |r-s|^\beta \quad\forall\, r,s \in \mathbb{R}, \quad |\mu(|\bR|)\bR-\mu(|\bS|)\bS| \leq \Lambda |\bR-\bS|^\beta\quad \forall\, \bR,\bS \in \mathbb{R}^{d \times d}_{\rm sym}.
\end{equation}
Then, assuming that $\bT_n \in L_\infty(\Omega)^{d \times d}_{\rm sym}$ for $n \geq 2$, the following error bound holds:
\begin{alignat}{2}
\label{e:e.inequ}
\begin{aligned}
& \int_\Omega \Phi_n(|\bT_{n,h} - \Pi_h \bT_n|) \dd \bx \\
&  \qquad \leq C(d,\Lambda,\beta,n,{\rm K}_n,{\rm K})\left(\inf_{\bv_h \in \mathbb{X}_{n,h}}\int_\Omega \Phi_n^*(|\strain{\bv_h-\bu_n}|)\dd \bx + \int_\Omega \Phi_n^*(|\bT_n - \Pi_h \bT_n|^{\min(\beta,\frac{1}{n})})\dd \bx\right).
\end{aligned}
\end{alignat}
Here,
\[ \Phi_n(s):= \frac{s^2}{(1+s)^{1-\frac{1}{n}}},\qquad s \in [0,\infty),\quad n \in \mathbb{N},\]
$\Phi_n^*$, defined by $\Phi_n^*(s):= \sup_{t \geq 0} (st - \Phi_n(t))$ for $s \in [0,\infty)$, is the convex conjugate of the function $\Phi_n$, ${\rm K}_n:=\max(1,\|\bT_n\|_{L_\infty(\Omega)})$, and ${\rm K}={\rm K}(n)$ is a positive constant that will be specified in the proof.
When $n=1$ the inequality \eqref{e:e.inequ} holds without the additional assumption that $\bT_n \in L_\infty(\Omega)^{d \times d}_{\rm sym}$.
\end{thm}

\begin{proof}
We proceed similarly as in the proof of Lemma~\ref{l:strong}.
From the relations \eqref{e:weak_reg} and \eqref{e:discrete_reg} we have, for all $\bS_h \in \mathbb M_{n,h}$, that
\begin{equation*}
\begin{split}
\int_\Omega (\mathcal A_n(\bT_{n,h}) - \mathcal A_n(\Pi_h \bT_n)):\bS_h \dd \bx
&= \int_\Omega \strain{\bu_{n,h} - \bu_n}:\bS_h \dd \bx \\
&\quad + \int_\Omega (\mathcal A_n (\bT_n) - \mathcal A_n(\Pi_h \bT_n)):\bS_h \dd \bx.
 \end{split}
 \end{equation*}
The choice $\bS_h = \bT_{n,h} - \Pi_h \bT_n \in \mathbb V_{n,h}$ guarantees that
$$
\int_\Omega \strain{\bv_h}: \bS_h \dd \bx = 0 \qquad \forall\, \bv_h \in \mathbb X_{n,h}.
$$
Thus by defining, for any $\bv_h$ in $\mathbb X_{n,h}$,
\[ \bU_{n,h}:= \strain{\bv_h-\bu_n} + \big(\mathcal A_n (\bT_n) - \mathcal A_n(\Pi_h \bT_n)\big),\]
and proceeding similarly as in the proof of the inequality \eqref{e:strong_conv1bis}, we have that
\begin{equation}\label{e:T-bound}
\begin{split}
&\frac{1}{n^2}\int_\Omega \frac{|\tr{\bT_{n,h}} - \tr{\Pi_h \bT_n}|^2}
{(|\tr{\Pi_h \bT_n}| + |\tr{\bT_{n,h}} - \tr{\Pi_h \bT_n}|)^{1-\frac{1}{n}}} \dd \bx
\\
&\qquad \qquad+ \frac{1}{n^2}\int_\Omega \frac{|\bT_{n,h}^\bd - \Pi_h \bT_n^\bd|^2}
{(|\Pi_h \bT_n^\bd| + |\bT_{n,h}^\bd - \Pi_h \bT_n^\bd|)^{1-\frac{1}{n}}}
\dd \bx
\\
&  \qquad \qquad\qquad \qquad \leq  \int_\Omega \bU_{n,h}:(\bT_{n,h}-\Pi_h \bT_n) \dd \bx.
\end{split}
\end{equation}

\noindent
Thanks to the equality \eqref{e:normT0},
\[|\Pi_h \bT^{\bd}_n(\bx) |^2 +  \frac 1 d |\tr{\Pi_h\bT_n}(\bx)|^2 = | \Pi_h \bT_n(\bx) |^2 \leq {\|\Pi_h\bT_n\|^2_{L_\infty(\Omega)}} \leq \|\bT_n\|^2_{L_\infty(\Omega)}\qquad \mbox{for a.e. $\bx \in \Omega$}.\]
Thus, by denoting ${\rm K}_n:=\max(1,\|\bT_n\|_{L_\infty(\Omega)})$, it follows that
\[ \|\Pi_h \bT^{\bd}_n \|_{L_\infty(\Omega)} \leq {\rm K}_n\qquad\mbox{and}\qquad \|\tr{\Pi_h\bT_n}\|_{L_\infty(\Omega)}
\leq d^{\frac{1}{2}} {\rm K}_n.\]
Hence we have from the inequality \eqref{e:T-bound} that
\begin{align*}
&\frac{1}{n^2}\int_\Omega \frac{|\tr{\bT_{n,h}} - \tr{\Pi_h \bT_n}|^2}
{(d^{\frac{1}{2}}{\rm K}_n + |\tr{\bT_{n,h}} - \tr{\Pi_h \bT_n}|)^{1-\frac{1}{n}}} \dd \bx
+ \frac{1}{n^2}\int_\Omega \frac{|\bT_{n,h}^\bd - \Pi_h \bT_n^\bd|^2}
{({\rm K}_n  + |\bT_{n,h}^\bd - \Pi_h \bT_n^\bd|)^{1-\frac{1}{n}}}
\dd \bx
\\
&  \qquad \leq  \int_\Omega \bU_{n,h}:(\bT_{n,h}-\Pi_h \bT_n) \dd \bx.
\end{align*}

\noindent
Because ${\rm K}_n \geq 1$ and by noting the decomposition $\bT = \frac{1}{d} \tr{\bT} \bI + \bT^\bd$, the above inequality implies that
\begin{align}
\begin{aligned}
&\hspace{-4mm}\frac{1}{n^2}(d^{\frac{1}{2}}{\rm K}_n)^{\frac{1}{n}-1} \int_\Omega \frac{|\tr{\bT_{n,h}} - \tr{\Pi_h \bT_n}|^2}
{(1 + |\tr{\bT_{n,h}} - \tr{\Pi_h \bT_n}|)^{1-\frac{1}{n}}} \dd \bx
\\
& \qquad \qquad + \frac{1}{n^2}({\rm K}_n)^{\frac{1}{n}-1} \int_\Omega \frac{|\bT_{n,h}^\bd - \Pi_h \bT_n^\bd|^2}
{(1  + |\bT_{n,h}^\bd - \Pi_h \bT_n^\bd|)^{1-\frac{1}{n}}}
\dd \bx
\\
&  \quad \leq
 \frac{1}{d} \int_\Omega \tr{\bU_{n,h}}\,\tr{\bT_{n,h}-\Pi_h \bT_n} \dd \bx  +  \int_\Omega \bU_{n,h}^\bd :(\bT_{n,h}-\Pi_h \bT_n)^\bd \dd \bx \label{eq:1sterbd}\\
&  \quad \leq
 \frac{1}{d} \int_\Omega |\tr{\bU_{n,h}}|\,|\tr{\bT_{n,h}}-\tr{\Pi_h \bT_n}| \dd \bx
+ \int_\Omega |\bU_{n,h}^\bd|\,|\bT_{n,h}^\bd-\Pi_h \bT_n^\bd| \dd \bx.
\end{aligned}
\end{align}

Let us consider the function $\Phi_n\,: \mathbb{R} \rightarrow \mathbb{R}_{\geq 0}$ defined by
\begin{equation}
\label{eq:modular}
 \Phi_n(s) := \frac{s^2}{(1 + |s|)^{1-\frac{1}{n}}},\qquad n \in \mathbb{N}.
 \end{equation}
The values of $s$ of interest to us below will always be in the range $[0,\infty)$, and therefore the absolute
value sign appearing in the denominator of $\Phi_n(s)$ can be ignored for such $s$.

Clearly, $\Phi_n(0)=0$, $\Phi_n$ is even, continuous, strictly monotonic increasing for $s \geq 0$, and convex, with
\begin{align}\label{e:asymp0}
\Phi_n(s) \asymp s^2\quad \mbox{as $s \rightarrow 0_+$}\qquad\mbox{and}\qquad \Phi_n(s) \asymp s^{1+\frac{1}{n}}\quad
\mbox{as $s \rightarrow +\infty$}.
\end{align}
Here $A \asymp B$ means that there exist constants $c$ and $\tilde c$ independent of $A$ and $B$ such that $cB \leq A \leq \tilde c B$. 
Following Rao \& Ren \cite{RR}, a function $\Phi : \mathbb{R} \rightarrow \mathbb{R}_{\geq 0}$ is called an \textit{N-function} (nice Young function), if: (i) $\Phi$ is even and convex; (ii) $\Phi(s)=0$ if, and only if, $s=0$; and (iii) $\lim_{s \rightarrow 0} \Phi(s)/s = 0$ and $\lim_{s \rightarrow +\infty} \Phi(s)/s = +\infty$.

Hence, $\Phi_n$ is an N-function.
Simple calculations show that
\begin{align}\label{e:P}
 \Phi_n(2s) \leq 4 \Phi_n(s)\quad \forall\, s \in [0,\infty)\quad\mbox{and}\quad
\frac{1}{2c} \Phi_n(cs) \geq \Phi_n(s)\quad \forall\, s \in [0,\infty), \quad \forall\, c \geq 2^n;
\end{align}
therefore $\Phi_n$ satisfies the $\Delta_2$ and $\nabla_2$ conditions on $[0,\infty)$ (cf. Definition 1 on p.2 of \cite{RR}). Now, let $\Phi_n^*$ denote the convex conjugate of the function $\Phi_n$.
Then, $(\Phi_n,\Phi_n^*)$ is a pair of complementary N-functions and, by Theorem 2 on p.3 in \cite{RR}, $\Phi_n^*$ also satisfies the $\Delta_2$ and $\nabla_2$ conditions on $[0,\infty)$; i.e., there exists
a constant ${\rm K}={\rm K}(n)>2$ such that
\begin{align}\label{e:Pstar}
\Phi_n^*(2s) \leq {\rm K} \Phi_n^*(s) \qquad \forall\, s \in [0,\infty),
\end{align}
and there exists a constant $c=c(n)>1$ such that

\[ \frac{1}{2c} \Phi_n^*(cs) \geq \Phi_n^*(s)\qquad \forall\, s \in [0,\infty).\]
More precisely, by the inequality \eqref{e:a-ineq},
\[2^{\frac{1}{n}-1}\min\big(s^2,s^{1+\frac{1}{n}}\big) \leq  \Phi_n(s) \leq \min\big(s^2,s(1+s)^{\frac{1}{n}}\big)\qquad \forall\, s \in [0,\infty).\]
By recalling that $\Phi_n^*(s):= \sup_{t \geq 0} (st - \Phi_n(t))$, we get from \eqref{e:asymp0} that
\begin{align}\label{e:asymp}
\Phi^*_n(s) \asymp \frac{1}{4} s^2\quad \mbox{as $s \rightarrow 0_+$}\qquad\mbox{and}\qquad \Phi^*_n(s) \asymp \frac{s^{n+1}}{n+1} \left(\frac{n}{n+1}\right)^n\quad
\mbox{as $s \rightarrow +\infty$}.
\end{align}
Therefore, there exist positive constants $c_{1,n}$ and $c_{2,n}$, with $c_{1,n} \leq c_{2,n}$, such that
\[0 \leq \Phi^*_n(s) \leq c_{1,n} s^2\qquad \mbox{for all $s \in [0,1]$}\]
and
\[c_{1,n} \leq \Phi^*_n(s) \leq c_{2,n} s^{n+1}\qquad \mbox{for all $s \in [1,\infty)$}.\]
Reverting to \eqref{eq:1sterbd}, by the Fenchel--Young inequality, for any real number $\delta>0$,
\begin{align*}
&\frac{1}{n^2}(d^{\frac{1}{2}}{\rm K}_n)^{\frac{1}{n}-1} \int_\Omega \Phi_n(|\tr{\bT_{n,h}} - \tr{\Pi_h \bT_n}|) \dd \bx
+ \frac{1}{n^2}({\rm K}_n)^{\frac{1}{n}-1} \int_\Omega \Phi_n(|\bT_{n,h}^\bd - \Pi_h \bT_n^\bd|)
\dd \bx
\\
&  \quad \leq   \frac{1}{d} \int_\Omega |\tr{\bU_{n,h}}|\,|\tr{\bT_{n,h}}-\tr{\Pi_h \bT_n}| \dd \bx +
\int_\Omega |\bU_{n,h}^\bd|\,|\bT_{n,h}^\bd-\Pi_h \bT_n^\bd| \dd \bx
\\
&  \quad =  \frac{1}{d\delta} \int_\Omega |\tr{\bU_{n,h}}|\,\delta\, |\tr{\bT_{n,h}}-\tr{\Pi_h \bT_n}| \dd \bx +
\frac{1}{\delta} \int_\Omega |\bU_{n,h}^\bd|\,\delta\,|\bT_{n,h}^\bd-\Pi_h \bT_n^\bd| \dd \bx
\\
&  \quad \leq \frac{1}{d\delta} \int_\Omega \Phi_n(\delta |\tr{\bT_{n,h}}-\tr{\Pi_h \bT_n}|) \dd \bx
+ \frac{1}{d\delta} \int_\Omega \Phi_n^*(|\tr{\bU_{n,h}}|) \dd \bx\\
& \qquad + \frac{1}{\delta} \int_\Omega \Phi_n(\delta |\bT_{n,h}^\bd-\Pi_h \bT_n^\bd|)  \dd \bx +
 \frac{1}{\delta} \int_\Omega \Phi_n^*(|\bU_{n,h}^\bd|) \dd \bx.
\end{align*}
Clearly, for any $a \in \mathbb{R}_{\geq 0}$ and $\delta \in (0,1]$, we have that
\[ \Phi_n(\delta a) = \frac{\delta^2 a^2}{(1 + \delta a)^{1-\frac{1}{n}}} \leq \delta^{1+ \frac{1}{n}}\Phi_n(a).\]
Hence,
\begin{align*}
&\frac{1}{n^2}(d^{\frac{1}{2}}{\rm K}_n)^{\frac{1}{n}-1} \int_\Omega \Phi_n(|\tr{\bT_{n,h}} - \tr{\Pi_h \bT_n}|) \dd \bx
+ \frac{1}{n^2}({\rm K}_n)^{\frac{1}{n}-1} \int_\Omega \Phi_n(|\bT_{n,h}^\bd - \Pi_h \bT_n^\bd|)
\dd \bx
\\
&  \quad \leq \frac{\delta^{\frac{1}{n}}}{d} \int_\Omega \Phi_n(|\tr{\bT_{n,h}}-\tr{\Pi_h \bT_n}|) \dd \bx
+ \frac{1}{d\delta} \int_\Omega \Phi_n^*(|\tr{\bU_{n,h}}|) \dd \bx\\
&\qquad + \delta^{\frac{1}{n}} \int_\Omega \Phi_n(|\bT_{n,h}^\bd-\Pi_h \bT_n^\bd|)  \dd \bx +
 \frac{1}{\delta} \int_\Omega \Phi_n^*(|\bU_{n,h}^\bd|) \dd \bx.
\end{align*}

\noindent
Let $\delta_1, \delta_2 > 0$ be such that
\[ \frac{d}{2n^2}(d^{\frac{1}{2}}{\rm K}_n)^{\frac{1}{n}-1} =  \delta_1^{\frac{1}{n}}\quad \mbox{and}\quad
\frac{1}{2n^2}({\rm K}_n)^{\frac{1}{n}-1} =  \delta_2^{\frac{1}{n}}
.\]
Thus, with $\delta:=\min(1,\delta_1,\delta_2)$, we have that
\begin{align}\label{e:delta}
\begin{aligned}
& \int_\Omega \Phi_n(|\tr{\bT_{n,h}} - \tr{\Pi_h \bT_n}|) \dd \bx
+  \int_\Omega \Phi_n(|\bT_{n,h}^\bd - \Pi_h \bT_n^\bd|)
\dd \bx
\\
&  \qquad \leq C(d,n,{\rm K}_n) \left(\int_\Omega \Phi_n^*(|\tr{\bU_{n,h}}|) \dd \bx +
\int_\Omega \Phi_n^*(|\bU_{n,h}^\bd|) \dd \bx\right).
\end{aligned}
\end{align}
Now, the assumption \eqref{e:lipsch} and \eqref{e:holder} yield
\begin{align*}
|\bU_{n,h}| &\leq |\strain{\bv_h-\bu_n}| + |\mathcal A_n (\bT_n) - \mathcal A_n(\Pi_h \bT_n)|\\
& \leq |\strain{\bv_h-\bu_n}| + C\left(|\bT_n - \Pi_h \bT_n|^{\frac{1}{n}} + |\bT_n - \Pi_h \bT_n|^{\beta}\right).
\end{align*}
As $\Phi_n^*$ is an N-function, it is strictly monotonic increasing (cf. the top of p.2 in \cite{RR}) and
convex, and therefore by \eqref{e:Pstar},
\begin{align}\label{e:Ubound}
\begin{aligned}
\Phi_n^*(|\bU_{n,h}|) &\leq \Phi_n^*(|\strain{\bv_h-\bu_n}| + |\mathcal A_n (\bT_n) - \mathcal A_n(\Pi_h \bT_n)|)
\\
& \leq \frac{1}{2}\bigg(\Phi_n^*(2|\strain{\bv_h-\bu_n}|) + \Phi_n^*(2|\mathcal A_n (\bT_n) - \mathcal A_n(\Pi_h \bT_n)|)\bigg)\\
& \leq \frac{{\rm K}}{2} \bigg(\Phi_n^*(|\strain{\bv_h-\bu_n}|) + \Phi_n^*(|\mathcal A_n (\bT_n) - \mathcal A_n(\Pi_h \bT_n)|)\bigg).
\end{aligned}
\end{align}
In order to proceed we need to bound the right-hand side of the last inequality and that involves comparing
\begin{equation*}
\mathcal A_{n}(\bT_n) := \lambda(\tr{\bT_n}) \tr{\bT_n} \bI + \mu( | \bT_n^{\bd}|) \bT_n^{\bd}
+ \frac{\tr{\bT_n} \bI}{n|\tr{\bT_n}|^{1-\frac{1}{n}}} + \frac{\bT_n^{\bd}}{n|\bT_n^\bd|^{1-\frac{1}{n}}}
\end{equation*}
with
\begin{equation*}
\mathcal A_{n}(\Pi_h\bT_n) := \lambda(\tr{\Pi_h\bT_n}) \tr{\Pi_h\bT_n} \bI + \mu( |\Pi_h \bT_n^{\bd}|) \Pi_h\bT_n^{\bd}
+ \frac{\tr{\Pi_h\bT_n} \bI}{n|\tr{\Pi_h\bT_n}|^{1-\frac{1}{n}}} + \frac{\Pi_h\bT_n^{\bd}}{n|\Pi_h\bT_n^\bd|^{1-\frac{1}{n}}}.
\end{equation*}
We have from inequalities \eqref{e:lipsch} and \eqref{e:holder} that
\begin{align*}
|\mathcal A_{n}(\bT_n) - \mathcal A_{n}(\Pi_h\bT_n)| &\leq d^{\frac{1}{2}}\Lambda |\tr{\bT_n} - \Pi_h \tr{\bT_n}|^\beta
+ \Lambda |\bT_n^\bd - \Pi_h \bT_n^\bd|^\beta\\
& \quad + \frac{d^{\frac{1}{2n}}}{n}\,2^{1-\frac{1}{n}} \, |\tr{\bT_n} - \tr{\Pi_h \bT_n}|^{\frac{1}{n}}\\
& \quad + \, C(d,n) |\bT_n^{\bd} - \Pi_h \bT_n^{\bd}|^{\frac{1}{n}}\\
& \leq C(d,\Lambda, \beta)\,|\bT_n - \Pi_h \bT_n|^\beta + C(d,n)\,|\bT_n - \Pi_h \bT_n|^{\frac{1}{n}}\\
& \leq C(d,\Lambda, \beta,n, {\rm K}_n)|\bT_n - \Pi_h \bT_n|^{\min(\beta,\frac{1}{n})}.
\end{align*}
By the inequality \eqref{e:Pstar}, $\Phi_n^*(2^\ell s) \leq {\rm K}^\ell \Phi_n^*(s)$ for all $s \in [0,\infty)$ and all $\ell \geq 1$. Hence, with
$$\ell:= [\log_2 C(d,\Lambda, \beta,n, {\rm K}_n)] + 1$$ we have that $C(d,\Lambda, \beta,n, {\rm K}_n)
\leq 2^\ell$, whereby
\begin{align*}
\Phi_n^*(|\mathcal A_{n}(\bT_n) - \mathcal A_{n}(\Pi_h\bT_n)|) &\leq
\Phi_n^*(C(d,\Lambda, \beta,n, {\rm K}_n)|\bT_n - \Pi_h \bT_n|^{\min(\beta,\frac{1}{n})})\\
&\leq \Phi_n^*(2^\ell|\bT_n - \Pi_h \bT_n|^{\min(\beta,\frac{1}{n})})\\
&\leq {\rm K}^\ell \Phi_n^*(|\bT_n - \Pi_h \bT_n|^{\min(\beta,\frac{1}{n})}).
\end{align*}
By substituting this into the inequality \eqref{e:Ubound} we deduce that
\begin{align*}
\Phi_n^*(|\bU_{n,h}|) &\leq \Phi_n^*(|\strain{\bv_h-\bu_n}| + |\mathcal A_n (\bT_n) - \mathcal A_n(\Pi_h \bT_n)|)
\\
& \leq \frac{1}{2} {\rm K} \Phi_n^*(|\strain{\bv_h-\bu_n}|) + \frac{1}{2}{\rm K}^{\ell+1} \Phi_n^*(|\bT_n - \Pi_h \bT_n|^{\min(\beta,\frac{1}{n})}).
\end{align*}
We then substitute this into the inequality \eqref{e:delta} and note, once again, the monotonicity of $\Phi_n^*$, which gives
\begin{align*}
& \int_\Omega \Phi_n(|\tr{\bT_{n,h}} - \tr{\Pi_h \bT_n}|) \dd \bx
+  \int_\Omega \Phi_n(|\bT_{n,h}^\bd - \Pi_h \bT_n^\bd|)
\dd \bx
\\
&  \qquad \leq C(d,\Lambda,\beta,n,{\rm K}_n,{\rm K})\left(\int_\Omega \Phi_n^*(|\strain{\bv_h-\bu_n}|)\dd \bx + \int_\Omega \Phi_n^*(|\bT_n - \Pi_h \bT_n|^{\min(\beta,\frac{1}{n})})\dd \bx\right).
\end{align*}
For any pair of numbers $a, b \in \mathbb{R}_{\geq 0}$, by \eqref{e:P} and convexity, we have $\Phi_n(a+b) \leq 2 \Phi_n(a) + 2 \Phi_n(b)$; hence, by the inequality \eqref{e:normT1},
\begin{align*}
& \int_\Omega \Phi_n(|\bT_{n,h} - \Pi_h \bT_n|) \dd \bx \\
&  \qquad \leq C(d,\Lambda,\beta,n,{\rm K}_n,{\rm K})\left(\int_\Omega \Phi_n^*(|\strain{\bv_h-\bu_n}|)\dd \bx + \int_\Omega \Phi_n^*(|\bT_n - \Pi_h \bT_n|^{\min(\beta,\frac{1}{n})})\dd \bx\right).
\end{align*}
As this inequality holds for all $\bv_h \in \mathbb{X}_{n,h}$, the bound \eqref{e:e.inequ} directly follows.
\end{proof}

The error bound \eqref{e:e.inequ} can be restated in the following equivalent form. Given an N-function $\Psi$, let
\[ \tilde{L}_{\Psi}(\Omega):= \bigg\{ \bS\,:\,\Omega \rightarrow \mathbb{R}^{d \times d}_{\rm sym}\quad \mbox{measurable, such that}\quad \rho_\Psi(\bS):=
\int_\Omega \Psi(|\bS(\bx)|) \dd \bx < \infty\bigg\};\]
the function $\rho_\Psi(\cdot)$ is called a \textit{modular}. In terms of the modulars $\rho_{\Phi_n}$ and $\rho_{\Phi_n^*}$ the error
bound \eqref{e:e.inequ} takes the form:
\begin{alignat}{2}
\label{e:e.inequ1}
\begin{aligned}
& \rho_{\Phi_n}(|\bT_{n,h} - \Pi_h \bT_n|) \\
&  \qquad \leq C(d,\Lambda,\beta,n,{\rm K}_n,{\rm K})\left(\inf_{\bv_h \in \mathbb{X}_{n,h}}\rho_{\Phi_n^*}(|\strain{\bv_h-\bu_n}|) + \rho_{\Phi_n^*}(|\bT_n - \Pi_h \bT_n|^{\min(\beta,\frac{1}{n})})\right).
\end{aligned}
\end{alignat}
Here, as before,
\[ \Phi_n(s):= \frac{s^2}{(1+s)^{1-\frac{1}{n}}},\qquad s \in [0,\infty),\quad n \in \mathbb{N},\]
and $\Phi_n^*$ is the convex conjugate of $\Phi_n$.

\smallskip

Under the above assumptions,  convergence rates can be derived by strengthening the regularity hypothesis
$\bT_n \in L_\infty(\Omega)^{d \times d}_{\rm sym}$ from Theorem \ref{thm:err.inequ}.
Thus, for example, suppose that
\[ \bT_n \in W^{r,q}(\Omega)^{d\times d}_{\rm sym} \qquad \mbox{with $1 \geq r > \frac{d}{q}$}\qquad \mbox{and}\qquad \bu_n \in W^{1+t,p}(\Omega)^d \cap W^{1,n+1}_0(\Omega)^d\qquad \mbox{with $1\geq t>\frac{d}{p}$},\]
and $q,p \in (1,\infty]$, which ensure, by Morrey's embedding theorem, that
\[ \bT_n \in W^{r,q}(\Omega)^{d\times d}_{\rm sym} \hookrightarrow \mathcal{C}^{0,\gamma}(\overline{\Omega})^{d \times d}_{\rm sym}
\hookrightarrow  W^{\gamma,\infty} (\Omega)^{d \times d}_{\rm sym} \qquad \mbox{with $\gamma := r- \frac{d}{q}$};\]
and
\[ \strain{\bu_n} \in W^{t,p}(\Omega)^{d\times d}_{\rm sym} \hookrightarrow \mathcal{C}^{0,\zeta}(\overline{\Omega})^{d \times d}_{\rm sym} \hookrightarrow  W^{\zeta,\infty}  (\Omega)^{d\times d}_{\rm sym}\qquad \mbox{with $\zeta := t - \frac{d}{p}$}.\]
With these stronger regularity hypotheses we then have that
\[|\bT_n(\bx) - \Pi_h \bT_n(\bx)|\leq C h_K^\gamma \|\bT_n\|_{W^{\gamma,\infty}(K)} \leq C h^\gamma \|\bT_n\|_{W^{\gamma,\infty}(\Omega)}\qquad \forall\, \bx \in K, \quad \forall\, K \in \mathcal{T}_h.\]
Thus, thanks to the fact that $\Phi_n^*$ is monotonic increasing, and by the first asymptotic property in \eqref{e:asymp},
\[ \rho_{\Phi_n^*}(|\bT_n - \Pi_h \bT_n|^{\min(\beta,\frac{1}{n})})  \leq \rho_{\Phi_n^*}(Ch^{\gamma\min(\beta,\frac{1}{n})}\|\bT_n\|_{W^{\gamma,\infty}(\Omega)}^{\min(\beta,\frac{1}{n})}) \asymp
Ch^{2\gamma\min(\beta,\frac{1}{n})}\|\bT_n\|_{W^{\gamma,\infty}(\Omega)}^{2\min(\beta,\frac{1}{n})}\]
as $h \rightarrow 0_+$. Analogously,
\[\inf_{\bv_h \in \mathbb{X}_{n,h}}\rho_{\Phi_n^*}(|\strain{\bv_h}-\strain{\bu_n}|) \leq
\rho_{\Phi_n^*}(Ch^{\zeta}\|\strain{\bu_n}\|_{W^{\zeta,\infty}(\Omega)}) \asymp
Ch^{2\zeta}\|\strain{\bu_n}\|_{W^{\zeta,\infty}(\Omega)}^2\qquad \mbox{as $h \rightarrow 0_+$}.
\]
By substituting these bounds into the error inequality \eqref{e:e.inequ1} we deduce that
\begin{alignat*}{2}
\rho_{\Phi_n}(|\bT_{n,h} - \Pi_h \bT_n|) \leq C\left(h^{2\zeta}\|\strain{\bu_n}\|_{W^{\zeta,\infty}(\Omega)}^2 +  h^{2\gamma\min(\beta,\frac{1}{n})}\|\bT_n\|_{W^{\gamma,\infty}(\Omega)}^{2\min(\beta,\frac{1}{n})}\right),
\end{alignat*}
as $h \rightarrow 0_+$. In particular, if $\beta = \frac{1}{n}$ and $\zeta = \frac{\gamma}{n}$,
\begin{alignat}{2}\label{e:modulbd}
\rho_{\Phi_n}(|\bT_{n,h} - \Pi_h \bT_n|) \leq Ch^{2\frac{\gamma}{n}}\left(\|\strain{\bu_n}\|_{W^{\frac{\gamma}{n},\infty}(\Omega)}^2 + \|\bT_n\|_{W^{\gamma,\infty}(\Omega)}^{\frac{2}{n}}\right),
\end{alignat}
as $h \rightarrow 0_+$, where $\gamma \in (0,1]$ and $n \in \mathbb{N}$.
The error bound \eqref{e:modulbd} on $\rho_{\Phi_n}(|\bT_{n,h} - \Pi_h \bT_n|)$ can be used to derive bounds on norms of
the error $|\bT_{n,h} - \Pi_h \bT_n|$. For example, in the special case when $n=1$, we have that $\Phi_n(s) = s^2$, and therefore
\begin{alignat*}{2}
\|\bT_{n,h} - \Pi_h \bT_n\|_{L_2(\Omega)} \leq Ch^{\gamma}\left(\|\strain{\bu_n}\|_{W^{\gamma,\infty}(\Omega)} + \|\bT_n\|_{W^{\gamma,\infty}(\Omega)}\right),
\end{alignat*}
as $h \rightarrow 0_+$, where $\gamma \in (0,1]$. In this special case, the regularity requirements on $\bu$ and $\bT$
can, in fact, be relaxed to $\bu_n \in W^{1+\gamma,2}(\Omega)^{d \times d}_{\rm sym} \cap W^{1,2}_0(\Omega)^{d \times d}_{\rm sym}$ and $\bT_n \in W^{\gamma,2}(\Omega)^{d \times d}_{\rm sym}$, $\gamma \in (0,1]$.

More generally, for $n \in \mathbb{N}$, we divide the inequality \eqref{e:modulbd} by $|\Omega|$, recall the definition of the modular $\rho_{\Phi_n}(\cdot)$, and apply Jensen's inequality on the left-hand side to deduce that
\[ \Phi_n\bigg(\Xint{-}_\Omega |\bT_{n,h} - \Pi_h \bT_n| \dd \bx\bigg) \leq Ch^{2\frac{\gamma}{n}}\left(\|\strain{\bu_n}\|_{W^{\frac{\gamma}{n},\infty}(\Omega)}^2 + \|\bT_n\|_{W^{\gamma,\infty}(\Omega)}^{\frac{2}{n}}\right),\]
as $h \rightarrow 0_+$, where $\gamma \in (0,1]$.
Because $\Phi_n^{-1}$, the inverse function of $\Phi_n$ (which is uniquely defined on $[0,\infty)$ thanks to the fact that
$\Phi_n$ is strictly monotonic increasing on $[0,\infty)$), is monotonic increasing,
we have that
\[ \Xint{-}_\Omega |\bT_{n,h} - \Pi_h \bT_n| \dd \bx \leq \Phi_n^{-1}\left(Ch^{2\frac{\gamma}{n}}\left(\|\strain{\bu_n}\|_{W^{\frac{\gamma}{n},\infty}(\Omega)}^2 + \|\bT_n\|_{W^{\gamma,\infty}(\Omega)}^{\frac{2}{n}}\right)\right),\]
as $h \rightarrow 0_+$, where $\gamma \in (0,1]$ and $n \in \mathbb{N}$. Since $\Phi_n(s) \asymp s^2$ as $s \rightarrow 0_+$,
it follows that $\Phi_n^{-1}(s) \asymp s^\frac{1}{2}$ as $s \rightarrow 0_+$, and therefore
\begin{equation}
\label{eq:finalerrbound}
 \|\bT_{n,h} - \Pi_h \bT_n\|_{L_1(\Omega)} \leq Ch^{\frac{\gamma}{n}}\left(\|\strain{\bu_n}\|_{W^{\frac{\gamma}{n},\infty}(\Omega)} + \|\bT_n\|_{W^{\gamma,\infty}(\Omega)}^{\frac{1}{n}}\right),
 \end{equation}
as $h \rightarrow 0_+$, where $\gamma \in (0,1]$, $n \in \mathbb{N}$, and $C=C(d,\Lambda, n, {\rm K}_n,{\rm K},\gamma, |\Omega|)$.


\subsection{Other elements that fit into the theory}
\label{subsec:otherfit}

 We shall comment here on some alternative choices of finite element spaces to which our analysis applies. Let $\mathcal{Q}^r_h$
denote the finite element space on quadrilateral or hexahedral meshes for $d=2$ or $d=3$, respectively, consisting of (possibly discontinuous) mapped piecewise $d$-variate functions that are polynomials of degree $r$ in each variable
over each element in the subdivision.   We consider the conforming finite element spaces
\begin{equation}
\label{e:pair_QQ}
\mathbb M_{n,h} :=  \left(\mathcal Q^{r}_h\right)^{d\times d}_{\rm sym}\subset \mathbb M_n,
\qquad \mathbb X_{n,h} := \left(\mathcal Q^{r}_h\right)^{d}  \cap \mathbb X_n \subset \mathbb X_n,
\end{equation}
for the approximation of $\bT_n$ and $\bu_n$, respectively. Clearly, $\strain{\mathbb X _{n,h}} \subset \mathbb M_{n,h}$ and
the $L_2(\Omega)^{d \times d}$ orthogonal projector $\Pi_h\,:\, \mathbb{M}_n \mapsto \mathbb{M}_{n,h}$ is stable in the $L_p(\Omega)^{d \times d}$ norm for all $p \in [1,\infty]$.\footnote{This stability result is a consequence of the stability  in the $L_p(-1,1)$ norm of the $L_2(-1,1)$ orthogonal projection onto the space of univariate polynomials of degree $r$ on the interval $(-1,1)$, for all $p \in [1,\infty]$, with a stability constant $C_{r,p} = C\cdot r^{\frac{1}{2}\left|1 - \frac{2}{p}\right|}$;
for $p=\infty$, see Gronwall \cite{Gron} eq. (29) on p.230; for $p=2$, $C_{r,2}=1$ for all $r \geq 1$; for $p \in (2,\infty)$, the form of $C_{r,p}$ follows by function space interpolation; and for $p \in [1,2)$ it follows from the result for $p=(2,\infty]$ by duality.}
Then, Lemma \ref{l:weak} can be shown to
hold by an identical argument;
if in addition it is assumed that $\bT_n \in L_\infty(\Omega)^{d \times d}_{\rm sym}$, then Lemma \ref{l:strong} and Theorem \ref{thm:err.inequ} also hold. We note that our proof of Lemma \ref{l:strong} in the special case of
\begin{equation}
\label{e:pair_PP}
\mathbb M_{n,h} = \left(\mathcal P^{0}_h\right)^{d \times d}_{\rm sym}\subset \mathbb M_n\qquad \mbox{and} \qquad \mathbb X_{n,h} = \left(\mathcal P^{1}_h\right)^{d}  \cap \mathbb X_n \subset \mathbb X_n
\end{equation}
did not require the additional assumption $\bT_n \in L_\infty(\Omega)^{d \times d}_{\rm sym}$, thanks to the connection
between the explicit formula for the projection onto piecewise constant functions and the Hardy--Littlewood maximal function.


\subsection{A simple quadrilateral/hexahedral element to which the theory does not apply}
\label{subsec:notheory}
The simplest extension to quadrilaterals or hexahedra of the spaces defined in \eqref{e:P0P1} is of course
\begin{equation}
\label{e:Q1--P0}
\mathbb M_{n,h} :=  \left(\mathcal Q^{0}_h\right)^{d\times d}_{\rm sym}\subset \mathbb M_n,
\qquad \mathbb X_{n,h} := \left(\mathcal Q^{1}_h\right)^{d}  \cap \mathbb X_n \subset \mathbb X_n,
\end{equation}
for the approximation of $\bT_n$ and $\bu_n$, respectively. Everything done previously applies to this pair of elements, {\em except the uniform discrete inf-sup condition}. Indeed the proof of Lemma \ref{l:discrete_infsup} does not carry over to this case because  $\strain{\mathbb X _{n,h}}$ is not contained in $\mathbb M_{n,h}$.

Let us look more closely at the greatest lower bound in \eqref{e:LBBh}, say $\beta_h$. First, for any given $\bv_{h}$, the choice in each element $K$ (which generalizes \eqref{e:PIh})
\begin{equation}
\label{e:bTh}
\bT_{h} = \frac{1}{|K|^n}\bigg(\int_K \strain{\bv_{h}}\dd\bx\bigg)\, \bigg|\int_K \strain{\bv_{h}}\dd\bx\bigg|^{n-1},
\end{equation}
shows that $\beta_h \ge 0$. The next lemma shows that on a structured mesh (i.e.,  a mesh with a Cartesian numbering), $\beta_h \ne 0$. To avoid excessive technicalities, it is stated for quadrilaterals, but it extends to structured hexahedral meshes.

\begin{prop}
\label{pro:beta>0}
Let $\mathcal T_h$ be a structured quadrilateral mesh. Then, the greatest lower bound $\beta_h$ in \eqref{e:LBBh} is strictly positive.
\end{prop}

\begin{proof}
We argue by contradiction. Suppose that $\beta_h =0$. Then there is a displacement $\bv_h$ in $\mathbb X _{n,h}$ such that
$$\sup_{\bS_{h} \in \mathbb M_{n,h}} b(\bS_{h},\bv_h) = 0.
$$
In particular $b(\bT_{h},\bv_h) = 0$ for $\bT_{h}$ defined by \eqref{e:bTh}. This implies that
\begin{equation}
\label{e:Thvh=0}
\bigg|\int_K \strain{\bv_{h}}\dd\bx \bigg| = 0\qquad \forall\, K \in \mathcal T_h.
\end{equation}
Let us examine the consequences of \eqref{e:Thvh=0} on specific elements $K$ of the mesh. Let $\hat K = [0,1]^2$ be the reference square with vertices $\hat\ba_1 = (0,0)$, $\hat\ba_2 = (1,0)$, $\hat\ba_3 = (1,1)$, $\hat\ba_4 = (0,1)$. Let $\ba_i, 1\le i \le 4$ denote the vertices of $K$ and $\mathcal F_K$ the bilinear mapping from $\hat K$ onto $K$ that maps $\hat\ba_i$ to $\ba_i$, $1\le i \le 4$. Since the mesh is assumed to be nondegenerate, $\mathcal F_K$ is invertible and the functions of $\mathcal Q^{1}_h$ are the images by $\mathcal F_K^{-1}$ of the functions of $\hat Q^{1}$ defined on $\hat K$. Their derivatives are transformed as follows:
\begin{equation}
\nonumber
\begin{split}
\frac{\partial v}{ \partial x_1}\circ {\mathcal F}_K =& \frac{1}{\mathcal J_K}\Big(\frac{\partial
\hat v}{\partial \hat x_1}\big(a^4_2-a^1_2+ \hat x_1 (a^3_2-a^2_2-a^4_2+a^1_2)\big)
-\frac{\partial \hat v}{\partial \hat x_2}\big(a^2_2-a^1_2+ \hat x_2 (a^3_2-a^2_2-a^4_2+a^1_2)\big)
\Big),\\
\frac{\partial v}{\partial x_2}\circ {\mathcal F}_K =& \frac{1}{\mathcal J_K}\Big(\frac{\partial
\hat v}{\partial \hat x_2}\big(a^2_1-a^1_1+ \hat x_2 (a^3_1-a^2_1-a^4_1+a^1_1)\big)
-\frac{\partial \hat v}{\partial \hat x_1}\big(a^4_1-a^1_1+ \hat x_1 (a^3_1-a^2_1-a^4_1+a^1_1)\big)
\Big),
\end{split}
\end{equation}
the subscript indicating the coordinate, and ${\mathcal J}_K$ the Jacobian of ${\mathcal F}_K$.

Now, let us start with a corner element; since the mesh is structured, all corner elements have at least two sides and three vertices on the boundary, say $\ba_1$, $\ba_2$, and $\ba_4$. As $\bv_h$ vanishes on $\partial \Omega$, this means that $\bv_h(\ba_1) = \bv_h(\ba_2)=\bv_h(\ba_4) = {\bf 0}$ and thus
\begin{equation}
\nonumber
\begin{split}
\Big(\int_K \strain{\bv_{h}}\dd\bx\Big) : \Big(\int_K \strain{\bv_{h}}\dd\bx\Big) =  \frac{1}{4}\Big[&(\hat v_1^3)^2 (a_2^4-a_2^2)^2 + (\hat v_2^3)^2 (a_1^4-a_1^2)^2\Big]\\
& + \frac{1}{4}\Big[\frac{1}{4}\big( \hat v_2^3 (a_2^4-a_2^2) + \hat v_1^3 (a_1^4-a_1^2)\big)^2\Big] =0.
\end{split}
\end{equation}
As $|\ba_4-\ba_2| >0$, we easily derive from this expression that $\bv_h(\ba_3) ={\bf 0}$, and hence $\bv_h$ vanishes on $K$. This implies that $\bv_h$ also vanishes at its neighbors adjacent to the boundary, and by progressing element by element along the boundary, we have that $\bv_h ={\bf 0}$ on all boundary elements. From here, the same argument gives $\bv_h ={\bf 0}$ on all elements of $\mathcal T_h$.
\end{proof}

The positivity of $\beta_h$ implies that \eqref{e:LBBh} holds with a positive constant for each $h$, but does not guarantee that the positive constant is uniformly bounded away from zero as $h$ tends to zero. Let us give an example when $\beta_h$ tends to zero, inspired by the checkerboard modes of the Stokes problem; see~\cite{ref:GiR}. The idea is to construct a displacement $\bv_h$ such that the integral average of $\strain{\bv_{h}}$ vanishes on a large number of elements, while $\strain{\bv_{h}}$ is nonzero there. Consider a square domain $\Omega = (0,1)^2$ divided into $(N+1)^2$ equal squares $K_{ij}$, $0 \le i,j \le N$, with
mesh-size $h =\frac{1}{N+1}$.
Take $\bv_h ={\bf 0}$ on $\partial \Omega$ and define each component $\bv_h$ by
$$
v_h(\bx_{ij}) = \left\{\begin{array}{rl}
        1,  &\text{if}\  i+j\ \text{is odd} \\
        -1,  &\text{if}\  i+j\ \text{is even}.
       \end{array}\right.\qquad \mbox{for $1 \le i,j \le N$}.
$$
It is easy to check that, in all interior elements $K$,
$$\int_K \strain{\bv_{h}}\dd\bx ={\bf 0},
$$
and in each boundary element $K$,
$$0 < c_1 h \le \Big|\int_K \strain{\bv_{h}}\dd\bx \Big|\le C_1 h,
$$
where here and below all constants are independent of $K$ and $h$. Let ${\mathcal T}_h^b$ denote the union of the boundary elements.
Since the choice of $\bT_{h}$ in all interior elements does not affect the value of $b(\bT_{h},\bv_h)$, let us choose $\bT_{h} ={\bf 0}$ in these elements; this will minimize its norm there. On the boundary elements $K$, we choose $\bT_{h}$ by \eqref{e:bTh}; this gives
$$b(\bT_{h},\bv_{h}) = \sum_{K \in {\mathcal T}_h^b}\frac{1}{|K|^n}\,\bigg|\int_K \strain{\bv_{h}}\dd\bx\bigg|^{n+1},
$$
and
$$\|\bT_{h}\|_{L_{1+\frac{1}{n}}(\Omega)}  =\bigg( \sum_{K \in {\mathcal T}_h^b} \frac{1}{|K|^n} \bigg|\int_K \strain{\bv_{h}}\dd\bx \bigg|^{n+1}\bigg)^{\frac{n}{n+1}},
$$
so that
$$\frac{b(\bT_{h},\bv_{h})}{\|\bT_{h}\|_{L_{1+\frac{1}{n}}(\Omega)}}  \le C_2 h^{-\frac{n}{n+1}}.
$$
On the other hand, $\strain{\bv_{h}}$ does not vanish in the interior elements, and we have
$$ \|\strain{\bv_{h}}\|_{L_{1+n}(\Omega)} \ge C_3 h^{-1}.
$$
Hence with this choice of $\bT_h$,
\begin{equation}
\label{eq:checker}
\inf_{\bv_h \in \mathbb X_{n,h}}\frac{b(\bT_h,\bv_h)}{\| \bT_h \|_{L_{1+\frac{1}{n}}(\Omega)} \| \strain{\bv_h} \|_{L_{n+1}(\Omega)}} \le C_4 h^{\frac{1}{n+1}}.
\end{equation}
Of course, we have not proved that this choice of $\bT_h$ realizes the supremum in \eqref{eq:checker}. But since the number of interior elements, that do not contribute to the numerator of \eqref{eq:checker} but do contribute to the norm of $\bv_h$, is much larger than that of the boundary elements, more precisely, this ratio is of the order of $h^{-1}$, no value of $\bT_h$ can balance this ratio.


\section{The case of smoother data}\label{s:smoother}

The regularization \eqref{e:reg_T} is a particular case of
\begin{equation}
\label{e:gen_reg}
\strain{\bu} = \lambda( \tr{\bT} ) \tr{\bT} \bI + \mu(|\bT^{\bd}|)\bT^{\bd} +
\frac{\tr{\bT} \bI}{n|\tr{\bT}|^{1-\frac{1}{t}}} + \frac{\bT^{\mathbf d}}{n|\bT^{\mathbf d}|^{1-\frac{1}{t}}},
\end{equation}
$n \in \mathbb{N}$, $t \in \mathbb{R}_{>0}$, with $t=n$ in \eqref{e:reg_T}. When the data are smoother, as in part (d) of Theorem \ref{t:existence}, the following simpler regularization is used in reference~\cite{BMRS14}
\begin{equation}
\label{e:strain,m=1}
\strain{\bu} = \lambda( \tr{\bT} ) \tr{\bT} \bI + \mu(|\bT^{\bd}|)\bT^{\bd} +\frac{1}{n}\bT ,
 \end{equation}
which corresponds to $t=1$ (up to the factor $\frac{1}{d}$ multiplying $\bT^{\bd}$). The analysis developed in the previous sections applies to \eqref{e:div}--\eqref{e:strain,m=1} but is in fact much simpler. Indeed, let $(\bT_{n,1},\bu_{n,1})$ denote a solution to \eqref{e:div}--\eqref{e:strain,m=1}, i.e.,  $(\bT_{n,1},\bu_{n,1})\in \mathbb M_{n,1} \times \mathbb X_{n,1}$ satisfies
\begin{alignat}{2}\label{e:weak_reg,m=1}
\begin{aligned}
\quad a_{n,1}(\bT_{n,1},\bS)+c(\bT_{n,1};\bT_{n,1},\bS)-b(\bS,\bu_{n,1}) &=0 \qquad &&\forall\, \bS \in\mathbb M_{n,1},\\
\quad b(\bT_{n,1},\bv) &= \int_\Omega \bef \cdot \bv \dd\bx \qquad &&\forall\, \bv \in \mathbb X_{n,1},
\end{aligned}
\end{alignat}
where
$$a_{n,1}(\bT,\bS):= \frac{1}{n}\int_\Omega \bT: \bS \dd \bx,
$$
and
$$
\mathbb M_{n,1}:= L_{2}(\Omega)^{d\times d}_\textrm{sym}, \qquad \mathbb X_{n,1}:= H^{1}_0(\Omega)^d.
$$
The function $\bF$ is used in deriving more regularity of the solution, but as far as the numerical scheme is concerned, we can simply proceed with the original data $\bef$.  Let us briefly sketch the analysis of \eqref{e:weak_reg,m=1}. We define the mapping $\mathcal A_{n,1}: L_{2}(\Omega)^{d\times d}_\textrm{sym} \rightarrow L_{2}(\Omega)^{d\times d}_\textrm{sym}$ by
\begin{equation}\label{e:def_A1}
\mathcal A_{n,1}(\bS) := \lambda(\tr{\bS}) \tr{\bS} \bI + \mu( | \bS^{\bd}|) \bS^{\bd}
+ \frac{1}{n}\bS,
\end{equation}
and we easily prove as in Lemma \ref{l:A_cont} that $\mathcal A_{n,1}$ is bounded, continuous and coercive for all $n \in \mathbb{N}$. The inf-sup condition is satisfied, as in Lemma \ref{l:infsup},
\begin{equation}
\label{e:LBB1}
\inf_{\bv \in \mathbb X_{n,1}} \,\sup_{\bS \in \mathbb M_{n,1}} \, \frac{b(\bS,\bv)}{\| \bS \|_{L_{2}(\Omega)} \| \strain{\bv} \|_{L_{2}(\Omega)}} \geq 1.
\end{equation}
The lifting $\bT^{\mathbf{f}}_{n,1}$ is defined by the analogue of \eqref{e:constraint}
\begin{equation}\label{e:constraint1}
\int_\Omega \bT^{\mathbf{f}}_{n,1} : \strain{\bv} \dd \bx = \int_{\Omega} \bef \cdot \bv \dd \bx\qquad \forall\, \bv \in \mathbb X_{n,1},
\end{equation}
and is bounded by
\begin{equation}\label{e:bound_Tf1}
\| \bT^{\mathbf{f}}_{n,1} \|_{L_{2}(\Omega)} \le C_K\| \bef \|_{L_{2}(\Omega)},
\end{equation}
where $C_K$ is the constant of \eqref{e:korn} with $p=2$. The a priori estimates of Lemma \ref{l:apriori} simplify, we have
\begin{equation}
\label{e:aprioriun1}
\| \strain{\bu_{n,1}} \|_{L_2(\Omega)}^2 \leq  \frac{4}{n^2} C_K^2 \| \bef \|_{L_2(\Omega)}^2 + \frac{8}{n}C_1 \kappa  |\Omega|  + 8C_2^2 d | \Omega|,
\end{equation}
\begin{equation}
\label{e:aprioriTn1}
\frac{1}{n} \| \bT_{n,1} \|_{L_2(\Omega)}^2 + C_1 \| \bT_{n,1} \|_{L_1(\Omega)}
\leq 2 C_1 \kappa | \Omega
| + C_K \| \bef \|_{L_2(\Omega)}
\left( \frac{4}{n^2} C_K^2 \| \bef \|_{L_2(\Omega)}^2 + \frac{8}{n}C_1 \kappa  |\Omega|  + 8C_2^2 d | \Omega| \right)^{\frac{1}{2}}.
\end{equation}
Thus, up to a subsequence, $\bu_{n,1}$ converges weakly in $W^{1,2}_0(\Omega)^d$, and thanks to the results in~\cite{BMRS14}
(see also part (d) of Theorem \ref{t:existence} and Remark \ref{rem:rangeq}), the additional regularity $\bF \in W^{2,2}(\Omega)^{d\times d}_{\rm sym}$ enables one to prove in particular that $\bT_{n,1}$
is bounded in $W^{1,q}(\Omega_0)^{d \times d}_{\rm sym}$ for any $\Omega_0 \subset \subset \Omega$, with $q \in [1,2)$ when $d=2$ and $q \in [1,\frac{3}{2}]$ when $d=3$, and therefore, up to a subsequence, weakly converges to $\bT$ in $W^{1,q}(\Omega_0)^{d \times d}_{\rm sym}$ for any $\Omega_0 \subset \subset \Omega$ for $q\in [1,2)$ when $d=2$ and $q \in [1, \frac{3}{2}]$ when $d=3$. Hence, by the Rellich--Kondrashov theorem, up to a subsequence, $\bT_{n,1}$
tends to $\bT$ strongly in $L_p(\Omega_0)^{d\times  d}_{\rm sym}$ on any $\Omega_0 \subset \subset \Omega$ for all $p \in [1,\infty)$ when $d=2$ and all $p \in [1, \frac{3}{2})$ when $d=3$.

\subsection{Finite element discretization}

\label{subsec:discretesm=1}

With the spaces $\mathbb M_{n,h}$ and $\mathbb X_{n,h}$ defined in \eqref{e:P0P1} or \eqref{e:pair_QQ}, the system \eqref{e:weak_reg,m=1} is discretized by : Find $(\bT_{n,1,h}, \bu_{n,1,h})$ in $\mathbb M_{n,h} \times \mathbb X_{n,h}$ such that
\begin{alignat}{2}\label{e:weak_regh,m=1}
\begin{aligned}
\quad a_{n,1}(\bT_{n,1,h},\bS_h)+c(\bT_{n,1,h};\bT_{n,1,h},\bS_h)-b(\bS_h,\bu_{n,1,h}) &=0 \qquad &&\forall\, \bS_h \in\mathbb M_{n,h},\\
\quad b(\bT_{n,1,h},\bv_h) &= \int_\Omega \bef \cdot \bv_h \dd\bx \qquad &&\forall\, \bv_h \in \mathbb X_{n,h}.
\end{aligned}
\end{alignat}
As previously, the constraint in the second part of \eqref{e:weak_regh,m=1} is lifted by means of the projection operator $\Pi_h$ defined in \eqref{e:PIh}, $\bT_{n,1,h}^f$ is defined by \eqref{e:Thf},
$$
\bT_{n,1,h}^{\mathbf f} = \Pi_h \bT^{\mathbf{f}}_{n,1},
$$
and
$$\bT_{n,1,h}^0 := \bT_{n,1,h} - \bT_{n,1,h}^{\mathbf f}.
$$
Existence and uniqueness of the discrete solution $(\bT_{n,1,h}, \bu_{n,1,h})$ is derived as in Lemma \ref{lem:uniq.discrete}.  Again, the a priori bounds
\eqref{e:aprioriun1} and \eqref{e:aprioriTn1} hold for $\bu_{n,1,h}$ and $\bT_{n,1,h}$. In fact, even without regularization, i.e.,  without the form $a_{n,1}(\cdot,\cdot)$, existence by a Brouwer's Fixed Point and if moreover \eqref{e:cond_lambda2bis} holds, uniqueness follow by a finite-dimensional argument. But we shall not pursue the no regularization option, because, as stated at the beginning of Section \ref{subsec:convg}, we are then unable to show convergence.

The arguments of Lemma \ref{l:weak}, under analogous assumptions, show that, as $h \rightarrow 0_+$, for each $n$,
$$
\bT_{n,1,h} \rightharpoonup \bT_{n,1} \qquad \textrm{weakly in $L_{2}(\Omega)^{d\times d}_{\textrm{sym}}$}.
$$
Let us sketch the proof of the strong convergence, which is much simpler than that of Lemma \ref{l:strong}.

\begin{lem}[Strong convergence]\label{l:strongh1}
Assume that $\bef \in L_{2}(\Omega)^{d }$, that the functions $\lambda$ and $\mu$ satisfy the assumptions \eqref{e:cond_lambda1}--\eqref{e:cond_mu2}, and let
$(\bT_{n,1},\bu_{n,1})$ denote the unique solution to the regularized problem \eqref{e:weak_reg,m=1},
with $n \in \mathbb{N}$. Then, for each fixed $n \in \mathbb{N}$,
as $h \rightarrow 0_+$,
$$
\bT_{n,1,h} \rightarrow \bT_{n,1} \quad \textrm{strongly in $\mathbb{M}_{n,1} = L_2(\Omega)^{d\times d}_{\textrm{sym}}$} \quad \textrm{and} \quad \bu_{n,1,h} \rightarrow \bu_{n,1} \quad \textrm{strongly in $\mathbb{X}_{n,1} =W^{1,2}_0(\Omega)^d$}.
$$
\end{lem}

\begin{proof}
We retain the notation and the setting of the proof of Lemma \ref{l:strong}. The discrepancy $ \bT_{n,1,h}^0 - \Pi_h \bT_{n,1}^0$ satisfies
\begin{equation}\label{e:strong_conv1}
\begin{split}
& \frac{1}{n} \| \bT_{n,1,h}^0 - \Pi_h \bT_{n,1}^0 \|_{L_2(\Omega)}^2 + C \int_\Omega \frac{|\bT_{n,1,h}^0 - \Pi_h \bT_{n,1}^0|^2}{(\kappa + |\bT_{n,1,h}^0| + |\Pi_h \bT_{n,1}^0|)^{1+\alpha}} \dd \bx
 \\
 & \qquad \qquad \leq \int_\Omega \left(\mathcal A_{n,1} (\bT_{n,1,h}) - \mathcal A_{n,1}(\Pi_h \bT_{n,1}^0 + \bT_{n,1,h}^f) \right):(\bT_{n,1,h}^0-\Pi_h \bT_{n,1}^0) \dd \bx,
 \end{split}
 \end{equation}
 where $C$ is the constant in \eqref{e:mon_mu1}.
As  $\bT_{n,1,h}^0 - \Pi_h \bT_{n,1}^0 \in \mathbb V_{n,h}$,  \eqref{e:strong_conv1} reduces to
\begin{align}\label{e:strong_conv2}
\frac{1}{n} \| \bT_{n,1,h}^0-\Pi_h \bT_{n,1}^0 \|_{L_2(\Omega)}^2  &+ C \int_\Omega \frac{|\bT_{n,1,h}^0-\Pi_h \bT_{n,1}^0|^2}{(\kappa + |\bT_{n,1,h}^0| + |\Pi_h \bT_{n,1}^0|)^{1+\alpha}} \dd \bx \nonumber\\
  & \leq - \int_\Omega \mathcal A_{n,1}(\Pi_h \bT_{n,1}^0 + \bT_{n,1,h}^f) :(\bT_{n,1,h}^0-\Pi_h \bT_{n,1}^0) \dd \bx.
\end{align}
Then the weak convergence of $\bT_{n,1,h}^0-\Pi_h \bT_{n,1}^0$ to zero, the strong convergence of $\Pi_h \bT_{n,1}^0 + \bT_{n,1,h}^f$ both  in $\mathbb{M}_{n,1}$ as $h \rightarrow 0_+$, and  the continuity of the mapping $\mathcal A_{n,1}\,: \mathbb{M}_{n,1} \rightarrow \mathbb{M}_{n,1}$ yield
$$
- \int_\Omega \ \mathcal A_{n,1}(\Pi_h \bT_{n,1}^0 + \bT_{n,1,h}^f):(\bT_{n,1,h}^0-\Pi_h \bT_{n,1}^0) \dd \bx\rightarrow 0\qquad \mbox{as $h \rightarrow 0_+$}.
$$
Whence, returning to \eqref{e:strong_conv2},
$$
 \frac{1}{n} \| \bT_{n,1,h}^0 - \Pi_h \bT_{n,1}^0 \|_{L_2(\Omega)}^2 \rightarrow 0\qquad \mbox{as $h \rightarrow 0_+$},
$$
and the asserted strong convergence of $\bT_{n,1,h}$ to $\bT_{n,1}$ in $\mathbb{M}_{n,1} = L_2(\Omega)^{d \times d}_{\rm sym}$, as $h \rightarrow 0_+$,  follows for any $n\ge 1$.

For the strong convergence of $\bu_{n,1,h}$, we use again the discrete inf-sup property \eqref{e:LBBh}  to define $\bR_h \in \mathbb V_{n,h}^\perp$ satisfying
$$
\int_\Omega \strain{\bu_{n,1,h} - \Pi_h^{sz} \bu_{n,1}} : \strain{\bv_h} \dd \bx= \int_\Omega \bR_h : \strain{\bv_h}\dd \bx\qquad
\forall\, \bv_h \in \mathbb X_{n,h},
$$
where $\Pi_h^{sz}$ is the Scott--Zhang projector onto $\mathbb X_{n,h}$; see~\cite{ScoZha}.
In particular, we have
 \begin{equation}\label{e:Rhinfsup}
 \| \bR_h \|_{L_2(\Omega)} \leq  \| \strain{\bu_{n,1,h} - \Pi_h^{sz} \bu_{n,1}}\|_{L_2(\Omega)}.
 \end{equation}
For $\bv_h = \bu_{n,1,h} - \Pi_h^{sz} \bu_{n,1}$ we  then  get
 \begin{align*}
 \| \strain{\bu_{n,1,h} - \Pi_h^{sz} \bu_{n,1}}\|_{L_2(\Omega)}^2 &= \int_\Omega \bR_h : \strain{\bu_{n,1,h} - \bu_{n,1}} \dd \bx + \int_\Omega \bR_h : \strain{\bu_{n,1} - \Pi_h^{sz} \bu_{n,1}} \dd \bx\\
 & = \int_\Omega (\mathcal A_{n,1}(\bT_{n,1,h}) - \mathcal A_{n,1}(\bT_{n,1})) :\bR_h \dd \bx +  \int_\Omega \bR_h : \strain{\bu_{n,1} - \Pi_h^{sz} \bu_{n,1}} \dd \bx,
 \end{align*}
where we have also used the relations \eqref{e:weak_reg,m=1} and \eqref{e:weak_regh,m=1} to obtain the second equality.
We now argue that both terms on the right-hand side of the above equality vanish as $h \to 0_+$.
To see this, it suffices to recall the uniform bound \eqref{e:Rhinfsup} on $\mathbb \bR_h$; hence, the strong convergence results $\Pi_h^{sz} \bu_{n,1} \rightarrow \bu_{n,1}$ in $\mathbb X_{n,1}$ and $\bT_{n,1,h} \rightarrow \bT_{n,1}$ in $\mathbb M_{n,1}$, as $h \rightarrow 0_+$, together with the continuity of $\mathcal A_{n,1}$ guaranteed by Lemma~\ref{l:A_cont}, imply the stated claim.
Thanks to Korn's inequality  \eqref{e:Korn1},
$$
\| \nabla (\bu_{n,1,h} - \Pi_h^{sz} \bu_{n,1}) \|_{L_2(\Omega)} \leq \mathcal K\, \|  \strain{\bu_{n,1,h} - \Pi_h^{sz} \bu_{n,1}}\|_{L_2(\Omega)} \rightarrow 0
\qquad \mbox{as $h \rightarrow 0_+$},
$$
and therefore $\bu_{n,1,h} \rightarrow \bu_{n,1}$ in $\mathbb X_{n,1}$.
\end{proof}
Thus when $\lambda$ satisfies \eqref{e:cond_lambda2bis}, we have again, for any $\Omega_0 \subset \subset \Omega$,
$$
 \lim_{n \rightarrow \infty} \lim_{h \rightarrow 0_+} \| \bT_{n,1,h} - \bT \|_{L_1(\Omega_0)} = 0,\;\!
\lim_{n \rightarrow \infty} \lim_{h \rightarrow 0_+} \| \bu_{n,1,h} - \bu \|_{C(\overline\Omega)} = 0,\;\!
\lim_{n \rightarrow \infty} \lim_{h \rightarrow 0_+} \|\strain{\bu_{n,1,h}} - \strain{\bu} \|_{L_2(\Omega_0)} = 0.
$$

As in the preceding section, an error inequality can be established when the functions $\lambda(s)s$ and $\mu(s)s$ are Lipschitz continuous, but again the situation is much simpler.

\begin{thm}
\label{thm:err.inequm=1}
In addition to the assumptions of Lemma \ref{l:strongh1}, suppose that the real-valued functions $s \in \mathbb{R} \mapsto \lambda(s)s$
and $s \in \mathbb{R}_{\geq 0} \mapsto \mu(s)s$ are Lipschitz continuous,
i.e.,  that there exists a positive constant $\Lambda$ such that
\begin{equation}
\label{e:lipschm=1}
|\lambda(s)s - \lambda(r)r| \leq \Lambda |r-s|\qquad\forall\, r,s \in \mathbb{R}, \qquad |\mu(s)s - \mu(r)r| \leq \Lambda |r-s|\qquad \forall\, r,s \in \mathbb{R}_{\geq 0}.
\end{equation}
Then, the following error inequality holds:
\begin{alignat}{2}
\label{e:e.inequm=1}
\begin{aligned}
\frac{1}{n} \| \bT_{n,1,h} - \bT_{n,1} \|_{L_2(\Omega)} &\leq  \inf_{\bv_h \in \mathbb X_h} \| \strain{\bv_h-\bu_{n,1}} \|_{L_2(\Omega)} + 2\bigg(\frac{1}{n}+\Lambda\bigg) \| \bT_{n,1} - \Pi_h (\bT_{n,1}) \|_{L_2(\Omega)}.
\end{aligned}
\end{alignat}
\end{thm}

\begin{proof}
As in the proof of Theorem \ref{thm:err.inequ},
from the relations \eqref{e:weak_reg,m=1} and \eqref{e:weak_regh,m=1}, we infer that on one hand,
$$
\int_\Omega \strain{\bv_h}: \bS_{h} \dd \bx = 0,
$$
and on the other hand, for any $\bv_h$ in $\mathbb X_{n,h}$,
\begin{equation}
\begin{split}\label{eq:T-bound}
 \frac{1}{n} \| \bT_{n,1,h} - \Pi_h \bT_{n,1} \|_{L_2(\Omega)}^2 &+ C \int_\Omega \frac{|\bT_{n,1,h} - \Pi_h \bT_{n,1}|^2}{(\kappa + |\bT_{n,1,h}| + |\Pi_h \bT_{n,1}|)^{1+\alpha}}\dd \bx
 \\
 &  \leq \left( \| \strain{\bv_h-\bu_{n,1}} \|_{L_2(\Omega)} +  \frac{1}{n} \| \bT_{n,1} - \Pi_h \bT_{n,1} \|_{L_2(\Omega)} \right) \| \bT_{n,1,h} -\Pi_h \bT_{n,1} \|_{L_2(\Omega)} \\
 &  \quad  + \int_\Omega \left(\mathcal A_{n,1} (\bT_{n,1}) - \mathcal A_{n,1}(\Pi_h\bT_{n,1}) \right):(\bT_{n,1,h}-\Pi_h \bT_{n,1}) \dd \bx,
 \end{split}
 \end{equation}
 where $C$ is the constant in \eqref{e:mon_mu1}.
The Lipschitz property \eqref{e:lipschm=1} implies that
$$
| \mathcal A_{n,1} (\bT_{n,1}) - \mathcal A_{n,1}(\Pi_h \bT_{n,1})| \leq \frac{1}{n} |\bT_{n,1} - \Pi_h \bT_{n,1} | + 2\Lambda | \Pi_h \bT_{n,1} - \bT_{n,1}|,
$$
so that
\begin{equation}\label{eq:T-bound1}
\frac{1}{n}\| \bT_{n,1,h} - \Pi_h \bT_{n,1} \|_{L_2(\Omega)}
 \leq \left( \| \strain{\bv_h-\bu_{n,1}} \|_{L_2(\Omega)} + 2\bigg(\frac{1}{n}+\Lambda\bigg) \| \bT_{n,1} - \Pi_h \bT_{n,1} \|_{L_2(\Omega)}\right) ,
\end{equation}
which yields \eqref{e:e.inequm=1}.
\end{proof}

Under the above assumptions,  convergence rates can be derived provided that  $\bT_{n,1} \in W^{1,q}(\Omega)^{d\times d}_{\rm sym}$ with
 $q > \frac  {2d} {2+d}$  (ensuring that $W^{1,q}(\Omega)^{d\times d} \hookrightarrow L_2(\Omega)^{d\times d}$)  and $\bu_{n,1} \in W^{1+t,2}(\Omega)^d$, $t>0$  (ensuring that $W^{1,1+t}(\Omega)^{d} \hookrightarrow W^{1,2}(\Omega)^{d}$). Rates of convergence for $\| \nabla(\bu_{n,1} - \bu_{n,1,h})\|_{L_2(\Omega)}$ are obtained using the inf-sup properties and interpolation theory again.


\section{Decoupled Iterative Algorithm}\label{s:decoupled}

The convergent iterative algorithm proposed in this section for the solution of the discrete problem \eqref{e:weak_regh,m=1}, is designed to dissociate the computation of the nonlinearity from that of the elastic constraint.  We have also applied it numerically to \eqref{e:discrete_reg} in Section~\ref{sec:experiment} but  proving its  convergence is still an open problem.

 The algorithm, which belongs to the class of alternating direction methods,  proceeds in two steps. In both steps, an artificial divided difference, analogous to a discrete time derivative, is added to enhance the stability of the algorithm. The first half-step involves the monotone nonlinearity while,  in the case of \eqref{e:weak_regh,m=1},  the second half-step solves for the elastic part from a system  of linear algebraic equations whose matrix is the mass-matrix (Gram matrix) generated by the basis functions of the finite element space $\mathbb{X}_{n,h}$.  In the case \eqref{e:discrete_reg}, this second system is nonlinear.  But in both cases, our choice of the finite element space $\mathbb{M}_{n,h}$, consisting of piecewise constant approximations for the stress tensor $\bT_{n,1}$ or $\bT_n$ allows us to deal with the monotone nonlinearity involved in the first half-step in an efficient way, by solving an algebraic system with $d(d+1)/2$ unknowns independently on each element $K$ in the subdivision $\mathcal{T}_h$ of the computational domain $\Omega$.   Let us describe the algorithm applied to \eqref{e:gen_reg}.

The initialization consists of finding $(\bT_h^{(0)},\bu_h^{(0)}) \in \mathbb M_{n,h} \times \mathbb X_{n,h}$ satisfying
\begin{alignat*}{2}
\int_\Omega \strain{\bv_h}: \bT_h^{(0)} \dd \bx&= \int_\Omega \bef \cdot \bv_h \dd \bx ,\qquad &&\forall\, \bv_h \in \mathbb X_{n,h}, \\
\int_\Omega \bT_h^{(0)}:\bS_h \dd \bx&= \int_{\Omega} \strain{\bu_h^{(0)}}: \bS_h \dd \bx\qquad &&\forall\, \bS_h \in \mathbb M_{n,h}.
\end{alignat*}

Let $\tau>0$. Given $(\bT_h^{(k)},\bu_h^{(k)})$ in $\mathbb M_{n,h} \times \mathbb X_{n,h}$ for a nonnegative integer $k$, the algorithm proceeds in the following two steps.

\smallskip

\noindent
\emph{Step~1.} Find $\bT_h^{(k+\frac 1 2)}$ in $\mathbb M_{n,h}$ such that, for all $\bS_h \in \mathbb M_{n,h}$,
\begin{align*}
& \frac{1}{\tau} \int_{\Omega} (\bT_h^{(k+\frac 1 2)} - \bT_h^{(k)}):\bS_h \dd \bx\\
&\quad + \int_\Omega  \left( \lambda(\tr{\bT_h^{(k+\frac 1 2)}})\tr{\bT_h^{(k+\frac 1 2)}} \tr{\bS_h} + \mu(|(\bT_h^{(k+\frac 1 2)})^{\bd}|)(\bT_h^{(k+\frac 1 2)})^{\bd} : \bS_h \right)\dd \bx\\
&\quad \quad = \int_{\Omega} \strain{\bu_h^{(k)}}: \bS_h \dd \bx- \int_{\Omega}\Big(\frac{\tr{\bT_h^{(k)}}\bI}{n|\tr{\bT_h^{(k)}}|^{1-\frac{1}{t}} } + \frac{(\bT_h^{(k)})^{\mathbf d}}{n|(\bT_h^{(k)})^{\mathbf d}|^{1-\frac{1}{t}}}\Big):\bS_h \dd \bx.
\end{align*}
As was already mentioned, because $\bT_h^{(k+\frac 1 2)}$ is piecewise constant, the above system reduces to decoupled algebraic systems of  $d(d+1)/2$ unknowns each, in every element in the subdivision of the computational domain.

\noindent
\emph{Step~2.} Find $\bT_h^{(k+1)} \in \mathbb M_{n,h}$ and $\bu_h^{(k+1)} \in \mathbb X_{n,h}$ such that
\begin{align*}
& \frac{1}{\tau} \int_{\Omega} (\bT_h^{(k+1)} - \bT_h^{(k+\frac 1 2)}):\bS_h \dd \bx +\int_{\Omega}\Big(\frac{\tr{\bT_h^{(k+1)}}\bI}{n|\tr{\bT_h^{(k+1)}}|^{1-\frac{1}{t}} } + \frac{(\bT_h^{(k+1)})^{\mathbf d}}{n|(\bT_h^{(k+1)})^{\mathbf d}|^{1-\frac{1}{t}}}\Big):\bS_h \dd \bx\\
&\quad = \int_\Omega \strain{\bu_h^{(k+1)}}:\bS_h \dd \bx - \int_\Omega  \left( \lambda(\tr{\bT_h^{(k+\frac 1 2)}})\tr{\bT_h^{(k+\frac 1 2)}} \tr{\bS_h} + \mu(|(\bT_h^{(k+\frac 1 2)})^{\bd}|)(\bT_h^{(k+\frac 1 2)})^{\bd} : \bS_h \right) \dd \bx,
\end{align*}
and such that, for all $\bv_h \in \mathbb X_{n,h}$,
$$
\int_\Omega \strain{\bv_h}:\bT_h^{(k+1)} \dd \bx = \int_\Omega \bef \cdot \bv_h \dd \bx.
$$

When $t=1$, the initialization is unchanged and the two steps simplify as follows:

\smallskip

\noindent
\emph{Step~1.} Find $\bT_h^{(k+\frac 1 2)}$ in $\mathbb M_{n,h}$ such that, for all $\bS_h \in \mathbb M_{n,h}$,
\begin{align*}
& \frac{1}{\tau} \int_{\Omega} (\bT_h^{(k+\frac 1 2)} - \bT_h^{(k)}):\bS_h \dd \bx\\
&\quad + \int_\Omega  \left( \lambda(\tr{\bT_h^{(k+\frac 1 2)}})\tr{\bT_h^{(k+\frac 1 2)}} \tr{\bS_h} + \mu(|(\bT_h^{(k+\frac 1 2)})^{\bd}|)(\bT_h^{(k+\frac 1 2)})^{\bd} : \bS_h\right)\dd \bx\\
&\quad \quad = \int_\Omega \strain{\bu_h^{(k)}}:\bS_h \dd \bx- \frac{1}{n} \int_\Omega \bT_h^{(k)} : \bS_h \dd \bx.
\end{align*}

\medskip

\noindent
\emph{Step~2.} Find $\bT_h^{(k+1)} \in \mathbb M_{n,h}$ and $\bu_h^{(k+1)} \in \mathbb X_{n,h}$ such that
\begin{align*}
& \frac{1}{\tau} \int_{\Omega} (\bT_h^{(k+1)} - \bT_h^{(k+\frac 1 2)}):\bS_h \dd \bx+ \frac{1}{n} \int_\Omega \bT_h^{(k+1)} : \bS_h \dd \bx\\
&\quad = \int_\Omega \strain{\bu_h^{(k+1)}}:\bS_h \dd \bx - \int_\Omega  \left( \lambda(\tr{\bT_h^{(k+\frac 1 2)}})\tr{\bT_h^{(k+\frac 1 2)}} \tr{\bS_h} + \mu(|(\bT_h^{(k+\frac 1 2)})^{\bd}|)(\bT_h^{(k+\frac 1 2)})^{\bd} : \bS_h\right) \dd \bx,
\end{align*}
and such that, for all $\bv_h \in \mathbb X_{n,h}$,
$$
\int_\Omega \strain{\bv_h}:\bT_h^{(k+1)} \dd \bx = \int_\Omega \bef \cdot \bv_h \dd \bx.
$$

Following the general theory of Lions and Mercier \cite{LM79}, we now prove that the iterative algorithm  for $t=1$  converges to the solution of the decoupled system.

\begin{thm}[Convergence of the Iterative Decoupled Algorithm]
\label{t:convLMm=1}
Assume that $\lambda$ and $\mu$ satisfy \eqref{e:cond_lambda1}--\eqref{e:cond_mu2} and that $n \ge 1$.
Let $\bT_{n,1,h} \in \mathbb M_{n,h}$ be the first component of the solution of \eqref{e:weak_regh,m=1} and let $\bT_h^{(k)} \in \mathbb M_{n,h}$, $k=1,2,\dots$,  be successive iterates computed by the iterative algorithm, with $\tau>0$.
We then have that
$$
\lim_{k\to \infty} \| \bT_h^{(k)} - \bT_{n,1,h} \|_{L_2(\Omega)} = 0.
$$
\end{thm}

\begin{proof}

The nonlinear part of the system  is represented by the following
operator, $\mathfrak A_h: \mathbb M_{n,h} \rightarrow \mathbb M_{n,h}$ defined by $\mathfrak A_h \bS_h = \bA_h$, where  for all $\bR_h \in \mathbb M_{n,h}$,
$$
\int_\Omega \bA_h  : \bR_h \dd \bx= \int_\Omega  \left( \lambda(\tr{\bS_h})\tr{\bS_h} \tr{\bR_h} + \mu(|(\bS_h^{\bd}|)\bS_h^{\bd} : \bR_h \right)\dd \bx,
$$
and the linear part,  excluding  the artificial time derivative, is represented by  the  function
$$\bB_h^{(k)} := \frac{1}{n} \bT_h^{(k)} - \strain{\bu_h^{(k)}}.$$
With these notations, the first step of the iterative algorithm reads
$$
(I+\tau \mathfrak A_h)\bT_h^{(k+\frac 1 2)} = \bT_h^{(k)} - \tau \bB_h^{(k)},
$$
or, equivalently,
$$
\bT_h^{(k+\frac 1 2)} =  (I+\tau \mathfrak A_h)^{-1} (\bT_h^{(k)} - \tau \bB_h^{(k)}).
$$
It is convenient to introduce  the following two auxiliary tensors:
\begin{equation}
\label{e:Lambd}
\bLambda_h^{(k)} := \bT_h^{(k)} + \tau \left(\frac{1}{n} \bT_h^{(k)} - \strain{\bu_h^{(k)}}\right) = \bT_h^{(k)} + \tau \bB_h^{(k)}
\end{equation}
and
$$
\bTheta_h^{(k)} := 2 \bT_h^{(k)} - \bLambda_h^{(k)},
$$
whereby
$$
\bT_h^{(k)} = \frac 1 2 ( \bTheta_h^{(k)} + \bLambda_h^{(k)} ).
$$
We shall see that the convergence of $\bT_h^{(k)}$ will result from  that of $\bLambda_h^{(k)}$ and $\bTheta_h^{(k)}$.  With these tensors,
the second step of the iterative algorithm reads
$$
\bLambda_h^{(k+1)} = (I-\tau \mathfrak A_h)\bT_h^{(k+\frac 1 2)} = (I-\tau \mathfrak A_h)(I+\tau \mathfrak A_h)^{-1} (\bT_h^{(k)} -  \tau \bB_h^{(k)}) .
$$
Notice that, from \eqref{e:Lambd}, $\bB_h^{(k)} = \frac 1 {2\tau} (\bLambda_h^{(k)} - \bTheta_h^{(k)})$, and we define $\bC_h^{(k)} := \frac 1 {2\tau}( \bTheta_h^{(k)} - \bLambda_h^{(k+1)})$.
In addition, we note for later that
$$
(I+\tau \mathfrak A_h)^{-1} \bTheta_h^{(k)} = \frac 1 2 (\bLambda_h^{(k+1)} + \bTheta_h^{(k)}),
$$
which implies \ that
$$
\mathfrak A_h\left( \frac{\bLambda_h^{(k+1)} + \bTheta_h^{(k)}}{2}\right) = \frac{1}{2\tau} (\bTheta_h^{(k)} - \bLambda_h^{(k+1)}) = \bC_h^{(k)}.
$$
We also define the analogous quantities
$$
\bB_h := \frac{1}{n} \bT_{n,1,h} - \strain{\bu_{n,1,h}}, \qquad \bLambda_h := \bT_{n,1,h} + \tau \bB_h, \qquad \bC_h:= \mathfrak A_h \bT_{n,1,h}, \qquad \bTheta_h = \bT_{n,1,h} + \tau \bC_h.
$$
With these notations, the first relation in \eqref{e:discrete_reg} reads
$$
\bC_h + \bB_h = \mathfrak A_h \bT_{n,1,h} + \frac{1}{n} \bT_{n,1,h} - \strain{\bu_{n,1,h}} = \mathbf 0,
$$
and so
$$
\bLambda_h + \bTheta_h = 2 \bT_{n,1,h} + \tau (\bB_h + \bC_h) = 2 \bT_{n,1,h},
$$
which in turn implies that
$$
\bC_h = \frac 1 \tau (\bTheta_h - \bT_{n,1,h}) = \frac 1 {2\tau} (\bTheta_h - \bLambda_h).
$$
Similarly, for $\bB_h$ we have the decomposition
$$
\bB_h = \frac{1}{2\tau} (\bLambda_h - \bTheta_h).
$$

We can now express the discrepancy between $\bT_{n,1,h}$ and $\bT_h^{(k)}$ as follows:
\begin{align*}
\frac{1}{n} \| \bT_h^{(k)} - \bT_{n,1,h} \|_{L_2(\Omega)}^2  &= \frac{1}{n} \int_\Omega (\bT_h^{(k)} - \bT_{n,1,h}): (\bT_h^{(k)} - \bT_{n,1,h}) \dd \bx\\
& = \int_\Omega (\bB_h^{(k)} - \bB_h): (\bT_h^{(k)} -\bT_{n,1,h}) \dd \bx
 +  \int_\Omega \strain{\bu_h^{(k)} - \bu_{n,1,h}} : (\bT_h^{(k)} - \bT_{n,1,h}) \dd \bx.
\end{align*}
Because, for all $\bv_h \in \mathbb X_{n,h}$,
$$
\int_\Omega \bT_h^{(k)} : \strain{\bv_h} \dd \bx = \int_\Omega \bef \cdot \bv_h =  \int_\Omega \bT_{n,1,h} : \strain{\bv_h} \dd \bx,
$$
we deduce that $\bT_h^{(k)} - \bT_{n,1,h} \in \mathbb V_{n,h}$, and therefore
$$
\frac{1}{n} \| \bT_h^{(k)} - \bT_{n,1,h} \|_{L_2(\Omega)}^2 = \int_\Omega (\bB_h^{(k)} - \bB_h): (\bT_h^{(k)} -\bT_{n,1,h}) \dd \bx.
$$
The relations
$$
\bB_h^{(k)} - \bB_h = \frac{1}{2\tau} \left( \bLambda_h^{(k)} - \bLambda_h - (\bTheta_h^{(k)} - \bTheta_h) \right)
$$
and
$$
\bT_h^{(k)} - \bT_{n,1,h} = \frac{1}{2} \left( \bLambda_h^{(k)} - \bLambda_h + (\bTheta_h^{(k)} - \bTheta_h) \right)
$$
further lead to
\begin{equation}\label{e:b-bk}
\frac{1}{n} \| \bT_h^{(k)} - \bT_{n,1,h} \|_{L_2(\Omega)}^2 = \frac{1}{4 \tau} \left(\| \bLambda_h^{(k)} - \bLambda_h \|_{L_2(\Omega)}^2 - \| \bTheta_h^{(k)} - \bTheta_h \|_{L_2(\Omega)}^2 \right).
\end{equation}
This, of course, implies that
\begin{equation}\label{e:theta<lamb}
 \| \bTheta_h^{(k)} - \bTheta_h \|_{L_2(\Omega)} \le  \| \bLambda_h^{(k)} - \bLambda_h \|_{L_2(\Omega)}.
\end{equation}
In addition, we have that
\begin{equation}\label{e:c_mon}
\begin{split}
& \int_\Omega (\bC_h^{(k)} -\bC_h) : \left( \frac{\bLambda_h^{(k+1)} + \bTheta_h^{(k)}}{2} - \bT_{n,1,h}\right) \dd \bx\\
& \qquad = \int_\Omega \left( \mathfrak A_h\left( \frac{\bLambda_h^{(k+1)} + \bTheta_h^{(k)}}{2}\right) - \mathfrak A_h \bT_{n,1,h} \right):  \left( \frac{\bLambda_h^{(k+1)} + \bTheta_h^{(k)}}{2} - \bT_{n,1,h}\right) \dd \bx \\
& \qquad  \geq 0,
\end{split}
\end{equation}
thanks to the monotonicity property of $\mathfrak A_h$ due to \eqref{e:mon_mu1} and \eqref{e:mon_lambda1}.
On the other hand, we compute
\begin{equation}\label{e:c-ck}
\int_\Omega (\bC_h^{(k)} -\bC_h) : \left( \frac{\bLambda_h^{(k+1)} + \bTheta_h^{(k)}}{2} - \bT_{n,1,h}\right) \dd \bx = \frac{1}{4\tau} \left( \| \bTheta_h^{(k)} - \bTheta_h \|_{L_2(\Omega)}^2 - \| \bLambda_h^{(k+1)} - \bLambda_h \|_{L_2(\Omega)}^2\right).
\end{equation}
Hence, we find that
\begin{equation}
\label{e:tht-lam}
\frac{1}{4\tau} \left( \| \bTheta_h^{(k)} - \bTheta_h \|_{L_2(\Omega)}^2 - \| \bLambda_h^{(k+1)} - \bLambda_h \|_{L_2(\Omega)}^2\right) \geq 0,
\end{equation}
and therefore, in view of \eqref{e:theta<lamb},
\begin{equation}
\label{e:tht-lam1}
\| \bLambda_h^{(k+1)} - \bLambda_h \|_{L_2(\Omega)} \leq   \| \bTheta_h^{(k)} - \bTheta_h \|_{L_2(\Omega)} \leq   \| \bLambda_h^{(k)} - \bLambda_h \|_{L_2(\Omega)}.
\end{equation}
This guarantees that the sequence $\| \bLambda_h^{(k)} - \bLambda_h \|_{L_2(\Omega)}$ of nonnegative real numbers is monotonic nonincreasing, and so converging.Furthermore, we have
$$
\lim_{k \to \infty} \| \bLambda_h^{k} - \bLambda_h \|_{L_2(\Omega)} = \lim_{k \to \infty} \| \bTheta_h^{k} - \bTheta_h \|_{L_2(\Omega)}.
$$
With these two limits, \eqref{e:b-bk} implies that
$$
\lim_{k \to \infty}  \frac{1}{n}\| \bT_h^{(k)} - \bT_{n,1,h} \|_{L_2(\Omega)} =0.
$$
That completes the proof.
\end{proof}

\begin{rem}[Post-processing]
\label{rm:convof uhk}
{\rm Since $\bT_h^{(k+\frac 1 2)}$ within the iterative algorithm does not satisfy the constraint,
it seems difficult to prove its convergence to $\bT_{n,1,h}$, and as a consequence the convergence of $\bu_h^{(k)}$ to $\bu_{n,1,h}$, as
$k \rightarrow \infty$.
Instead, given $\bT_h^{(k)}$, one can define $\widetilde{\bu}_h^{(k)} \in \mathbb X_{n,h}$ as the solution to the elasticity problem
$$
\int_\Omega \strain{\widetilde{\bu}_h^{(k)}}:\strain{\bv_h} \dd \bx = \frac{1}{n} \int_\Omega \bT_h^{(k)} : \strain{\bv_h} \dd \bx + \int_\Omega \mathfrak A_h(\bT_h^{(k)}) :
\strain{\bv_h}\dd \bx\qquad  \forall\, \bv_h \in \mathbb X_{n,h}.
$$
The convergence of $\widetilde{\bu}_h^{(k)}$ towards $\bu_{n,1,h}$ follows from the convergence of $\bT_h^{(k)}$ towards $\bT_{n,1,h}$, as $k \rightarrow \infty$.}
\end{rem}


\section{Numerical Experiments}
\label{sec:experiment}

We now illustrate the performance of the decoupled algorithm in several situations.
We start with a setting where the exact solution is accessible, in order to demonstrate the asymptotic behavior of the algorithm and to determine adequate values for the numerical parameters to be used in other situations.
We then challenge our algorithm in the two-dimensional case of a crack.

The numerical results presented below are obtained using the \emph{deal.ii} library \cite{dealII85}.
The subdivisions of $\Omega$ consist of quadrilaterals/hexahedra.
Unless stated otherwise, the stress tensor $\bT$ is approximated using piecewise constant polynomials while the displacement $\bu$ is approximated by piecewise polynomials of degree one in each co-ordinate direction; see Section~\ref{subsec:otherfit}.


\subsection{Details of the Implementation}
\label{subsec:implement}
For a given tolerance parameter \textrm{TOL}$>0$, the decoupled iterative algorithm described in Section~\ref{s:decoupled} is terminated once the relative tolerance on the increment
\begin{equation}\label{e:tolerance}
\frac{\| \bT_{h}^{(k+1)} - \bT_{h}^{(k)} \|_{L_p(\Omega)} + \| \nabla  (\bu_{h}^{(k+1)} - \bu_{h}^{(k)}) \|_{L_{2}(\Omega)}}{\| \bT_{h}^{(k)}\|_{L_p(\Omega)} + \| \nabla \bu_{h}^{(k)}\|_{L_2(\Omega)}} \leq \textrm{TOL}
\end{equation}
is satisfied, where $p=2$ when $t=1$ and $p=1$ otherwise.

Each step of the decoupled algorithm requires subiterations (only step 1 when $t=1$), which are terminated once the relative tolerance on the increments is smaller than $ \textrm{TOL/5}$.

\subsection{Validation on Smooth Solutions}
\label{subsec:validation}

We illustrate the performance of the decoupled algorithm introduced in Section~\ref{s:decoupled} on the discretization of
the regularized system
\begin{alignat}{2}\label{e:weak_reg_mod}
\begin{aligned}
\quad a_n(\bT,\bS)+c(\bT;\bT,\bS)-b(\bS,\bu) &= \int_{\Omega} \bG:\bS \dd\bx \qquad &&\forall\, \bS \in\mathbb M,\\
\quad b(\bT,\bv) &= \int_\Omega \bef \cdot \bv \dd\bx \qquad &&\forall\, \bv \in \mathbb X.
\end{aligned}
\end{alignat}
The presence of the given tensor $\bG:\Omega \rightarrow \mathbb R^{d\times d}_\textrm{sym}$ on
the right-hand side of the first equation allows us to exhibit an exact solution in closed form;
compare with \eqref{e:weak_reg}. In fact, we let
$\lambda(s) = \mu(s) = (1+s^2)^{-\frac{1}{2}}$, $\Omega = (0,1)^2$ and, given $n \geq 1$,
we define $\bef$ and $\bG$ so that
\begin{equation}\label{e:exact_closed}
\bu(x,y) = \begin{pmatrix} y(1-y) \\ 0 \end{pmatrix}, \qquad \bT(x,y) = \begin{pmatrix} {\rm e}^x & 0 \\ 0 & \cos y  \end{pmatrix}
\end{equation}
solves \eqref{e:weak_reg_mod}.

Regarding the numerical parameters, we fix the pseudo-time increment parameter $\tau = 0.01$ and perform simulations for several values of the regularization parameter $n$ and for $t=1$ (linear regularization) and $t=n$.
The computational domain $\Omega$ is subdivided by using a sequence of uniform partitions consisting of squares of diameter $h=2^{-i}$, $i=0,\dots,7$.
The target tolerance for the iterative algorithm is set to $\textrm{TOL} = 10^{-5}$.

\subsubsection*{Convergence as $h \to 0$}
We provide in Table~\ref{table:conv_smooth} the corresponding errors $e_\bu:= \| \nabla (\bu_n - \bu_{n,h})\|_{L_2(\Omega)}$ and $e_\bT:= \| \bT_n - \bT_{n,h}\|_{L_p(\Omega)}$.
Theorem~\ref{thm:err.inequ} predicts a rate of convergence of $O(h^{\frac{1}{t}})$ for both quantities  which seems to be pessimistic (in this model problem
with a smooth solution) since convergence of order $O(h)$ is observed for $t=1$ and $t=n$.
In fact, we also ran tests with other values of $t>1$ and observed the same order $O(h)$.

\begin{table}[ht!]
\begin{tabular}{c?c|c?c|c|c|c?}
$h$   & \multicolumn{2}{c?}{$n = 1$}  & \multicolumn{4}{c?}{$n=2$}  \\
  \Xhline{2\arrayrulewidth}
  &      \multicolumn{2}{c?}{$t=1$} &   \multicolumn{2}{c|}{$t=1$} &   \multicolumn{2}{c?}{$t=2$} \\
   &     $e_\bu$ &  $e_\bT$ &   $e_\bu$ &  $e_\bT$   &  $e_\bu$ &  $e_\bT$  \\
\cline{2-7}
 $2^{-2}$ &0.14438 & 0.03946	& 0.14436 &  0.05453 & 0.14434 &  0.05182  \\
 $2^{-3}$ &0.07217 & 0.01973	& 0.07217 &  0.02725 & 0.07217  & 0.02486 \\
 $2^{-4}$ &0.03609 & 0.00986	& 0.03609 &  0.01363 & 0.03609  & 0.01224 \\
 $2^{-5}$ &0.01804 & 0.00493	& 0.01804 &  0.00681 & 0.01804  & 0.00625 \\
 $2^{-6}$ &0.00902 & 0.00247 	& 0.00902 &  0.00341 & 0.00902  & 0.00327 \\
 $2^{-7}$ &0.00451 & 0.00124 	& 0.00451 &  0.00171 & 0.00451  & 0.00177 \\
\end{tabular}

\bigskip

\caption{Asymptotic behaviour of $e_\bu$ and $e_\bT$ for $n=t=1$ and $n=2$ with $t=1$ or $t=2$.
The method exhibits convergence of order one in all cases.
This is in accordance with Theorem~\ref{thm:err.inequ} when $t=1$ but better than predicted for $t>1$.
}\label{table:conv_smooth}
\end{table}

\subsubsection*{Convergence as $n \to \infty$}
We now turn our attention to the convergence of the algorithm when $n \to \infty$ for a fixed subdivision corresponding to $h=2^{-7}$.
Again, we consider two cases: $t=1$ (linear regularization) and $t=n$.
The data $\bef$ and $\bG$ are modified so that $(\bu,\bT)$ given by \eqref{e:exact_closed} solves \eqref{e:weak_reg_mod} without regularization, i.e., without the bilinear form $a_n(\cdot,\cdot)$.
The results are reported in Table~\ref{tab:conv_delta}; they indicate that in this smooth setting, $e_\bu+e_\bT \to 0$ as $n \to \infty$.
\begin{table}[ht!]
\begin{tabular}{c|cccc}
\multicolumn{4}{c}{$t=1$}\\
              & $n = 1.0$ &$n = 500$  & $n=1000$  \\
              \hline
$e_\bu$ &  0.80168   & 0.00927 & 0.00617   \\
$e_\bT$ & 1.53397   & 0.06777 &  0.03583
\end{tabular}\qquad
\begin{tabular}{c|cccc}
\multicolumn{4}{c}{$t=n$}\\
              & $n = 1.0$ & $n=500$ & $n=1000$  \\
              \hline
$e_\bu$ &  0.80167   & 0.00519  & 0.00470  \\
$e_\bT$ & 2.18173   &  0.04052 & 0.02234
\end{tabular}

\bigskip

\caption{Convergence of the decoupled algorithm when $n \to \infty$ for a fixed spatial
resolution ($h=2^{-7}$) using linear ($t=1$) and nonlinear ($t=n$) regularization.
In the nonlinear regularization case, the error in the stress is always measured in $L_1(\Omega)$ (instead of $L_2(\Omega)$ when $t=1$).
The two algorithms yield similar results.
}\label{tab:conv_delta}
\end{table}

\subsection{Inf-Sup condition}

We conclude the section containing our numerical experiments with an observation on the inf-sup condition
when using quadrilaterals. We consider the discretization of the linear problem, for which the solution
$(\bu,\bT) \in \mathbb X \times \mathbb M$ is defined as the one satisfying
$$
\int_\Omega \bT:\bS - \int_\Omega \strain{\bu}:\bS + \int_\Omega \strain{\bv}:\bT= \int_\Omega \bef \cdot \bv \qquad \forall\, (\bv,\bS) \in \mathbb X \times  \mathbb M.
$$

In view of the discussion in Section~\ref{subsec:otherfit}, any pair of discrete spaces satisfying $\strain{\mathbb X _{n,h}} \subset \mathbb M_{n,h}$, such as in \eqref{e:pair_QQ} or in \eqref{e:pair_PP}, yields an inf-sup stable scheme.
In contrast, unstable modes (that violate the discrete inf-sup condition with an $h$-independent positive inf-sup constant) can be proved to exist when using the pair in \eqref{e:Q1--P0}.
However, for the exact (smooth) solution
\begin{equation*}
\bu(x,y) = \begin{pmatrix} x\,{\rm e}^y \\ \sin x \end{pmatrix}, \qquad \bT(x,y) = \strain{\bu(x,y)}
\end{equation*}
on a square domain $\Omega=(0,1)^2$, the finite element approximations  using this unstable pair showed no signs
of instability in our numerical experiments.
In fact a linear rate of convergence for $\| \nabla(\bu-\bu_h) \|_{L_2(\Omega)}$ and $\| \bT-\bT_h\|_{L_1(\Omega)}$ was observed in the limit of $h \rightarrow 0$; see Figure~\ref{f:infup}.

\begin{figure}
\includegraphics[scale=0.6]{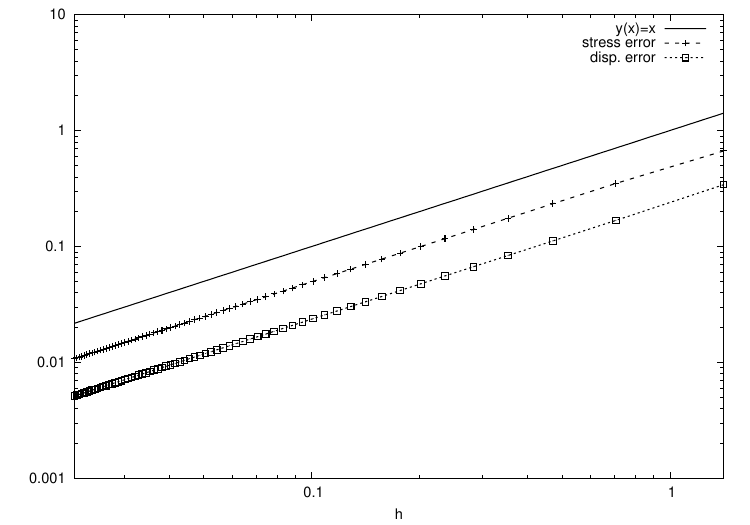}

\caption{Decay of $\| \bu-\bu_h \|_{L_2(\Omega)}$ and $\| \bT-\bT_h\|_{L_1(\Omega)}$ as a function of the mesh-size $h$ using the unstable pair in  \eqref{e:Q1--P0}. Both quantities decay linearly.}\label{f:infup}
\end{figure}

It is worth mentioning that, when using $(\mathcal Q^1_h)^{d \times d}_{\rm sym}$ instead of
$(\mathcal Q^0_h)^{d \times d}_{\rm sym}$  for $\bT_h$, the approximation of $\bu_h$ remains
exactly the same while the approximation of $\bT_h$ is more accurate on any given subdivision,
but it still only exhibits first-order convergence as $h \rightarrow 0$.
The intriguing fact that, for the exact solution $(\bu, \bT)$ considered above, the scheme
exhibits the  optimal rate of convergence dictated by interpolation theory, even though an
inf-sup unstable finite element pair is being used, will be the subject to future work.

\subsection{Crack problem}
\label{subsec:crackpb}

We consider the ``crack problem'' described in Figure~\ref{s:num:crack_setting}.
A horizontal force of magnitude $f $  is applied to the right face of the domain (III),
while the left faces (I and II) are free to deform (i.e.,  no external force is being applied there).
The  top and bottom (IV) are fixed with $\bu = \mathbf 0$.

\begin{figure}
\begin{picture}(0,0)%
\includegraphics{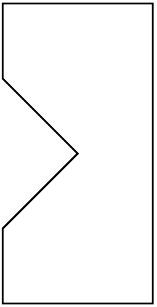}%
\end{picture}%
\setlength{\unitlength}{3947sp}%
\begingroup\makeatletter\ifx\SetFigFont\undefined%
\gdef\SetFigFont#1#2#3#4#5{%
  \reset@font\fontsize{#1}{#2pt}%
  \fontfamily{#3}\fontseries{#4}\fontshape{#5}%
  \selectfont}%
\fi\endgroup%
\begin{picture}(1244,2444)(2379,-2783)
\put(2200,-300){$(0,2)$}
\put(3500,-300){$(1,2)$}
\put(2000,-1100){$(0,\frac 3 2)$}
\put(2000,-2100){$(0,\frac 1 2)$}
\put(3100,-1600){$\!\!(\frac 1 2,1)$}
\put(2000,-2800){$(0,0)$}
\put(3650,-2800){$(1,0)$}
\put(2200,-700){$I$}
\put(2200,-2500){$I$}
\put(2700,-1600){$II$}
\put(3700,-1600){$III$}
\put(2900,-600){$IV$}
\put(2900,-2700){$IV$}
\end{picture}%

\caption{Crack problem. A horizontal compressive force $\bT \bn = (f,0)^{\rm T}$ for $f>0$ is applied on the side $III$, while no force (i.e.,  $\bT \bn = \mathbf 0$)
is imposed on the side marked by $I$ and $II$. The top and bottom sides are fixed, i.e.,  $\bu = \mathbf 0$.} \label{s:num:crack_setting}
\end{figure}
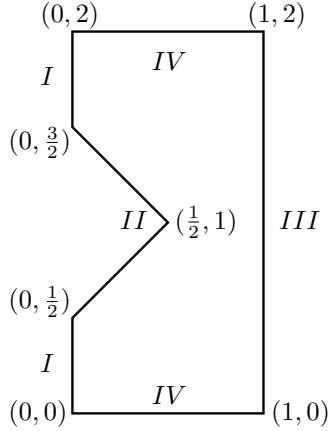

We set $\lambda(s)=\mu(s) = (1+s^2)^{-\frac 1 2}$. In view of the performance observed in Section~\ref{subsec:validation}, we set the numerical parameters at $\tau=2$, $n=100$, and $t=1$.
The domain is partitioned into 16384 quadrilaterals  of minimal diameter $h=0.011$.
The stress is approximated in $(\mathcal Q^0_h)^{d \times d}_{\rm sym}$ and the displacement in $(\mathcal Q^1_h)^{d} \cap \mathbb X_n$.
In Figure~\ref{fig:uVSf}, we provide the deformed domain predicted by the algorithm for different values of $f$.
We also report in Table~\ref{tab:uTvsf} the evolution of $\| \nabla \bu_h \|_{L_\infty(\Omega)}$ and $\| \bT_h \|_{L_\infty(\Omega)}$ as the
magnitude of the force increases.
The influence of the latter is severe on $\| \bT_h \|_{L_\infty(\Omega)}$ while relatively moderate on $\| \strain{\bu_h} \|_{L_\infty(\Omega)} \leq \| \nabla \bu_h \|_{L_\infty(\Omega)}$. This is in accordance with the properties of the strain-limiting model considered.

\begin{figure}[ht!]
\includegraphics[width=0.9\textwidth]{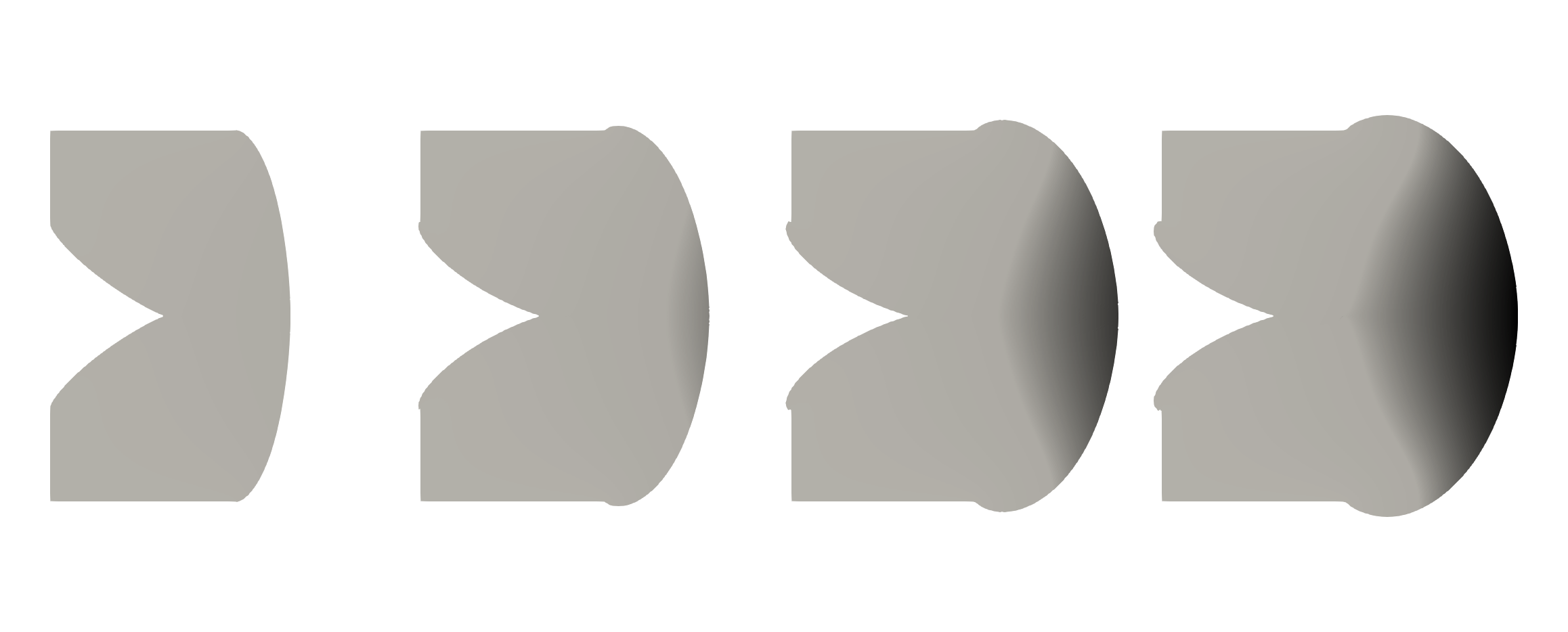}
\caption{Crack problem. The deformed domain for different force-magnitudes $f=0.25,0.5,0.75,1$ (from left to right) pulling the right face of the computational domain.
The gray scale describes the magnitude of the displacement $| \bu|$, where white corresponds to $0$ and black to $0.92$.}\label{fig:uVSf}
\end{figure}
\begin{table}[ht!]
\begin{tabular}{c|c|c|c|c|c|c}
   & $f=0.25$ & $f=0.5$ & $f=0.75$ & $f=1$ & $f=1.25$ & $f=1.5$ \\
\hline
$\| \nabla \bu_{n,1,h} \|_{L_\infty(\Omega)}$ &  1.0656 & 2.2510  & 3.5032 & 5.2703 & 7.0492  & 8.8003 \\
$\| \bT_{n,1,h} \|_{L_\infty(\Omega)}$    & 0.92231 & 5.3090  &  18.17  & 46.5215  & 95.3902   & 166.335
\end{tabular}

\bigskip

\caption{Evolutions of $\| \nabla \bu_h \|_{L_\infty(\Omega)}$ and $\| \bT_h \|_{L_\infty(\Omega)}$
as functions of the force-magnitude $f$ pulling the right face of the domain.
The influence of increasing the magnitude of the force is severe on the stress while relatively moderate on the strain. This is in accordance with the properties of the
strain-limiting model considered.}\label{tab:uTvsf}
\end{table}


\begin{thebibliography}{10}

\bibitem{AcostaDuranMuschietti2006}
{\sc G.~Acosta, R.~G. Dur{\'a}n, and M.~A. Muschietti}, {\em Solutions of the
  divergence operator on {J}ohn domains}, Adv. Math., 206 (2006), pp.~373--401.

\bibitem{dealII85}
{\sc D.~Arndt, W.~Bangerth, D.~Davydov, T.~Heister, L.~Heltai, M.~Kronbichler,
  M.~Maier, J.-P. Pelteret, B.~Turcksin, and D.~Wells}, {\em The
  \texttt{deal.II} library, version 8.5}, Journal of Numerical Mathematics, 25
  (2017), pp.~137--146.

\bibitem{ref:BBMS}
{\sc L.~Beck, M.~Bul\'{\i}\v{c}ek, J.~M\'{a}lek, and E.~S\"{u}li}, {\em On the
  existence of integrable solutions to nonlinear elliptic systems and
  variational problems with linear growth}, Archive for Rational Mechanics and
  Analysis, 225 (2017), pp.~717--769.

\bibitem{BMRS14}
{\sc M.~{Bul\'\i\v cek}, J.~M\'alek, K.~Rajagopal, and E.~S\"uli}, {\em On
  elastic solids with limiting small strain: modelling and analysis}, EMS Surv.
  Math. Sci., 1 (2014), pp.~283--332.

\bibitem{Cia}
{\sc P.~G. Ciarlet}, {\em Basic error estimates for elliptic problems}, in
  Handbook of {N}umerical {A}nalysis, {V}ol.\ {II}, Handb. Numer. Anal., II,
  North-Holland, Amsterdam, 1991, pp.~17--351.

\bibitem{ref:GiR}
{\sc V.~Girault and P.-A. Raviart}, {\em Finite {E}lement {M}ethods for
  {N}avier-{S}tokes {E}quations: {T}heory and {A}lgorithms}, vol.~5 of Springer
  Series in Computational Mathematics, Springer-Verlag, Berlin, 1986.

\bibitem{ref:Grafakos}
{\sc L.~Grafakos}, {\em Classical {F}ourier {A}nalysis}, vol.~249 of Graduate
  Texts in Mathematics, Springer, New York, third~ed., 2014.

\bibitem{Gron}
{\sc T.~H. Gronwall}, {\em \"{U}ber die {L}aplacesche {R}eihe}, Math. Ann., 74
  (1913), pp.~213--270.

\bibitem{Jiang}
{\sc R.~Jiang and A.~Kauranen}, {\em Korn's inequality and {J}ohn domains},
  Calc. Var. Partial Differential Equations, 56 (2017), pp.~Art. 109, 18.

\bibitem{ref:Lions}
{\sc J.-L. Lions}, {\em Quelques {M}\'ethodes de {R}\'esolution des
  {P}robl\`emes aux {L}imites {N}on {L}in\'eaires}, Dunod, Paris, France, 1969.

\bibitem{LM79}
{\sc P.-L. Lions and B.~Mercier}, {\em Splitting algorithms for the sum of two
  nonlinear operators}, SIAM J. Numer. Anal., 16 (1979), pp.~964--979.

\bibitem{MNRR96}
{\sc J.~M\'alek, J.~{Ne\v cas}, M.~Rokyta, and M.~{Ru\v zi\v cka}}, {\em Weak
  and {M}easure-valued {S}olutions to {E}volutionary {PDE}s}, vol.~13 of
  Applied Mathematics and Mathematical Computation, Chapman \& Hall, London,
  1996.

\bibitem{Raj2003}
{\sc K.~R. Rajagopal}, {\em On implicit constitutive theories}, Applications of
  Mathematics, 48 (2003), pp.~279--319.

\bibitem{Raj2007}
\leavevmode\vrule height 2pt depth -1.6pt width 23pt, {\em The elasticity of
  elasticity}, Zeitschrift angew. Math. Phys., 58 (2007), pp.~309--317.

\bibitem{RR}
{\sc M.~M. Rao and Z.~D. Ren}, {\em Applications of {O}rlicz {S}paces},
  vol.~250 of Monographs and Textbooks in Pure and Applied Mathematics, Marcel
  Dekker, Inc., New York, 2002.

\bibitem{ScoZha}
{\sc L.~R. Scott and S.~Zhang}, {\em Finite element interpolation of nonsmooth
  functions satisfying boundary conditions}, Mathematics of Computation, 54
  (1990), pp.~483--493.

\bibitem{ref:Sch}
{\sc R.~E. Showalter}, {\em Monotone {O}perators in {B}anach {S}paces and
  {N}onlinear {P}artial {D}ifferential {}Equations}, vol.~49 of Math. Surveys
  and Monographs, AMS, Providence, R.I., 1997.

\end{thebibliography}
\end{document}